\documentclass[twoside,final]{amsart}

\usepackage[utf8]{inputenc}
\usepackage[T1]{fontenc}
\usepackage[english]{babel}

\usepackage{mathpazo}
\usepackage{mathptmx}
\usepackage{amsmath}
\usepackage{amssymb}
\usepackage{mathrsfs}
\usepackage{amsthm}
\usepackage{amsfonts}
\usepackage{fancyhdr}
\usepackage[all]{xy}
\usepackage{cancel}
\usepackage{amscd,amsxtra}
\usepackage{mathtools,mathrsfs,dsfont,xparse}
\usepackage{graphicx,epstopdf}
\usepackage{multirow}
\usepackage{cases}
\usepackage{enumerate}
\usepackage{multicol,bbm}

\usepackage[notcite, notref]{showkeys}

\usepackage{mathtools}

\mathtoolsset{showonlyrefs}

\usepackage[margin=1 in]{geometry}

\date{}

\usepackage[colorlinks=true, pdfstartview=FitV, linkcolor=blue,
            citecolor=blue, urlcolor=blue]{hyperref}

\allowdisplaybreaks[4]

\numberwithin{equation}{section}

\newtheorem{Thm}{Theorem}[section]

\newtheorem{Def}{Definition}[section]

\newtheorem*{Thm*}{Theorem}

\usepackage{cjhebrew}

\newcommand{\E}{\mathbb{E}}

\newcommand{\al}{\alpha}

\newcommand{\T}{T}

\usepackage[numbers,sort&compress]{natbib}

{\theoremstyle {definition} \newtheorem {defi} {Definition} [section] }
{\theoremstyle {plain}  \newtheorem {thm} [defi] {Theorem}}
{\theoremstyle {plain}  \newtheorem {cor} [defi]{Corollary}}
{\theoremstyle {plain} \newtheorem {prop} [defi]{Proposition}}
{\theoremstyle {plain} \newtheorem {lem}[defi] {Lemma}}
{\theoremstyle {plain} }
{\theoremstyle {definition} \newtheorem {rmq}[defi] {Remark}}
{\theoremstyle {definition} \newtheorem {claim}[defi] {Claim}}

{\theoremstyle {plain}  }
{\theoremstyle {plain}  }
{\theoremstyle {plain}  }
{\theoremstyle {plain} }
{\theoremstyle {plain} }
{\theoremstyle {plain} }
{\theoremstyle {plain} }

\def\E{{\Bbb{E}}}
\def\T{{\Bbb{T}}}
\def\P{{\Bbb{P}}}
\def\R{{\Bbb{R}}}
\def\C{{\Bbb{C}}}
\def\Z{{\Bbb{Z}^3}}
\def\N{{\Bbb{N}}}

\def\i{{\textbf{i}}}

\def\dt{{\partial_t}}

\def\lleq{{\ \lesssim\ \ }}

\usepackage{tikz}
\usetikzlibrary{matrix}

\DeclareMathOperator\supp{supp}
\DeclareMathOperator{\re}{Re}

\DeclareDocumentCommand{\abs}{s m}{
  \operatorname{}
  \IfBooleanTF{#1}{#2}{\left|#2\right|}}

\DeclareDocumentCommand{\norm}{s m}{
  \operatorname{}
  \IfBooleanTF{#1}{#2} {\left\| #2\right\|}}

\DeclareDocumentCommand{\inner}{s m}{
  \operatorname{}
  \IfBooleanTF{#1}{#2}{\left \langle#2\right \rangle}}

\DeclareDocumentCommand{\parenthese}{s m}{
  \operatorname{}
  \IfBooleanTF{#1}{#2}{\left(#2\right)}}

\DeclareDocumentCommand{\square}{s m}{
  \operatorname{}
  \IfBooleanTF{#1} {#2}{\left[#2\right]}}

\DeclareDocumentCommand{\bracket}{s m}{
  \operatorname{}
  \IfBooleanTF{#1}{#2}{\left\{#2\right\}}}

\begin{document}

\author[Sy and Yu]{Mouhamadou Sy$^1$ and Xueying Yu$^2$}

\address{Mouhamadou Sy
\newline 
\indent Department of Mathematics, Imperial College London \indent 
\newline \indent  Huxley Building, London SW7 2AZ, United Kingdom,\indent }
\email{m.sy@imperial.ac.uk}
\thanks{$^1$ The first author's current address is Department of Mathematics, Imperial College London, United Kingdom. This work was written at Department of Mathematics, University of Virginia, Charlottesville, VA.}

\address{Xueying  Yu
\newline \indent Department of Mathematics, University of Washington\indent 
\newline \indent  C138 Padelford Hall Box 354350, Seattle, WA 98195,\indent }
\email{xueyingy@uw.edu}
\thanks{$^2$  The second author's current address is Department of Mathematics, University of Washington, Seattle, WA. This work was written at Department of Mathematics, MIT, Cambridge, MA.}

\title[A.S. GWP for energy supercritical NLS on the unit ball]{Almost sure global well-posedness for the energy supercritical NLS on the unit ball of $\Bbb R^3$}

\begin{abstract}
In this paper, we present two almost sure global well-posedness (GWP)  results for the energy supercritical nonlinear Schr\"odinger equations (NLS) on the unit ball of $\Bbb R^3$ using two different approaches. First, for the NLS with algebraic nonlinearities with the subcritical initial data, we show the almost sure global well-posedness  and the invariance of the underlying measures, and establish controls on the growth of Sobolev norms of the solutions. This global result is based on a deterministic local theory and a probabilistic globalization. Second, for the NLS with generic power nonlinearities with critical and supercritical initial conditions, we prove the almost sure global well-posedness, and the invariance of the measure under the solution flows. This global result is built on a compactness argument and the Skorokhod representation theorem.


\noindent
\textbf{Classification:} 28D05, 60H30, 35Q55, 35BXX, 37K05, 37L50.\\
\end{abstract}

\maketitle

\nocite{*}

\setcounter{tocdepth}{1}
\tableofcontents

\parindent = 10pt     
\parskip = 8pt

\section{Introduction}
We consider  the nonlinear Schr\"odinger initial value problem
\begin{align}\label{NLS}
\begin{cases}
\dt u=\i(\Delta u-|u|^{2q}u) , &  q> 0,\\
u(0,x) = u_0(x) , &
\end{cases} 
\end{align}
posed on the unit ball $\Theta = \{x \in \R^3 \, : \,  \abs{x} < 1\} $, where $u=u(t,x)$ is a complex-valued function in spacetime $\R \times \Theta$. We assume the radial symmetry on the initial datum $u_0$ and the Dirichlet boundary condition:
\begin{align*}
u\big|_{\partial \Theta}=0 .
\end{align*}

The solution of \eqref{NLS} conserves both the mass:
\begin{align}\label{eq Intro mass}
M(u(t)) : = \int_{\Theta} \abs{u(t,x)}^2 \, dx = M(u_0) ,
\end{align}
and the energy:
\begin{align}\label{eq Intro energy}
E(u(t)) :  = \int_{\Theta} \frac{1}{2} \abs{\nabla u(t,x)}^2 + \frac{1}{2q+2} \abs{u(t,x)}^{2q+2} \, dx = E(u_0) .
\end{align}
Conservation of mass and energy gives the control of the $L^2$ and $\dot{H}^1$ norms of the solutions, respectively.

To best frame this problem, let us start by recalling the scaling of NLS in the language of Euclidean spaces setting. First, the critical scaling exponent of NLS on $\R^d$
\begin{align}\label{pNLS}
\dt u=\i(\Delta u-|u|^{2q}u) , \quad q >0
\end{align}
is given by
\begin{align}\label{eq sc}
s_c:=\frac{d}{2} - \frac{1}{q}.
\end{align}
The problem \eqref{pNLS} can be classified as subcritical, critical or supercritical depending on whether the regularity of the initial data is below, equal or above the scaling $s_c$ of \eqref{pNLS}. Also we say that the NLS is with energy supercritical nonlinearities when $s_c >1$ (in dimension $d=3$, $s_c >1$ implies $q >2$).

We study the energy supercritical NLS \eqref{NLS}  ($s_c>1$) in this paper and our goals are 
\begin{enumerate}
\item
to establish the global well-posedness{\footnote{With {\it local/global well-posedness} we refer to local/global in time existence, uniqueness and continuous dependence of the data to solution map.}} from a probabilistic point of view
\item
to show the existence of an invariant measure under the global flow
\item
and to obtain controls on the growth of the solutions.
\end{enumerate}
Notice that in  the energy supercritical range, the authors do not know any  unconditional global regularity theory for the equation \eqref{NLS} on Sobolev spaces, and this issue is a long lasting open question. In this paper we address this question from a probabilistic point of view.

\subsection{History and related works}
The well-posedness problem of NLS \eqref{pNLS} has been extensively studied in recent years under a wide variety of settings.  Let us recall the known results in this subsection.

\subsubsection{Subcritical and critical regimes} 
On $\R^d$, the local well-posedness results are well known in both subcritical ($s> s_c$) and critical ($s= s_c$) regimes in \cite{cazweis,caz}. When the initial data are at the level of conservation laws (of energy and mass), an iterative argument gives the global well-posedness of the $H^1$-subcritical and the $L^2$-subcritical initial value problems. At the level of no conservation quantities ($ s \neq 0,1$), one needs more sophisticated methods to globalize the subcritical problems iteratively, such as the high-low method \cite{bourRef} and I-method \cite{ckst}. Let us point out that both methods heavily rely on the conservation of energy. In order to grab the energy, they smooth out the rough solutions into the energy space with the help of a proper  frequency truncation or a well-designed Fourier multiplier.

The critical global well-posedness problems are much more delicate than the subcritical ones. This is because the criticality nature of this problem brings an extra factor (the profile of the initial data) into the local theory, which effects on the length of local time existence and ruins the iterative process as one did in subcritical setting. As a result, at the level of $L^2$ and $H^1$, even though the conservation laws are still playing fundamental roles, they are not enough to globalize the local solutions. More ingredients were introduced to deal with the critical setting, such as the induction on energy method, Morawetz inequalities and concentration compactness/rigidity argument, see \cite{bourgNLS96, GrllakscrtNLS, CKSTTcrtNLS, kenigmerle} and also references for the compact settings \cite{htt, ionpaus,ptw}. It is worth mentioning that the proofs are not iterative, but based on a contradiction argument by ruling out the existence of the certain types of minimal blow-up objects. Moreover, in the case where there are no conservation quantities ($ s_c \neq 0,1$), when assuming that the critical Sobolev norm of the solution stays uniformly bounded on the entire time of existence (which serves a similar role of the conservation quantities as one has in the mass/energy critical settings), one can conclude the global well-posedness, see for instance \cite{km}. However, due to the criticality nature and missing of conservation laws, it is much more challenging to verify the universal bound on the critical Sobolev norm of the solution, or to obtain the critical global well-posedness without such {\it a priori} assumption. This is not our center in this work, so let us provide a numerical simulation result  \cite{css} and a recent work \cite{dod} on this  topic.

\subsubsection{Supercritical regime and probabilistic methods}

On the contrast, in the supercritical regime ($s < s_c$), the scaling is against the well-posedness. To the best of authors' knowledge, there are no (deterministic) well-posedness results in the supercritical scenario. Even worse, it is known in for example \cite{cctAsy} that certain known supercritical data would lead to ill-posednesss{\footnote{{\it Ill-posedness} means that problems are not well-posed, that is, problems that violate any of the three properties of well-posedness.}}.

For the deterministic setting there is no hope for a well-posedness theory, however this type of problem has been receiving a lot of attention from a probabilistic point of view. In fact, in some cases, ill-posedness can be circumvented by an appropriate probabilistic/stochastic method. With initial data lying in some well-constructed probability space, Bourgain \cite{bourgNLS96} proved the local well-posedness of the two dimensional cubic NLS with data below $L^2$. By using a different randomisation method, \cite{thomann2009random} showed an almost sure local well-posed result for the cubic NLS in the supercritical regularity. A recent work \cite{dny} established an almost-sure local theory for NLS that covers the full subcritical regime in the probabilistic scaling{\footnote{{\it The probabilistic scaling} is defined to be $-\frac{1}{2q}$, which is lower than the regular scaling $\frac{d}{2} -\frac{1}{q}$ defined in \eqref{eq sc}.}}.

\cite{bourgNLS96} also extended to a global solution with the invariant Gibbs measure constructed from a suitable renormalized Hamiltonian. The important point to note here is that in \cite{bourgNLS96}, the invariant measure replaces the role of the conservation laws in the deterministic setting, and such invariance property is very crucial in making an iterative argument possible. In fact, the Gibbs measures technique for dispersive PDEs goes back to Lebowitz-Rose-Speer \cite{LRS}, where they constructed the Gibbs measures for the one-dimensional NLS. More results on the probabilistic well-posedness and invariant Gibbs measures can be found in  \cite{bourg94,zhidwave,zhd,tzvNLS06,bt2007,tzvNLS, btrandom1,btrandom2,asds,colloh,MR3131480,pocovnicu2014,burktzvt_probwav,
bourbulW,btt18,nsnls,yue,dnygibbs} and references therein. It is worth pointing out that the regularity where the Gibbs measure lives is $H^{1-\frac{d}{2}-\varepsilon}$, so in high dimensions, for example $d=3$, it is very rough $H^{-\frac{1}{2}-\varepsilon} $. This negative regularity causes difficulties in constructing or even making sense of the Gibbs measure.

The fluctuation-dissipation approach is another method to construct invariant measures. It is introduced in Kuksin, Kuksin-Shirikyan\cite{kuk_eul_lim, KS04, KS12} where they obtained an invariant measure on the Sobolev space $H^2$  for the two dimensional Euler equation and the cubic NLS with dimensions $d \leq 4$. This method is based on the compact approximations of the equation. By considering a suitable modification  (introducing a dissipation term and a carefully normalized white noise in time) of the equation, one then obtains that the stationary measure converges to a non-trivial  invariant measure under the inviscid limit. See also \cite{kuksin_nondegeul,armen_nondegcgl,sybo,sykg,syNLS7,fsSQG,lat} for related works.

As we explained above, the globalization under the setting where there are no conservation quantities is less favorable. Both results in this paper in fact deal with this type of setup. Let us pause here and present our two main theorems.

\subsection{Main results}
\subsubsection{Statements of the main results}
\begin{thm}[Subcritical almost sure global well-posedness]\label{main thm1}
For any  $q \in \N^+$, any $\beta\in (s_c,\frac{3}{2}]$, any increasing, one-to-one, concave function $\xi:\R_+\to\R_+$, There is a measure $\mu=\mu_{q,\beta,\xi}$ concentrated on $H^{\beta}(\Theta)$ and a set $\Sigma=\Sigma_{q,\beta,\xi}\subset H^{\beta}$ such that
\begin{enumerate}
\item $\mu (\Sigma)=1;$ 
\item For any $u_0\in \Sigma$, there is a unique $u\in C(\R, H^{\beta})$ satisfying \eqref{NLS};
\item The flow $\phi_t$ induced by the existence and uniqueness property is continuous and leaves the set $\Sigma$ invariant, that is, $\phi_t\Sigma=\Sigma$; 
\item The measure $\mu$ is invariant under the flow $\phi_t$;
\item The following identity holds\label{Est}
\begin{align}
\int_{L^2}e^{\rho(\|u\|_{H^{\beta -}})}\left(\|u\|_{H^{\beta-1}}^2+\|u\|_{L^{2q+2}}^{2q+2}\right)\mu(du)=\frac{A_0}{2},\label{Id_intro}
\end{align}
where $A_0<\infty$ is given by the characteristics of the noise used in the approximation;

\item The set $\Sigma$ contains data of arbitrary large sizes, that is, for any $M>0$, we have that $\mu(\{u\ |\ \|u\|_{H^\beta}>M\})>0$;\label{EstCons}
\item For any $u_0\in \Sigma$, we have the slow growth control
\begin{align}
\|\phi^tu_0\|_{H^{\beta-}}\leq C_{\xi}(u_0)\xi(\ln(1+t)). \label{Est:Contr}
\end{align}
\end{enumerate}
\end{thm}

\begin{rmq}\label{Rmk:Intro}
\begin{enumerate}
\item
In the theorem above, $A_0$ is an assigned `size' on the noise, then it is a choice. Therefore we can remark that the point \eqref{EstCons} follows from this fact combined with the point \eqref{Est}.

\item
Let us focus on the control \eqref{Est:Contr}. The fact that $\xi$ is our choice makes us able to consider very slow growing function; for instance any combination of finite number of logarithmic  functions can be considered. We can compare this with the control given by the Gibbs measures approach where $\xi$ is fixed to be the square root function, namely the Gibbs measure control is
\begin{align*}
\|\phi^tu_0\|_{H^{\frac{1}{2}-}}\leq C(u_0)\sqrt{\ln(1+t)}.
\end{align*}

\item
Notice that Theorem \ref{main thm1} also covers the cubic $q =1$ (energy subcritical) and quintic $q =2$  (energy critical)  cases. Since our method admits these powers, we just include them for completeness. However, the bounds \eqref{Est:Contr} are new to these cases. The main concern of this paper is still on the energy supercritical NLS.  
\end{enumerate}
\end{rmq}

\begin{thm}[Critical and supercritical almost sure global well-posedness]\label{main thm2} 
For any $q>2$, any $\beta \in (1,s_c]$, there is a process $\{ u^\omega(t) , \,  (\omega, t ) \in  \Omega \times \R \}$ belonging to the space
\begin{align*}
C(\R,H^{\beta-})\cap L^\infty(\R,L^{2q+2})\cap L^2_{loc}(\R,H^{\beta})\cap L^p_{loc}(\R,L^r) \quad \forall p,\ r\in [1,\infty),
\end{align*}
solving  \eqref{NLS}. The distribution $\mu$ of $u$ is invariant in time. Moreover for all $\omega \in\Omega$, the solution curve $u^\omega(\cdot)$ is unique, hence we obtain a global flow.

Furthermore, we have the following properties
\begin{enumerate}
\item The mass $M(u)$ in \eqref{eq Intro mass} and the energy $E(u)$ in \eqref{eq Intro energy} are conserved by the constructed flow;
\item We have the identity
\begin{align*}
\int_{L^2}\left(\|u\|_{H^{\beta-1}}^2+2e^{2\|u\|_{L^{4q+2}}^{4q+2}}\|u\|_{L^{2q+2}}^{2q+2}+\sum_{p\in\N}\frac{\|u\|_{L^{2p+2}}^{2p+2}}{p!}\right)\mu(du)=\frac{A_0}{2};
\end{align*}
\item The support of $\mu$ contains data of arbitrary large sizes, that is, for any $M>0$, we have that $\mu(\{u\ |\ \|u\|_{H^\beta}>M\})>0$.
\item The solution  continuously depends on the initial data  in any $H^s$, $s <1$ and $L^{p}$, $1 \leq p < 2q+2$ spaces,  in the sense that for two solutions $u$ and $v$ with initial data $u_0 , v_0$ respectively,  we have that 
\begin{align*}
\lim_{\|u_0-v_0\|_{H^1}+\|u_0-v_0\|_{L^{2q+2}}\to 0}\sup_{t\in [-T,T]}(\|u-v\|_{H^s}+\|u-v\|_{L^{p}})=0  .
\end{align*}
\end{enumerate}
\end{thm}

It is traditional to ask about qualitative properties in the context of fluctuation-dissipation measures.  Without giving details of computation we refer to  \cite{armen_nondegcgl} and  Theorem $9.2$ and Corollary $9.3$ in \cite{syNLS7} as a justification of the following statement which is valid for all subcritical, critical and supercritical settings. 
\begin{thm}
The distributions via $\mu$ of the functionals $M(u)$ and $E(u)$ have densities with respect to the Lebesgue measure on $\Bbb R.$
\end{thm}

\subsubsection{Discussion on the setting}
In what follows, let us recall that `subcritical/supercritical' data means the initial condition being smoother/rougher than the scaling  $s_c$ of NLS, while `energy supercritical' NLS/nonlinearities refers to the scaling of the NLS being greater than the energy (when $d=3$, it is equivalent to saying that the nonlinear power $q$ being greater than the energy critical power $2$). 

Let us look more closely at our setting.

\paragraph{$\bullet$ \it Energy supercritical nonlinearities}

We are interested in the well-posedness theory for the energy supercritical models in this current paper. In such energy supercritical range ($s_c >1$), the energy conservation law is supercritical to the scaling, hence weaker to use when compared to the energy critical and energy subcritical scenarios. This is in fact a quite interesting and rich setting. As we mentioned above, there have been studies on the deterministic critical global well-posedness in the Euclidean spaces, under the assumption that  the critical Sobolev norm of the solution stays uniformly bounded, see for example \cite{kv}.  Also \cite{bul} showed that for radially  NLS with an energy supercritical nonlinearity in $\R^3$, with some additional regularity on the initial datum  the global well-posedness remains true  if a solution stays bounded in slightly subcritical Sobolev norms. On the probabilistic side, the first author of this paper was able to show almost sure global well-posedness and the  invariance of the  measure for the energy supercritical NLS on $\T^3$ with $H^s, s \geq 2$ initial data in \cite{syNLS7}. From the blow-up perspective,  recently a groundbreaking work \cite{mrrs} proved the existence of smooth and well localized spherically symmetric initial data such that the corresponding unique strong solution blows up in finite time  in dimensions $d \geq 5$.

However, as we mentioned earlier the question of the global existence for energy supercritical models in fact remains open in many contexts, subcritical, critical and supercritical. To the best knowledge of the authors, there are no known probabilistic global well-posedenss results in the energy supercritical setting with supercritical initial data and few studies with subcritical initial data. Both of our results focus on the energy supercritical contexts and especially the second result deals with this open problem in very rough regularities while the first result deals with all the subcritical regularities left open by \cite{syNLS7}.

The first result mainly concerns the globalization in the subcritical setting with algebraic energy supercritical nonlinearities (the scaling $s_c >1$ and initial data in $H^s, \,  s> s_c$). Our original interest was to study the general nonlinearities, however in order to obtain a satisfactory local theory with suitable control on the growth of the solution, we develop the multilinear estimates, which consequently restrict our attention to the odd integer powers. Such algebraic nonlinearities fall mainly into the energy supercritical regime (except when $q=1$, $s_c = \frac{1}{2}$ energy subcritical and  when $q=2$, $s_c = 1$  energy critical{\footnote{Since the proof sees no differences in these two cases, we include them for completeness.}}).

Theorem \ref{main thm2} gives an answer to the open question that we mentioned above. Roughly speaking, the second result investigates the good behavior of the solution with rough data while \cite{mrrs} studies the bad behavior with smooth data. Here we consider the NLS  with generic nonlinear power $q> 2$  (the  energy supercritical range) with critical and supercritical data (the scaling $s_c >1$  and  initial data  in $H^s,  s \leq   s_c$). The reason why we are not restricted only on algebraic nonlinearities here is that our globalization argument employed in this part does not rely on a local theory at all. On the other hand, due to the absence of the local theory,  the conclusion is less strong compared to the first result. That is,  in contrast to our first result we lose the individual control of the solution, and there is also lack of estimates on the growth of the solutions.

\paragraph{$\bullet$ \it Compact manifold}
Besides energy supercritical models, we are interested in the bounded manifold. In fact, for dispersive equations, the bounded manifold settings are less favorable due to the weaker dispersion. Mathematically we can see this  phenomenon (`loss of regularity')  in the Strichartz estimates on the bounded manifolds.  For example in \cite{bss} the loss of $\frac{1}{p}$ derivatives  was established for compact Riemannian manifold $\Omega$ with boundary
\begin{align}\label{eq loss of reg}
\norm{e^{\i t \Delta} f }_{L^P ([0,T] ; L^q (\Omega))} \leq C \norm{f}_{H^{s+ \frac{1}{p}} (\Omega)}
\end{align}
for  fixed finite $T$, $p > 2$, $q < \infty$ and $\frac{2}{p} + \frac{d}{q} = \frac{d}{2} -s$ (see Remark \ref{rmk Strichartz} for more detailed discussion). To beat the weaker dispersion, we assume the radial symmetry on the initial data in both of our main theorems. Actually under this assumption we can benefit from the properties of  eigenfunctions of the radial Laplace operator, which is extremely useful in the proof of the local theory in the first result. More precisely, it is known that the functions
\begin{align*}
e_n(r) = \frac{\sqrt{2} \sin (\pi n r)}{r} , \quad n \in \N^+,
\end{align*}
(where $r = \abs{x}$)  are the radial eigenfunctions of the Laplace operator $-\Delta$ with Dirichlet boundary conditions, associated to eigenvalues $\lambda_n = (\pi n)^2$.

With the differences and difficulties in our setting well explained, let us give the main ideas of the proofs.

\subsubsection{Outline of the proofs}

\paragraph{$\bullet$ \it Theorem \ref{main thm1}}

As we mentioned above, due to the `loss of regularity' in the Strichartz estimates, it is not easy to obtain the local theory for the full subcritical range (a direct Sobolev embedding gives the  local well-posedness in $H^{\frac{3}{2}+}$, which leaves a gap in the local theory $s > \frac{3}{2} > s_c$). To defeat this loss, we first develop the  multilinear Strichartz estimates based on the decay given by the eigenfunctions of the radial Dirichlet Laplacian. A substantial benefit from these multilinear estimates is that they close the regularity gap that we lose in linear Strichartz (see Remark \ref{rmk Strichartz} for detailed discussion). To best utilize the  multilinear estimates, we consider the algebraic nonlineaities in the first result.

Our multilinear estimates are generalized  from \cite{an}, where the author proved the bilinear Strichartz estimates in the same setting and obtained the local well-posedness for the cubic NLS in the unit ball of $\R^3$ with data in $H^s, \, \frac{1}{2} < s < 2$. It is worth mentioning that the improvement of the loss of regularity in the bilinear estimates of \cite{an} is due to a counting lemma of number theory. In our case, doing estimates with more functions naturally increases the difficulties of the analysis, which forces us to understand this  counting lemma in a better way. Let us also mention that there is an optimization argument employed in this proof to obtain a better upper bound in the multilinear estimates.   Another ingredient in the multilinear estimates is that we take advantage of the nice transfer principle{\footnote{With {\it  transfer principle} we refer to the property that once one proves a certain estimate for the linear solutions, it automatically holds for any general functions with suitable norms adjusted.}} on the  Fourier restriction spaces or Bourgain spaces, where the estimates are lying. Therefore  we are free to reduce the estimates to linear solutions and take advantage of the decay  in the eigenfunctions of the  Dirichlet  Laplacian.

In comparison with \cite{syNLS7}, thanks to the multilinear estimates, we are able to lower the  local well-posedness index all the way to the scaling $s_c$ of the NLS while \cite{syNLS7} considered the very subcritical data in $H^{\frac{d}{2}+}$. Let us also mention that \cite{bourbulNLS} established a probabilistic supercritical local theory for the cubic NLS by taking  the random initial data in  $H^s , s < \frac{1}{2}$ (and then they globalize by an invariance consideration). However according to Theorem 4 of \cite{atGRS} the nonlinear part of the energy is not integrable against the Gaussian measure if $q>2$ when the equation is posed on the unit ball of $\R^3$.  We believe that a similar supercritical local theory holds with general nonlinearity if one follows the argument in \cite{bourbulNLS} using a different type of measure. But we are not going to address this  in the present work. Here we employ a pure deterministic approach that will give a local theory that does not depend on the structure of the probability space that we will construct in the next step.

With this wider range in the local theory, we are now in a position to globalize the solution. We make no attempt to reach out to the Gibbs measure, since it is very hard to jump over the huge regularity gap between the initial data living in $s > s_c >1$ and the Gibbs measure supported on $H^{\frac{1}{2}-} $ (because of the radial assumption).  Also with our robust deterministic local theory, we are more interested in the fluctuation-dissipation method in spirit of the Gibbs measure manner, which was first used in \cite{syNLS7}.

We then consider a suitable fluctuation-dissipation modification of the finite dimensional approximation equation by adding a  dissipation term like $f( -\Delta , \norm{u}_{H^{\beta}}) u$ and a temporal white noise  into the equation,  
\begin{align}\label{eq Intro SNLS}
du = \i (\Delta u - \Pi^N \abs{u}^{2q} u) \, dt - \alpha f(-\Delta , \norm{u}_{H^{\beta}}) u \, dt +  \sqrt{\alpha} \, \text{ white noise} .
\end{align}
Here  $f (-\Delta, x) $ is a function satisfying certain monotonicity property in the second variable and $\alpha$ is the viscosity parameter. In particular, we choose $f (-\Delta, \|u\|_{H^{\beta-}})$ of the form $e^{\rho ( \|u\|_{H^{\beta-}})}[(-\Delta)^{\beta -1}u +\Pi^N|u|^{2q}u]$ (where $\beta$ is any subcritical regularity), which serves best as a dissipation term in  our setting (energy supercritical NLS with subcritical data). Actually, there are infinite ways of manipulating such dissipation term (we can vary both variables of $f$ or even the form of $f$), and we would like to appreciate  this great flexibility, which gives us the possibility to treat variety of settings, for example a different scaling or a different class of initial data. We will see another choice of $f$ in the second result.

Now the following proof falls into two parts, that is, taking the inviscid limit and  infinite dimensional limit. First on the stochastic Galerkin approximation, we construct stationary measures whose bounds are independent on both the viscosity parameter and  the dimension of the approximating equation. Here the well chosen dissipation term in \eqref{eq Intro SNLS} plays a significant role in giving certain strong uniform bounds. These bounds imply large deviations controls of the measures and play a crucial role in the so-called Bourgain invariant measure argument (see for example \cite{bourg94}), which is one key-ingredient in our proof. Accordingly, we are able to take the inviscid limit. Then with the Bourgain's argument, we use the invariance of the measure as a conserved quantity and take the  infinite dimensional limit assisted by the robust local theory and independence of  the dimension of the truncated system.

\paragraph{$\bullet$ \it Theorem \ref{main thm2}}

In the second result, since we have no good (deterministic) supercritical local theory, we would like to construct the global flow using a different method -- compactness argument. Again, we first consider \eqref{eq Intro SNLS} with $f$ delicately constructed  of the form $ (-\Delta)^{\beta -1}u +e^{\|\Pi^N|u|^{2q}u\|_{L^2}^2}\Pi^N|u|^{2q}u+\Pi^Ne^{|u|^2}u$  where $\beta\in (1,s_c]$. This is actually an evidence of the flexibility of choosing dissipation terms that we mentioned above. An important step in our strategy is the choice of a suitable dissipation. For instance, we could `prescribe' an $L^\infty$ norm of the solution in the dissipation in order to guarantee properties such as uniqueness, continuity or regularity for the limiting equation. However we also need an identity such as \eqref{Id_intro} to ensure the global existence for large data. In order to obtain this identity a compactness property is required, which completely prohibits the $L^\infty$ regularity in our setting. Therefore this choice of dissipation is a tricky process where we have to balance the freedom in the choices of $f$ in \eqref{eq Intro SNLS} and the identity constraint.  It turns out that the term $\Pi^Ne^{|u|^2}u$ serves well for both the identity and  regularity problems. In fact, by Taylor expanding it, we see that this new element in the dissipation  gives us the control of all the $L^p$ norms of the solution and the `regularization counterparts' of these $L^p$ norms providing the compactness. Also, this allows to gain  regularity properties in the supercritical regularity setting. 

The construction of the corresponding invariant measures for the finite dimensional approximating equation follows in a similar fashion as in the previous result. However, without a local theory, we can not perform the Bourgain's argument as before to pass to the infinite dimensional limit. Instead, we use a compactness argument with the Skorokhod representation theorem as an excellent substitute.

As we mentioned earlier, certain classes of solutions have been constructed to demonstrate lack of local well-posedness, in particular in the sense that the dependence of solutions upon initial data fails to be uniformly continuous. In fact, the uniqueness property and the continuity requirement in the well-posedness definition are generally unfavorable and very likely to fail. In our second result, we prove  the uniqueness via a nontraditional  way, that is, we utilize the radial Sobolev embedding with Gronwall's inequality,  then obtain the uniqueness by a limiting argument. As for the continuity property, we employ a limiting argument as well, in which the analysis at the origin of the ball (where the singularity may occur) depends on information on the data. We believe this new approach may be of an independent interest.

\subsection{Organization of the paper}
In Section \ref{sec Preliminaries}, we introduce the notations and the functional spaces with their properties needed in this paper. We prove Theorem \ref{main thm1} in Sections \ref{sec LWP}, \ref{Sect:Galerk} and \ref{Sect:Alg:pow}. In Section \ref{sec LWP}, we first present the local well-posedness result via multilinear estimates. In Section \ref{Sect:Galerk}, we work on the fluctuation-dissipation equations and the  inviscid limits. Then Section \ref{Sect:Alg:pow} is devoted to the proof of the almost sure global well-posedness and the invariance of the measures. We finally show Theorem \ref{main thm2} in Section \ref{sec supercritical} using a compactness argument.

\subsection*{Acknowledgement} 
X.Y. was funded in part by the Jarve Seed Fund and an AMS-Simons travel grant. Both authors are very grateful for the anonymous referees for valuable comments.

\section{Preliminaries}\label{sec Preliminaries}
In this section, we introduce  the notation and the functional spaces with their properties needed in this paper. 
\subsection{Notation}
We define on a time interval $I$
\begin{align*}
\norm{f}_{L_t^q L_x^r (I \times \Theta)} : = \square{\int_I \parenthese{\int_{\Theta} \abs{f(t,x)}^r \, dx}^{\frac{q}{r}} dt}^{\frac{1}{q}}.
\end{align*}

For $x\in \R$, we set $\inner{x} = (1 + \abs{x}^2)^{\frac{1}{2}}$. We adopt the usual notation that $A \lesssim  B$ or $B \gtrsim A$ to denote an estimate of the form $A \leq C B$ , for some constant $0 < C < \infty$ depending only on the {\it a priori} fixed constants of the problem. We write $A \sim B$ when both $A \lesssim  B $ and $B \lesssim A$.

For a Banach space $E$, we denote by $C_b(E)$ the space of bounded continuous functions $f:E\to\R$. $\mathcal{B}(E)$ is the Borel $\sigma-$ algebra, and $\mathfrak{p}(E)$ the set of Borel probability measures on $E$. We also denote by $B_R(E)$ the closed ball in $E$ of radius $R$ and centered at $0.$

In Section \ref{Sect:Galerk} we introduce some more notation that we will need in the proof.

\subsection{$H_{rad}^s$ spaces}
We denote $e_n (r)$ to be the eigenfunctions of the radial Laplace operator $\Delta$ with Dirichlet boundary condition $\Theta$, and $e_n$'s have the following explicit form
\begin{align}\label{eq e_n}
e_n(r) = \frac{ \sqrt{2} \sin (n \pi r)}{r}, \qquad n = 1,2,3,\cdots ,
\end{align}
where $r = \abs{x}$. The eigenvalues associated to $e_n$ are 
\begin{align}\label{eq z_n}
\lambda_n = z_n^2 = (\pi n)^2 .
\end{align}
With this notation, we have the following estimates  (Lemma 2.5 in \cite{AT}) on the norms of the eigenfunctions:
\begin{align}\label{eq e_n bdd}
\norm{e_n}_{L_x^p} \lesssim 
\begin{cases}
1 , & \text{ if } 1 \leq p < 3,\\
(\ln n)^{\frac{1}{3}} & \text{ if } p= 3,\\
n^{1-\frac{3}{p}} , & \text{ if } p> 3 
\end{cases} 
\end{align}
Moreover, $(e_n)_{n=1}^{\infty}$ form an orthonormal bases of the Hilbert space of $L^2$ radial functions on $\Theta$. That is, 
\begin{align*}
\int e_n^2 \, dL = 1
\end{align*}
where $dL = \frac{1}{4\pi} r^2 \, d\theta dr$ is the normalized Lebesgue measure on $\Theta$. 
Therefore, we have the expansion formula for a function $u \in L^2 (\Theta)$, 
\begin{align}\label{eq proj}
u=\sum_{n=1}^{\infty} (u , e_n) e_n .
\end{align}
For $\sigma \in \R$, we define the Sobolev space $H^{\sigma} (\Theta)$ on the closed unit ball $\Theta$ as 
\begin{align*}
H_{rad}^{\sigma} (\Theta) : = \bracket{ u = \sum_{n=1}^{\infty} c_n e_n, \, c_n \in \C : \norm{u}_{H^{\sigma} (\Theta)}^2 = \sum_{n=1}^{\infty} (n \pi)^{2\sigma} \abs{c_n}^2 < \infty } .
\end{align*}
We can equip $H_{rad}^{\sigma} (\Theta)$ with the natural complex Hilbert space structure. In particular, if $\sigma =0$, we denote $H_{rad}^{0} (\Theta)$ by $L_{rad}^2 (\Theta)$. For $\gamma \in \R$, we define the map $\sqrt{-\Delta}^{\gamma}$ acting as isometry from $H_{rad}^{\sigma} (\Theta)$ and $H_{rad}^{\sigma - \gamma} (\Theta)$ by
\begin{align*}
\sqrt{-\Delta}^{\gamma} (\sum_{n=1}^{\infty} c_n e_n) = \sum_{n=1}^{\infty} z_n^{\gamma} c_n e_n .
\end{align*}
We denote $S(t) = e^{\i t \Delta}$ the flow of the linear Schr\"odinger equation with Dirichlet boundary conditions on the unit ball $\Theta$, and it can be written into
\begin{align*}
S(t) (\sum_{n=1}^{\infty} c_n e_n) = \sum_{n=1}^{\infty} e^{-\i t z_n^2 } c_n e_n.
\end{align*}

\begin{lem}[Radial Sobolev embedding]\label{lem Radial Sobolev}
For $r = \abs{x}$ and any function $u$
\begin{align*}
\norm{r u}_{L_x^{\infty}} \lesssim \norm{u}_{H_x^1}.
\end{align*}
\end{lem}
\begin{proof}[Proof of Lemma \ref{lem Radial Sobolev}]
By an approximation argument, we can assume $u \in C_0^{\infty}$. Using the fundamental theorem of calculus with the boundary condition on $\Theta$,  we write
\begin{align*}
\abs{u}^2 =-  \int_r^1 2 u(R) u'(R) \, dR \lesssim \frac{1}{r^2} \parenthese{\int_r^1 u (R) R^2 \, dR }^{\frac{1}{2}} \parenthese{\int_r^1 u' (R) R^2 \, dR }^{\frac{1}{2}} \lesssim \frac{1}{r^2} \norm{u}_{H_x^1}^2 .
\end{align*}
Hence the proof of Lemma \ref{lem Radial Sobolev} is complete.
\end{proof}

\subsection{$X_{rad}^{s,b}$ spaces}
Recall the $L^2$ orthonormal basis of eigenfunctions $( e_n )_{n=1}^{\infty}$ of the Dirichlet Laplacian $-\Delta$ with eigenvalues $z_n^2$ on $\Theta$ defined in the previous subsection.  Then we define the $X^{s,b}$ spaces of functions on $\R \times \Theta$ which are radial with respect to the second argument.
\begin{defi}[$X_{rad}^{s,b}$ spaces]\label{defn Xsb}
For $s \geq 0$, $b \in \R$ and $u(t) = \sum_{n=1}^{\infty} c_n (t) e_n$, we define
\begin{align}\label{eq Xsb}
\norm{u}_{X_{rad}^{s,b} (\R \times \Theta)}^2 = \sum_{n=1}^{\infty} \norm{\inner{\tau + z_n^2}^b \inner{z_n}^{s} \widehat{c_n} (\tau) }_{L^2(\R_{\tau} ) }^2  ,
\end{align}
and 
\begin{align*}
X_{rad}^{s,b} (\R \times \Theta) = \{ u \in \mathcal{S}' (\R , L^2(\Theta)) : \norm{u}_{X_{rad}^{s,b} (\R \times \Theta)} < \infty \} .
\end{align*}
Moreover, for $u \in X_{rad}^{0, \infty} (\Theta) =  \cap_{b \in \R} X_{rad}^{0,b} (\Theta)$ we define, for $s \leq 0$ and $b \in \R$, the norm $\norm{u}_{X_{rad}^{s,b} (\R \times \Theta)}$ by \eqref{eq Xsb}.
\end{defi}
Equivalently, we can write the norm \eqref{eq Xsb} in the definition above into
\begin{align*}
\norm{u}_{X_{rad}^{s,b} (\R \times \Theta)} = \norm{S(-t) u}_{H_t^b H_x^s (\R \times \Theta)} .
\end{align*}

For $T > 0$, we define the restriction spaces $X_T^{s,b} (\Theta)$ equipped with the natural norm
\begin{align*}
\norm{u}_{X_T^{s,b} ( \Theta)} = \inf \{ \norm{\tilde{u}}_{X_{rad}^{s,b} (\R \times \Theta)} : \tilde{u}\big|_{(-T,T) \times \Theta} =u\} .
\end{align*}

\begin{lem}[Basic properties of $X_{rad}^{s,b}$ spaces]\label{lem X property1}
\begin{enumerate}
\item
We have the trivial nesting 
\begin{align*}
X_{rad}^{s,b} \subset  X_{rad}^{s' , b' } 
\end{align*}
whenever $s' \leq s$ and $b' \leq b$, and
\begin{align*}
X_{T}^{s,b} \subset  X_{T'}^{s,b} 
\end{align*}
whenever $T' \leq T$ .
\item
The $X_{rad}^{s,b}$ spaces interpolate nicely in the $s, b$ indices.
\item
For $b > \frac{1}{2}$, we have the following embedding
\begin{align*}
\norm{u}_{L_t^{\infty} H_x^{s} (\R \times \Theta) } \leq C \norm{u}_{X_{rad}^{s,b} (\R \times \Theta)}.
\end{align*}
\end{enumerate}
\end{lem}

\begin{lem}\label{lem X property2}
Let $b,s >0$ and $u_0 \in H_{rad}^s (\Theta)$. Then there exists $c >0$ such that for $0 < T \leq 1$,
\begin{align*}
\norm{S(t) u_0}_{X_{T}^{s,b} } \leq c \norm{u_0}_{H^s}.
\end{align*} 
\end{lem}

The proofs of Lemma \ref{lem X property1} and Lemma \ref{lem X property2}  can be found in \cite{an}.

\section{Local well-posedness with the subcritical regularities}\label{sec LWP}
Consider the NLS with odd-integer power nonlinearities
\begin{align}\label{eq PNLS}
\begin{cases}
\dt u=\i(\Delta u-  |u|^{2q}u) & \text{ on } \R \times \Theta\\
u(0,x) =u_0 & \text{ on } \Theta\\
u\big|_{\partial \Theta} (t,x) =0 ,
\end{cases}
\end{align}
where $q \in \N^+$.

The main result in this section is the following local theory.
\begin{thm}[Deterministic local well-posedness]\label{thm LWP}
Let $q \in \N^+$. For $u_0 \in H_{rad}^s (\Theta)$, $s > s_c = \frac{3}{2} - \frac{1}{q}$, \eqref{eq PNLS} is locally well-posed. More precisely, let us fix $s > s_c $, then there exist $b> \frac{1}{2}$, $\beta > 0$, $C >0$, $\widetilde{C} >0$ and $c \in (0,1]$ such that for every $R> 0$ if we set $T_R = c (R^{2q})^{-\beta}$, then for every $u_0 \in H_{rad}^s (\Theta)$ satisfying $\norm{u_0}_{H_{rad}^s (\Theta)} \leq R$, there exists a unique solution of \eqref{eq PNLS} in $X_{rad}^{s,b} ([-T, T] \times \Theta)$ with initial condition $u(0) = u_0$. Moreover,
\begin{align}
\norm{u}_{L_t^{\infty} H_x^{s} ([-T, T] \times \Theta) } \leq C \norm{u}_{X_{rad}^{s,b}([-T,T] \times \Theta)} \leq  \widetilde{C} \norm{u_0}_{H_{rad}^s (\Theta)} . \label{Local:increase}
\end{align}
\end{thm}
Note that the the cubic case $q =1$ was studied in Theorem 1.1 in \cite{an}, and we include this case for completeness. From now on, for simplicity of notation, we write $H^s$ and $X^{s,b}$ for the spaces $H_{rad}^s$ and $X_{rad}^{s,b}$ defined in Section \ref{sec Preliminaries}.

\begin{rmq}
Notice that in Theorem \ref{thm LWP}, the local theory covers the full subcritical region. This is much stronger than the version that one gets with the embedding $H^{\frac{3}{2}+} (\R^3) \hookrightarrow L^{\infty} (\R^3)$ alone. That is, if we run the fixed point argument naively using the embedding itself (without any Strichartz estimates), then the solution map 
\begin{align*}
F (u)(t) : = e^{\i t\Delta}  u_0  - \i \int_0^t e^{\i (t-s) \Delta}  (\abs{u}^{2q} u) \, ds 
\end{align*}
will be bounded by
\begin{align*}
\norm{F u}_{H^{\gamma}}  \leq \norm{ u_0}_{H^{\gamma}} + \int_0^{T} \norm{\abs{u}^{2q} u}_{H^{\gamma}} \, ds   & \lesssim \norm{ u_0}_{H^{\gamma}} + T \norm{u}_{L_{t,x}^{\infty}([0,T] \times \Theta)}^{2q} \norm{u}_{L_t^{\infty} H_x^{\gamma} ([0,T] \times \Theta)}\\
& \lesssim \norm{ u_0}_{H^{\gamma}} + T \norm{u}_{L_t^{\infty} H_x^{\gamma} ([0,T] \times \Theta)}^{2q+1},
\end{align*}
for any $\gamma > \frac{3}{2}$. Accordingly, the local theory would be obtained for $\gamma > \frac{3}{2} > s_c$ instead of the full subcritical range $s > s_c$. Hence the multilinear Strichartz estimates play a fundamental role in improving the regularity in the local theory.
\end{rmq}

To prove the local well-posedness, we first present the following multilinear estimates, which are crucial to obtain the nonlinear estimates in Subsection \ref{subsec nonlinear}. 

\subsection{Multilinear estimates}
\begin{prop}[m-linear estimates]\label{prop m-linear}
For $j \in \{1,2,\cdots , m \}$ ($m \geq 3$), we assume that  $N_1 \geq N_2 \geq \cdots \geq N_m > 0$, and $u_j \in L_{rad}^2 (\Theta)$ satisfying 
\begin{align*}
\mathbb{1}_{\sqrt{-\Delta} \in [N_j , 2 N_j]} u_j = u_j , 
\end{align*}
we have the following m-linear estimates. 
\begin{enumerate}
\item
The m-linear estimate without derivatives.\\
For any $\varepsilon >0 $
\begin{align}\label{eq m-linear1}
\norm{\prod_{j=1}^m S(t) u_j}_{L_{t,x}^2 ((0,1) \times \Theta)} \lesssim N_2^{1- \frac{1}{2m-2}+\varepsilon} (N_3 \cdots  N_m)^{\frac{3}{2}-\frac{1}{2m-2}} \prod_{j=1}^m \norm{u_j}_{L_x^2 (\Theta)} .
\end{align}

\item
The m-linear estimate with derivatives.\\
Moreover, if $u_j \in H_0^1 (\Theta)$ and  for any $\varepsilon >0 $
\begin{align}\label{eq m-linear2}
\norm{ \nabla S(t) u_1 \prod_{j=2}^m S(t) u_j}_{L_{t,x}^2 ((0,1) \times \Theta)} \lesssim N_1 N_2^{1- \frac{1}{2m-2}+\varepsilon} (N_3 \cdots  N_m)^{\frac{3}{2}-\frac{1}{2m-2}} \prod_{j=1}^m \norm{u_j}_{L_x^2 (\Theta)}.
\end{align}
\end{enumerate}
\end{prop}

\begin{rmq}
When $m=2$, the bilinear estimates in \cite{an} read
\begin{align}\label{eq bilinear}
\norm{S(t) u_1 S(t) u_2}_{L_{t,x}^2 ((0,1) \times \Theta)} \lesssim N_2^{\frac{1}{2} +\varepsilon}  \norm{u_1}_{L_x^2 (\Theta)} \norm{u_2}_{L_x^2 (\Theta)} .
\end{align}
Moreover, if $u_1, u_2 \in H_0^1 (\Theta)$
\begin{align*}
\norm{ \nabla S(t) u_1 S(t) u_2}_{L_{t,x}^2 ((0,1) \times \Theta)} \lesssim N_1 \min \{ N_1,  N_2 \}^{\frac{1}{2} + \varepsilon} \norm{u_1}_{L_x^2 (\Theta)} \norm{u_2}_{L_x^2 (\Theta)} .
\end{align*}
\end{rmq}

\begin{rmq}\label{rmk Strichartz}
Notice that the m-linear estimates make up for the loss in derivatives in Strichartz estimates (although there is still a logarithmic loss in the multilinear estimates). In fact, such loss of  regularity appears  in the linear Strichartz estimates. More precisely, if we consider the two extreme end cases of the Strichartz estimates for radial data, we have
\begin{align*}
\norm{S(t) e_n}_{L_t^2 L_x^6 (I \times \Theta)} & = c \norm{\frac{\sin (n \pi r)}{r	}}_{L_x^6 (\Theta)} = c  n^{\frac{1}{2}} = c \norm{e_n}_{H^{\frac{1}{2}} (\Theta)} ,\\
\norm{S(t) e_n}_{L_t^{\infty} L_x^2 (I \times \Theta)} & = c  \norm{e_n}_{L_x^2 (\Theta)} . 
\end{align*}
There is a loss of $\frac{1}{2}$ derivative in the end point linear Strichartz estimate, although there is no loss in the trivial $L^2$ case. Then the loss of regularities of the cases in between can be described by interpolating the two cases above (for example, the $L_t^4 L_x^3$ pair has a loss of $\frac{1}{4}$ regularities). We can also observe the same loss of derivatives from \eqref{eq loss of reg} on general compact Riemannian manifolds with boundary. 

To have a better understanding of the regularity improvement in the m-linear estimates, we take $m=2$ (for simplicity) as an example, which is the bilinear estimate in \cite{an}. For $m \geq 3$,  we should  have very similar behaviors. When $m=2$, a naive estimate by H\"older inequality and Sobolev embedding is given by
\begin{align*}
\norm{ S(t) u_1 S(t) u_2}_{L_{t,x}^2 ((0,1) \times \Theta)} & \lesssim  \norm{ S(t) u_1}_{L_{t,x}^4 (\Theta)} \norm{ S(t)  u_2}_{L_{t,x}^4 (\Theta)} \lesssim \norm{ \abs{\nabla}^{\frac{1}{4}} S(t) u_1}_{L_{t}^4 L_x^{3} (\Theta)} \norm{ \abs{\nabla}^{\frac{1}{4}} S(t) u_2}_{L_{t}^4 L_x^{3} (\Theta)} \\
& \lesssim 
\begin{cases}
( N_1 N_2 )^{\frac{1}{2}} \norm{u_1}_{L_x^2 (\Theta)} \norm{u_2}_{L_x^2 (\Theta)} & \text{using the loss of regularities for $L_t^4 L_x^3$ case}\\
( N_1 N_2 )^{\frac{1}{4}} \norm{u_1}_{L_x^2 (\Theta)} \norm{u_2}_{L_x^2 (\Theta)} & \text{if one had the Euclidean Strichartz} .
\end{cases}
\end{align*}
On the other hand, the bilinear estimate without derivatives in \eqref{eq bilinear} reads
\begin{align*}
\norm{ S(t) u_1 S(t) u_2}_{L_{t,x}^2 ((0,1) \times \Theta)} \lesssim N_2^{\frac{1}{2}+ \varepsilon}  \norm{u_1}_{L_x^2 (\Theta)} \norm{u_2}_{L_x^2 (\Theta)} .
\end{align*}
This is not definitely as good as the bilinear estimate in $\R^d$, but it is much better than the one that we estimated naively using the loss of regularities for $L_t^4 L_x^3$ case ($( N_1 N_2 )^{\frac{1}{2}} $). Also it is almost equally as good as the case in which we pretend that we had the Euclidean Strichartz. This is the reason why we can prove the local theory upto its criticality. 
\end{rmq}

\begin{lem}[Transfer principle]\label{lem m-linear}
For any $b > \frac{1}{2}$ and for $j \in \{1,2,\cdots , m \}$ ($m \geq 3$), we assume $N_1 \geq N_2 \geq \cdots \geq N_m >0$ and $f_j \in X^{0,b} (\R \times \Theta)$ satisfying 
\begin{align*}
\mathbb{1}_{\sqrt{-\Delta} \in [N_j , 2 N_j]} f_j = f_j ,
\end{align*}
one has the following m-linear estimates. 
\begin{enumerate}
\item
The m-linear estimate without derivatives.

For any $\varepsilon >0 $
\begin{align}\label{eq m-linear1'}
\norm{\prod_{j=1}^m f_j}_{L_{t,x}^2 ((0,1) \times \Theta)} \lesssim N_2^{1- \frac{1}{2m-2}+\varepsilon} (N_3 \cdots  N_m)^{\frac{3}{2}-\frac{1}{2m-2}}  \prod_{j=1}^m \norm{f_j}_{X^{0,b} (\Theta)} .
\end{align}

\item
The m-linear estimate with derivatives.

Moreover, if $f_j \in H_0^1 (\Theta)$ and  for any $\varepsilon >0 $
\begin{align}\label{eq m-linear2'}
\norm{ \nabla f_1 \prod_{j=2}^m f_j}_{L_{t,x}^2 ((0,1) \times \Theta)} \lesssim N_1 N_2^{1- \frac{1}{2m-2}+\varepsilon} (N_3 \cdots  N_m)^{\frac{3}{2}-\frac{1}{2m-2}}  \prod_{j=1}^m \norm{f_j}_{X^{0,b} (\Theta)}.
\end{align}
\end{enumerate}
\end{lem}

\begin{proof}[Proof of Lemma \ref{lem m-linear}]
This proof is generalized from Lemma 2.3 in \cite{bgtBil}. 

We first suppose that $f_j (t)$'s are supported in the time interval $(0,1)$ and write
\begin{align*}
f_j (t) =S(-t) S(t) f_j(t)  =:  S(-t) F_j(t)
\end{align*}
Then
\begin{align*}
f_j (t) = \frac{1}{2\pi} \int_{-\infty}^{\infty} e^{\i t \tau} S(-t) \widehat{F_j}(\tau) \, d \tau ,
\end{align*}
and hence
\begin{align*}
\prod_{j=1}^m f_j (t) = \frac{1}{(2\pi)^m} \int_{\R^m} e^{\i t \sum_{j=1}^m \tau_j} \prod_{j=1}^m S(-t) \widehat{F_j}(\tau_j) \, d \tau_1 \cdots d \tau_m.
\end{align*}
Ignoring the oscillating factor $e^{\i t \sum_{j=1}^m \tau_j} $, using Proposition \ref{prop m-linear} and the Cauchy-Schwarz inequality in $(\tau_1 , \cdots , \tau_m)$ (in this place we use that $b > \frac{1}{2}$ to get the needed integrability) yields
\begin{align*}
\norm{\prod_{j=1}^m f_j}_{L_{t,x}^2 ((0,1) \times \Theta)} & \lesssim N_2^{1- \frac{1}{2m-2}+\varepsilon} (N_3 \cdots  N_m)^{\frac{3}{2}-\frac{1}{2m-2}} \int_{\R^m} \prod_{j=1}^m \norm{\widehat{F_j} (\tau_j)}_{L_x^2 (\Theta)} \, d \tau_1 \cdots d \tau_m\\
& \leq N_2^{1- \frac{1}{2m-2}+\varepsilon} (N_3 \cdots  N_m)^{\frac{3}{2}-\frac{1}{2m-2}} \prod_{j=1}^m \norm{ \inner{\tau_j}^b \widehat{F_j} (\tau_j)}_{L_{\tau_j , x}^2 (\R \times \Theta)}\\
& = N_2^{1- \frac{1}{2m-2}+\varepsilon} (N_3 \cdots  N_m)^{\frac{3}{2}-\frac{1}{2m-2}} \prod_{j=1}^m \norm{f_j}_{X^{0,b} (\R \times \Theta)} .
\end{align*}
Finally, by decomposing $f_j(t) = \sum_{n =1}^{\infty} \Psi (t - \frac{n}{2}) f_j(t)$ with a suitable $\Psi \in C_0^{\infty} (\R)$ supported in $(0,1)$, the general case for $f_j (t)$'s follow from the considered particular case of $f_j (t)$'s supported in the time interval $(0,1)$.

The proof of \eqref{eq m-linear2'} is very similar to the proof of \eqref{eq m-linear1'}. In fact, since the functions $u_j$'s are all frequency localized, we can treat the derivative on $u_1$ on the left hand side of \eqref{eq m-linear2'} as its frequency $N_1$ (or its $N_1$-th eigenvalue). Other parts of the proof follows closely from the proof of \eqref{eq m-linear1'}.

Hence we finish the proof of Lemma \ref{lem m-linear}.
\end{proof}

Before proving Proposition \ref{prop m-linear}, let us first recall the following lemma. 
\begin{lem}[Lemma 3.2 in \cite{bgtBil}]\label{lem counting}
Let $M,N \in \N$, then for any $\varepsilon > 0$, there exists $C>0$ such that
\begin{align*}
\# \{ (k_1 , k_2) \in \N^2 : N \leq k_1 \leq 2N, k_1^2 + k_2^2 = M \} \leq C N^{\varepsilon} .
\end{align*}
\end{lem}

Now we are ready to proof the  m-linear estimates.
\begin{proof}[Proof of Proposition \ref{prop m-linear}]
For $j \in \{1,2,\cdots , m \}$, using the expansion formula \eqref{eq proj} we write
\begin{align*}
u_j = \sum_{n_j \sim N_j} (u_j , e_{n_j})   e_{n_j}(r) ,
\end{align*}
Then 
\begin{align*}
S(t) u_j = \sum_{n_j \sim N_j} e^{-\i t n_j^2 \pi^2} (u_j , e_{n_j}) e_{n_j}(r) .
\end{align*}
Therefore, the m-linear objects that one needs to estimate are the $L_{t,x}^2$ norms of 
\begin{align*}
E_0(N_1, N_2, \cdots  , N_m) & = \sum_{n_1 \sim N_1} \sum_{n_2 \sim N_2}  \cdots \sum_{n_m \sim  N_m} e^{-\i t \sum_{j=1}^m n_j^2 \pi^2} \prod_{j=1}^m (u_j , e_{n_j}) \prod_{j=1}^m e_{n_j} (r) ,\\
E_1(N_1, N_2, \cdots  , N_m) & = \sum_{n_1 \sim N_1} \sum_{n_2 \sim  N_2}  \cdots \sum_{n_m \sim  N_m} e^{-\i t \sum_{j=1}^m n_j^2 \pi^2} \prod_{j=1}^m (u_j , e_{n_j}) (\nabla e_{n_1}) \prod_{j=2}^m e_{n_j} (r) .
\end{align*}
Let us focus on \eqref{eq m-linear1} first. Using Parseval's theorem in time
\begin{align*}
(\text{LHS  of } \eqref{eq m-linear1})^2 & = \norm{E_0(N_1, N_2, \cdots  , N_m) }_{L^2 ((0, \frac{2}{\pi}) \times \Theta)}^2 \\
& = \sum_{\tau \in \N} \norm{\sum_{(N_1, N_2, \cdots  , N_m) \in \Lambda_{N_1, N_2 , \cdots , N_m} (\tau)} \prod_{j=1}^m (u_j , e_{n_j}) \prod_{j=1}^m e_{n_j} (r)}_{L^2 (\Theta)}^2 \\
& \lesssim \sum_{\tau \in \N} \# \Lambda_{N_1, N_2 , \cdots , N_m} (\tau) \sum_{(N_1, N_2, \cdots  , N_m) \in \Lambda_{N_1, N_2 , \cdots , N_m} (\tau)} \abs{\prod_{j=1}^m (u_j , e_{n_j}) }^2 \norm{\prod_{j=1}^m e_{n_j} (r)}_{L^2 (\Theta)}^2  ,
\end{align*}
where
\begin{align*}
\Lambda_{N_1, N_2 , \cdots , N_m} (\tau) : = \{ (n_1, n_2 , \cdots , n_m) \in \N^m : n_j \sim  N_j , \sum_{j=1}^m n_j^2 = \tau\} .
\end{align*}
Here we periodized the time into $(0, \frac{2}{\pi})$, since in the definition of $\Lambda_{N_1, N_2 , \cdots , N_m} (\tau)$ we notice that $\sum_{j=1}^m n_j^2 = \tau$, which implies that the time frequencies are integers, hence allows us to periodize the time. 

We claim that
\begin{claim}\label{claim m-linear1}
\begin{enumerate}
\item
$\# \Lambda_{N_1, N_2 , \cdots , N_m} (\tau) = \mathcal{O}(N_2^{\varepsilon} N_3 N_4 \cdots N_m)$ ;
\item
$\norm{\prod_{j=1}^m e_{n_j} }_{L^2(\Theta)}^2 \lesssim (N_2 N_3 \cdots  N_m)^{2-\frac{1}{m-1}}  $ .
\end{enumerate}
\end{claim}
Assuming Claim \ref{claim m-linear1}, one has
\begin{align*}
(\text{LHS  of } \eqref{eq m-linear1})^2 \lesssim N_2^{2-\frac{1}{m-1}+ \varepsilon} (N_3 N_4 \cdots  N_m)^{3-\frac{1}{m-1}}  \prod_{j=1}^m \norm{u_j}_{L^2 (\Theta)}^2 .
\end{align*}
Therefore, \eqref{eq m-linear1} follows.

Now we are left to prove Claim \ref{claim m-linear1}. 
\begin{proof}[Proof of Claim \ref{claim m-linear1}]
(1) By Lemma \ref{lem counting}, we have
\begin{align*}
\# \Lambda_{N_1, N_2 , \cdots , N_m} (\tau) & = \# \{ (n_1, n_2 , \cdots , n_m) \in \N^m : n_j \sim  N_j , \sum_{j=1}^m n_j^2 = \tau\} \\
& = \# \{ (n_1, n_2 , \cdots , n_m) \in \N^m : n_j \sim  N_j ,  n_1^2 + n_2^2 = \tau - \sum_{j=3}^m n_j^2\} \\
& \sim \mathcal{O} (N_2^{\varepsilon} N_3 N_4 \cdots N_m).
\end{align*}
Note that the extra factor $N_3 N_4 \cdots N_m$ is because we treat $\tau - \sum_{j=3}^m n_j^2$ as $M$ in Lemma \ref{lem counting} and there are $N_3 N_4 \cdots N_m$ possible $M$'s in our setting.

(2) To compute the $\norm{\prod_{j=1}^m e_{n_j} (r)}_{L^2 (\Theta)}^2$ term, we write it out explicitly using \eqref{eq e_n},  and by a change of variables $r = \frac{s}{n_m \pi}$ we have
\begin{align*}
\norm{\prod_{j=1}^m e_{n_j} (r)}_{L^2 (\Theta)}^2 & = c \int_0^1 \prod_{j=1}^m \frac{\sin^2 (n_j \pi r)}{r^2} r^2 \, dr = c n_m^{2m-3} \int_0^{n_m \pi} \frac{1}{s^{2m-2}} \prod_{j=1}^m \sin^2 (\frac{n_j}{n_m} s) \, ds\\
& \leq c n_m^{2m-3} \int_0^{\infty} \frac{1}{s^{2m-2}} \prod_{j=2}^m \sin^2 (\frac{n_j}{n_m} s) \, ds.
\end{align*}
Note that we dropped the term $\sin^2 (\frac{n_1}{n_m}s)$ in the integrand above, since $\abs{\sin^2 (\frac{n_1}{n_m}s)} \leq 1$.
Here we denote
\begin{align}\label{eq f}
f(s) : = \frac{1}{s^{2m-2}} \prod_{j=2}^m \sin^2 (\frac{n_j}{n_m} s) .
\end{align}
For a fixed (small) positive number $a$, the contribution of $s>a$ is obviously bounded by 
\begin{align}\label{eq op1}
\int_{a}^{\infty} f(s) \, ds \lesssim \int_{a}^{\infty} \frac{1}{s^{2m-2}} \, ds \sim a^{-(2m-3)}.
\end{align}
For $s \leq a$, since $\sin^2 x \leq x^2$ holds for all $x$, we write
\begin{align}\label{eq op2}
0 \leq f(s) \leq  \prod_{j=2}^{m-1} \parenthese{\frac{n_j}{n_m}}^2,
\end{align}
which leads to
\begin{align*}
\int_0^{a} f(s) \, ds \leq  \prod_{j=2}^{m-1} \parenthese{\frac{n_j}{n_m}}^2 a .
\end{align*}
Hence combining the two cases \eqref{eq op1} and \eqref{eq op2}, we have that for all $a$
\begin{align*}
\int_{0}^{\infty} f(s) \, ds \lesssim a^{-(2m-3)} + \prod_{j=2}^{m-1} \parenthese{\frac{n_j}{n_m}}^2 a .
\end{align*}
Then an optimization argument in $a$ gives
\begin{align*}
\int_{0}^{\infty} f(s) \, ds \lesssim  \prod_{j=2}^{m-1} \parenthese{\frac{n_j}{n_m}}^{2-\frac{1}{m-1}},
\end{align*}
where we choose $a \sim  \prod_{j=2}^{m-1} \parenthese{\frac{n_j}{n_m}}^{-\frac{1}{m-1}}$.
Then one has
\begin{align*}
\norm{\prod_{j=1}^m e_{n_j} (r)}_{L^2 (\Theta)}^2 \lesssim n_m^{2m-3} \prod_{j=2}^{m-1} \parenthese{\frac{n_j}{n_m}}^{2-\frac{1}{m-1}} =  (n_2 n_3 \cdots n_{m-1}  n_m)^{2-\frac{1}{m-1}}  .
\end{align*}
We finish the proof of Claim \ref{claim m-linear1}.

\end{proof}

We estimate \eqref{eq m-linear2} similarly. First write 
\begin{align*}
(\text{LHS  of } \eqref{eq m-linear2})^2 & = \norm{E_1 (N_1, N_2, \cdots  , N_m) }_{L^2 ((0, \frac{2}{\pi}) \times \Theta)}^2 \\
& \lesssim \sum_{\tau \in \N} \# \Lambda_{N_1, N_2 , \cdots , N_m} (\tau) \sum_{\Lambda_{N_1, N_2 , \cdots , N_m} (\tau)} \abs{\prod_{j=1}^m  (u_j , e_{n_j}) }^2 \norm{ \nabla e_{n_1} \prod_{j=2}^m e_{n_j} (r)}_{L^2 (\Theta)}^2 .
\end{align*}
Here we claim
\begin{align}\label{eq Claim m-linear2}
\norm{\nabla e_{n_1} \prod_{j=2}^m e_{n_j} }_{L^2(\Theta)}^2  \lesssim N_1^2 (N_2 N_3 \cdots N_{m-1} N_m)^{2-\frac{1}{m-1}}  .
\end{align}
With the claim above and Claim \ref{claim m-linear1}, we obtain
\begin{align*}
(\text{LHS  of } \eqref{eq m-linear2})^2 \lesssim N_1^2 N_2^{2- \frac{1}{m-1}+\varepsilon} (N_3 \cdots N_{m-1} N_m)^{3-\frac{1}{m-1}} \prod_{j=1}^m \norm{u_j}_{L^2 (\Theta)}^2 .
\end{align*}
This yields \eqref{eq m-linear2}.

Now let us work on the claim \eqref{eq Claim m-linear2}. Replacing $e_{n_j}$'s by their exact expression, we have
\begin{align*}
\norm{\nabla e_{n_1} \prod_{j=2}^m e_{n_j} }_{L^2(\Theta)}^2 & = c \int_0^1 \parenthese{\frac{\partial}{\partial r} \frac{\sin (n_1 \pi r)}{r}}^2 \prod_{j=2}^m \frac{\sin^2 (n_j \pi r)}{r^2} r^2 \, dr \\
& = c \int_0^1 \frac{n_1^2 \pi^2}{r^2} \parenthese{\cos(n_1 \pi r) - \frac{\sin (n_1 \pi r)}{n_1 \pi r}}^2 \prod_{j=2}^m \frac{\sin^2 (n_j \pi r)}{r^2} r^2 \, dr .
\end{align*}
Similarly, we perform a change of variables. Recall the assumption $ N_1 \geq N_2  \geq \cdots \geq N_m$, and we take $r = \frac{s}{n_m \pi}$.  
Then
\begin{align*}
\norm{\nabla e_{n_1} \prod_{j=2}^m e_{n_j} }_{L^2(\Theta)}^2  = c n_1^{2} n_m^{2m-3} \int_0^{n_m \pi} \parenthese{\cos (\frac{n_1}{n_m}s) - \frac{\sin (\frac{n_1}{n_m}s)}{\frac{n_1}{n_m}s}}^2 \frac{1}{s^{2m-2}} \prod_{j=2}^m \sin^2 (\frac{n_j}{n_m} s) \, ds .
\end{align*}
Define
\begin{align*}
g(s) : = \parenthese{\cos (\frac{n_1}{n_m}s) - \frac{\sin (\frac{n_1}{n_m}s)}{\frac{n_1}{n_m}s}}^2 \frac{1}{s^{2m-2}} \prod_{j=2}^m \sin^2 (\frac{n_j}{n_1} s) ,
\end{align*}
and by similar analysis as we did in \eqref{eq f}, we have
\begin{align*}
\int_0^{\infty} g(s) \, ds \lesssim \prod_{j=2}^{m-1} \parenthese{\frac{n_j}{n_m}}^{2-\frac{1}{m-1}}.
\end{align*}
Therefore, we proved \eqref{eq Claim m-linear2}
\begin{align*}
\norm{\nabla e_{n_1} \prod_{j=2}^m e_{n_j} }_{L^2(\Theta)}^2   \lesssim n_1^2 (n_2 n_3 \cdots n_{m-1} n_m)^{2-\frac{1}{m-1}} .
\end{align*}

The proof of  Proposition \ref{prop m-linear} is complete now.
\end{proof}

\subsection{Nonlinear estimates}\label{subsec nonlinear}
In the local theory, we need a nonlinear estimate of the following form
\begin{align*}
\int_{\R} \int_{\Theta} u_0 u_1 u_2 \cdots u_{2m+1} \, dx dt \leq c \norm{u_0}_{X^{-s,b'} (\R \times \Theta)} \prod_{j=1}^{2q+1} \norm{u_j}_{X^{s,b'} (\R \times \Theta)} ,
\end{align*}
where $u_j \in \{ u, \bar{u} \}$ for $j = 1 , 2, \cdots , m$. 

To this end, we first study the nonlinear behavior of all frequency localized $u_j$'s based on the multilinear estimates that we obtained, then sum over all frequencies. 

For $j \in \{ 0, 1,2, \cdots , 2q+1\}$, let $N_j = 2^k, k \in \N$. We denote 
\begin{align*}
u_{j, N_j} = \mathbb{1}_{\sqrt{-\Delta} \in [N_j , 2 N_j]} u_j .
\end{align*}
By the definition of $X^{s,b} (\R \times \Theta)$ spaces, we have
\begin{align*}
\norm{u_j}_{X^{s,b} (\R \times \Theta)}^2 \sim \sum_{N_j = 2^k, k\in \N}  \norm{u_{j,N_j}}_{X^{s,b} (\R \times \Theta)}^2 \sim \sum_{N_j = 2^k, k\in \N}  N_j^{2s}\norm{u_{j,N_j}}_{X^{0,b} (\R \times \Theta)}^2  .
\end{align*}
We denote by $\underline{N} = (N_0, N_1 , \cdots , N_{2q+1})$ the $(2q+2)$-tuple of $2^k$ numbers, $k \in \N$, and
\begin{align*}
I(\underline{N}) = \int_{\R \times \Theta} \prod_{j=0}^{2q+1} u_{j, N_j} \, dx dt .
\end{align*}

\begin{lem}[Localized nonlinear estimates]\label{lem nonlinear est}
Assume that $N_1 \geq N_2 \geq \cdots \geq N_{2q+1}$.  Then there exists $0 < b' < \frac{1}{2}$ such that one has
\begin{align}
\abs{I(\underline{N})} & \lesssim  (N_2 N_3)^{1- \frac{1}{2q} +} (N_4 N_5 \cdots N_{2q} N_{2q+1})^{\frac{3}{2} - \frac{1}{2q}+}  \prod_{j=0}^{2q+1} \norm{u_{j, N_j}}_{X^{0, b'}} , \label{eq nonlinear est1} \\
\abs{I(\underline{N})} & \lesssim (\frac{N_1}{N_0})^2 (N_2 N_3)^{1- \frac{1}{2q} +} (N_4 N_5 \cdots N_{2q} N_{2q+1})^{\frac{3}{2} - \frac{1}{2q}+}  \prod_{j=0}^{2q+1}  \norm{u_{j, N_j}}_{X^{0, b'}} . \label{eq nonlinear est2}
\end{align}
\end{lem}
\begin{rmq}
Lemma \ref{lem nonlinear est} will play an important role in the local theory. Moreover, the first estimate \eqref{eq nonlinear est1} will be used in the case $ N_0 \leq c N_1$ while the second one will be used in the case $N_0 \geq c N_1 $. 
\end{rmq}

\begin{proof}[Proof of Lemma \ref{lem nonlinear est}]
We start with \eqref{eq nonlinear est1}. By Sobolev embedding in the time variable, we have
\begin{align}\label{eq Embedding1}
\norm{f}_{L_t^{2q+2} L_x^2} = \norm{S(t) f}_{L_t^{2q+2} L_x^2} \leq \norm{S (t) f}_{H_t^{\frac{q}{2q+2}} L_x^2}  = \norm{f}_{X^{0, \frac{q}{2q+2}}} .
\end{align}
Using the definition of $X^{s,b} (\R \times \Theta)$ spaces, we have
\begin{align}\label{eq Embedding2}
\norm{u_N}_{L_t^{2q+2} L_x^{\infty}} \sim N^{\frac{3}{2}} \norm{u_N}_{L_t^{2q+2} L_x^2} \leq N^{\frac{3}{2}} \norm{u_N}_{X^{0, \frac{q}{2q+2}}} .
\end{align}

On one hand, by H\"older inequality, \eqref{eq Embedding1} and \eqref{eq Embedding2}
\begin{align*}
\abs{I(\underline{N})} & \lesssim \norm{u_{0, N_0}}_{L_t^{2q+2} L_x^2} \norm{u_{1, N_1}}_{L_t^{2q+2} L_x^2} \prod_{j=2}^{2q+1} \norm{u_{j, N_j}}_{L_t^{2q+2} L_x^{\infty}}  \lesssim (N_2 \cdots N_{2q+1})^{\frac{3}{2}} \prod_{j=0}^{2q+1} \norm{u_{j, N_j}}_{X^{0, \frac{q}{2q+2}}} .
\end{align*}
On the other hand, we can estimate $\abs{I(\underline{N})} $ using Lemma \ref{lem m-linear} with $m=q+1$. That is,
\begin{align*}
\abs{I(\underline{N})} & \lesssim \norm{u_{0, N_0}  u_{2, N_2} u_{4, N_4} \cdots u_{2q, N_{2q}}}_{L_{t,x}^2 } \norm{u_{1, N_1}  u_{3, N_3} u_{5, N_5} \cdots u_{2q+1, N_{2q+1}}}_{L_{t,x}^2 } \\
& \lesssim \square{N_2^{1- \frac{1}{2q} + \varepsilon} (N_4\cdots N_{2q})^{\frac{3}{2}-\frac{1}{2q}}  } \square{ (N_3^{1-\frac{1}{2q+2}+\varepsilon} (N_5 \cdots N_{2q+1})^{\frac{3}{2} -\frac{1}{2q}} } \prod_{j=0}^{2q+1} \norm{u_{j, N_j}}_{X^{0, b_0}} \\
& = (N_2 N_3)^{1- \frac{1}{2q} +\varepsilon} (N_4 N_5 \cdots N_{2q} N_{2q+1})^{\frac{3}{2} - \frac{1}{2q}}  \prod_{j=0}^{2q+1} \norm{u_{j, N_j}}_{X^{0, b_0}}  ,
\end{align*}
where $b_0 > \frac{1}{2}$.

Interpolation between the two estimates above implies
\begin{align*}
\abs{I(\underline{N})} & \lesssim  (N_2 N_3)^{1- \frac{1}{2q} +} (N_4 N_5 \cdots N_{2q} N_{2q+1})^{\frac{3}{2} - \frac{1}{2q}+}  \prod_{j=0}^{2q+1} \norm{u_{j, N_j}}_{X^{0, b'}} ,
\end{align*}
where $b' \in (0, \frac{1}{2})$ and  $\varepsilon' $ is a small positive power after interpolation. This finishes the computation of \eqref{eq nonlinear est1}.

Let us focus on \eqref{eq nonlinear est2} now. Recall the Green's theorem,
\begin{align*}
\int_{\Theta} \Delta f g - f \Delta g \, dx = \int_{\mathbb{S}^2} \frac{\partial f}{\partial v} g - f \frac{\partial g}{\partial v} \, d \sigma .
\end{align*}
Note that
\begin{align*}
-\Delta e_k = z_k^2 e_k ,
\end{align*}
where $ z_k^2$'s are the eigenvalues defined in \eqref{eq z_n}. Then we write
\begin{align*}
u_{0, N_0} = -\frac{\Delta}{N_0^2} \sum_{z_{n_0} \sim N_0} c_{n_0} (\frac{N_0}{z_{n_0}})^2 e_{n_0} .
\end{align*}
Define
\begin{align*}
T u_{0, N_0} & = \sum_{z_{n_0} \sim N_0} c_{n_0} (\frac{N_0}{z_{n_0}})^2 e_{n_0} , \qquad V u_{0, N_0}  = \sum_{z_{n_0} \sim N_0} c_{n_0} (\frac{z_{n_0}}{N_0})^2 e_{n_0} .
\end{align*}
It is easy to see that for all $s$
\begin{align*}
TV  u_{0, N_0} & = VT u_{0, N_0} = u_{0, N_0} ,\\
\norm{T u_{0, N_0} }_{H_x^s} & \sim \norm{ u_{0, N_0} }_{H_x^s} \sim \norm{V u_{0, N_0} }_{H_x^s} .
\end{align*} 
Using this notation, we write
\begin{align*}
u_{0, N_0} = -\frac{\Delta}{N_0^2} T u_{0, N_0} 
\end{align*}
and
\begin{align*}
I(\underline{N}) = \frac{1}{N_0^2} \int_{\R \times \Theta} T u_{0, N_0} \Delta (\prod_{j=1}^{2q+1} u_{j, N_j}) .
\end{align*}
By the product rule and the assumption that $N_1 \geq N_2 \geq \cdots \geq N_{2q+1}$, we only need to consider the two largest cases of $\Delta (\prod_{j=1}^{2q+1} u_{j, N_j}) $. They are
\begin{enumerate}
\item
$(\Delta u_{1, N_1}) \prod_{j=2}^{2q+1} u_{j, N_j}$
\item
$(\nabla u_{1, N_1}) \cdot (\nabla u_{2, N_2}) \prod_{j=3}^{2q+1} u_{j, N_j} $ .
\end{enumerate}
We denote
\begin{align*}
J_{11} (\underline{N}) & = \int_{\R \times \Theta} T u_{0, N_0} (\Delta u_{1, N_1}) \prod_{j=2}^{2q+1} u_{j, N_j} , \\
J_{12} (\underline{N}) & =\int_{\R \times \Theta} T u_{0, N_0} (\nabla u_{1, N_1}) \cdot (\nabla u_{2, N_2}) \prod_{j=3}^{2q+1} u_{j, N_j} .
\end{align*}
Using $\Delta u_N = -N^2 V u_N$, we obtain
\begin{align*}
\frac{1}{N_0^2} \abs{J_{11} (\underline{N})} \lesssim (\frac{N_1}{N_0})^2  (N_2 N_3)^{1- \frac{1}{2q} +} (N_4 N_5 \cdots N_{2q} N_{2q+1})^{\frac{3}{2} - \frac{1}{2q}+}  \prod_{j=0}^{2q+1} \norm{u_{j, N_j}}_{X^{0, b'}} .
\end{align*}
Now for $\abs{J_{12} (\underline{N})} $, we estimate it in a similar fashion that we did in \eqref{eq nonlinear est1}. On one hand, by H\"older inequality, \eqref{eq Embedding1} and \eqref{eq Embedding2}, we have
\begin{align*}
\abs{J_{12} (\underline{N})} & \lesssim N_1 N_2 (N_2 \cdots N_{2q+1})^{\frac{3}{2}} \prod_{j=0}^{2q+1} \norm{u_{j, N_j}}_{X^{0, \frac{q}{2q+2}}} .
\end{align*}
On the other hand, using Proposition \ref{prop m-linear}, we get
\begin{align*}
\abs{J_{12} (\underline{N})} & \lesssim \norm{\nabla u_{1, N_1}  u_{3, N_3} u_{5, N_5} \cdots u_{2q+1, N_{2q+1}}}_{L_{t,x}^2 } \norm{ \nabla u_{2, N_2} Tu_{0, N_0}  u_{4, N_4} \cdots u_{2q, N_{2q}}}_{L_{t,x}^2 } \\
& \lesssim N_1 N_2  (N_2 N_3)^{1- \frac{1}{2q} +} (N_4 N_5 \cdots N_{2q} N_{2q+1})^{\frac{3}{2} - \frac{1}{2q}+}  \prod_{j=0}^{2q+1} \norm{u_{j, N_j}}_{X^{0, b_0}} .
\end{align*}
Interpolation between the two estimates above implies
\begin{align*}
\frac{1}{N_0^2} \abs{J_{12} (\underline{N})} \lesssim  \frac{N_1 N_2}{N_0^2} (N_2 N_3)^{1- \frac{1}{2q} +} (N_4 N_5 \cdots N_{2q} N_{2q+1})^{\frac{3}{2} - \frac{1}{2q}+}   \prod_{j=0}^{2q+1} \norm{u_{j, N_j}}_{X^{0, b'}}. 
\end{align*}
Similarly  $\varepsilon'$ is a small positive power after interpolation.
The proof of Lemma \ref{lem nonlinear est} is complete.
\end{proof}

\begin{prop}[Nonlinear estimates]\label{prop nonlinear est}
For $s > s_c$, there exist $b, b' \in \R$ satisfying
\begin{align*}
0 < b' < \frac{1}{2} < b, \quad b + b' < 1 ,
\end{align*}
such that for every $(2q+2)$-tuple $(u_j)$ in $X^{s,b} (\R \times \Theta)$,
\begin{align*}
\norm{\prod_{j=1}^{2q+1} u_j}_{X^{s, -b'} (\R \times \Theta)} \lesssim  \prod_{j=1}^{2q+1} \norm{u_j}_{X^{s,b} (\R \times \Theta)} .
\end{align*} 
\end{prop}

\begin{proof}[Proof of Proposition \ref{prop nonlinear est}]
Based on Lemma \ref{lem nonlinear est}, we only need to consider $I= \sum_{\underline{N}} I(\underline{N})$. By symmetry, we can reduce the sum into the following two cases:
\begin{enumerate}
\item
$ N_0 \leq c N_1$
\item
$ N_0 \geq c N_1$.
\end{enumerate}

Case 1: $ N_0 \leq c N_1$.

Using Lemma \ref{lem nonlinear est} and Cauchy–Schwarz inequality, we obtain
\begin{align*}
\abs{\sum I(\underline{N})} & \lesssim \sum_{ \substack{ N_1 \geq N_2 \geq N_3 \geq \cdots \geq N_{2q+1} \\ N_0 \leq c N_1}}   (N_2 N_3)^{1- \frac{1}{2q} +} (N_4 N_5 \cdots N_{2q} N_{2q+1})^{\frac{3}{2} - \frac{1}{2q}+} \prod_{j=0}^{2q+1} \norm{u_{j, N_j}}_{X^{0, b'}} \\
& \lesssim \sum_{ \substack{ N_1 \geq N_2 \geq N_3 \geq \cdots \geq N_{2q+1} \\ N_0 \leq c N_1}}   \frac{N_0^s}{N_1^s}  (N_2 N_3)^{1- \frac{1}{2q} -s +} (N_4 N_5 \cdots N_{2q} N_{2q+1})^{\frac{3}{2} - \frac{1}{2q}-s+}  \norm{u_{0,N_0}}_{X^{-s,b'}} \prod_{j=1}^{2q+1} \norm{u_{j, N_j}}_{X^{s, b'}} \\
& \lesssim \norm{u_0}_{X^{-s,b'}} \prod_{j=1}^{2q+1} \norm{u_j}_{X^{s, b'}}   ,
\end{align*}
where $s > s_c = \frac{3}{2} - \frac{1}{q}$. For the last inequality, we sum from the smallest index $N_{2q+1}$ to the largest one  using Cauchy–Schwarz inequality. Then by the embedding $X^{s,b} \subset  X^{s , b' }  $ and duality, we have the estimate in Case 1.

Case 2: $ N_0 \geq c N_1$.

Similarly, by Lemma \ref{lem nonlinear est} and Cauchy–Schwarz inequality again, we have
\begin{align*}
\abs{\sum I(\underline{N})} & \lesssim \sum_{ \substack{ N_1 \geq N_2 \geq N_3 \geq \cdots \geq N_{2q+1} \\ N_0 \geq c N_1}}    (\frac{N_1}{N_0})^2  (N_2 N_3)^{1- \frac{1}{2q} +} (N_4 N_5 \cdots N_{2q} N_{2q+1})^{\frac{3}{2} - \frac{1}{2q}+} \prod_{j=0}^{2q+1} \norm{u_{j, N_j}}_{X^{0, b'}} \\
& \lesssim \sum_{ \substack{ N_1 \geq N_2 \geq N_3 \geq \cdots \geq N_{2q+1} \\ N_0 \geq c N_1}}   (\frac{N_1}{N_0})^{2} \frac{N_0^s}{N_1^s}   (N_2 N_3)^{1- \frac{1}{2q} -s+} (N_4 N_5 \cdots N_{2q} N_{2q+1})^{\frac{3}{2} - \frac{1}{2q} -s +}  \norm{u_{0,N_0}}_{X^{-s,b'}} \prod_{j=1}^{2q+1} \norm{u_{j, N_j}}_{X^{s, b'}} \\
& \lesssim \norm{u_0}_{X^{-s,b'}} \prod_{j=1}^{2q+1} \norm{u_j}_{X^{s, b'}}  ,
\end{align*}
where $s > s_c = \frac{3}{2} - \frac{1}{q}$. Then by the embedding  $X^{s,b} \subset  X^{s , b' }  $  and duality, we have the estimate in Case 2.
The proof of Proposition \ref{prop nonlinear est} is complete.
\end{proof}

\subsection{Local well-posedness}
With the nonlinear estimates in hand, we are ready to prove Theorem \ref{thm LWP}. 

We first recall the following lemma in \cite{BourExp, gi}
\begin{lem}\label{lem Duhamel}
Let $0 < b' < \frac{1}{2}$ and $0 < b < 1-b'$. Then for all $f \in X_T^{s, -b'} (\Theta)$, we have the Duhamel term $w(t) = \int_0^t e^{\i (t-s) \Delta} f(\tau) \, ds \in X_T^{s,b} (\Theta)$ and moreover
\begin{align*}
\norm{w}_{X_T^{s,b} (\Theta)} \leq C T^{1-b-b'} \norm{f}_{X_T^{s,-b'} (\Theta)} .
\end{align*} 
\end{lem}

\begin{proof}[Proof of Theorem \ref{thm LWP}]
Let $u_0 \in H^{s}$. We first define a map
\begin{align*}
F (u)(t) : = e^{\i t\Delta}  u_0  - \i \int_0^t e^{\i (t-s) \Delta}  (\abs{u}^{2q} u) \, ds.
\end{align*}
We prove  this locally well-posedness theory using a standard fixed point argument. 
Let $R_0>0$ and $u_0 \in H^s (\Theta)$ with $\norm{u_0}_{H^s} \leq R_0$. We show that there exists $R>0 $ and $0 < T= T(R_0) <1 $ such that $F$ is a contraction mapping from $B(0, R) \subset X_T^{s,b} (\Theta)$ onto itself.

Define $R = 2 c_0 R_0$. 
\begin{enumerate}
\item $F$ is a self map from $B(0, R) \subset X_T^{s,b} (\Theta)$ onto itself. 

For $T< 1$, by Lemma \ref{lem Duhamel} and Proposition \ref{prop nonlinear est}
\begin{align*}
\norm{F (u)}_{X_T^{s,b} (\Theta)} & \leq c_0 \norm{u_0}_{H^s} + c_1 T^{1- b - b'} \norm{\abs{u}^{2q} u}_{X_T^{s, -b'}(\Theta)} \\
& \leq  c_0 \norm{u_0}_{H^s} + c_2 T^{1- b - b'} \norm{u}_{X_T^{s, b}(\Theta)}^{2q+1} \\
& \leq c_0 \norm{u_0}_{H^s} + c_2 T^{1- b - b'} R^{2q+1}.
\end{align*}
Taking $T_1 =  (\frac{c_0}{c_2 R^{2q}})^{\frac{1}{1-b-b'}} <1$ such that $ c_2 T_1^{1- b - b'} R^{2q+1} = c_0 R$, we see that $F$ is a self map.

\item $F$ is a contraction mapping.

Similarly, by Lemma \ref{lem Duhamel} and Proposition \ref{prop nonlinear est}
\begin{align}\label{eq LWP1}
\norm{F (u) - F (v)}_{X_T^{s,b}(\Theta)} \leq c_3 T^{1- b - b'} (\norm{u}_{X_T^{s, b}(\Theta)}^{2q} + \norm{v}_{X_T^{s, b}(\Theta)}^{2q}) \norm{u-v}_{X_T^{s, b}(\Theta)} \leq c_4 T^{1- b - b'} R^{2q} \norm{u-v}_{X_T^{s, b}(\Theta)} . 
\end{align}
Taking $T_2 = (\frac{1}{2c_4 R^{2q}})^{\frac{1}{1-b-b'}}$ such that $c_4 T_2^{1- b - b'} R^{2q} = \frac{1}{2}$, we can make $F$ a contraction mapping.

Now we choose 
\begin{align}\label{eq T_R}
T_R = \min\{T_1, T_2 \}.
\end{align}
As a consequence of the fixed point argument, we have
\begin{align*}
\norm{u}_{X_T^{s,b}(\Theta)} \leq 2 c_0 \norm{u_0}_{H^s} . 
\end{align*}

\item Stability.

If $u$ and $v$ are two solutions to \eqref{eq PNLS} with initial data $u(0) =u_0$ and $v(0) = v_0$ respectively. Then by \eqref{eq LWP1} and the choice of $T_R$ \eqref{eq T_R}
\begin{align*}
\norm{u-v}_{X_T^{s,b}(\Theta)} & \leq c_0 \norm{u_0 -  v_0}_{ H^s} + c_4 T^{1- b - b'} R^{2q}  \norm{u-v}_{X_T^{s, b}(\Theta)} ,
\end{align*}
which implies
\begin{align*}
\norm{u-v}_{X_T^{s, b}(\Theta)} \leq c \norm{u_0 -  v_0}_{H^s} . 
\end{align*}
\end{enumerate}
\end{proof}

\begin{rmq}
For finite $N$ let us denote $\Pi^N L^2$ by $E_N$ and the local flow on $E_N = \Pi^N L^2$ by $\phi_t^N$. We verify easily that the local existence time obtained in Theorem \ref{thm LWP} is valid for the Galerkin approximations for the same radius $R$ of the balls. We set $\phi_t$ to be the local flow constructed in Theorem \ref{thm LWP} in $X^{s,b}$. 
\end{rmq}

\begin{prop}[Convergence of finite dimensional projections]\label{Prop:CvgLocSol}
Let $s>s_c$, $u_0 \in H^{s}$ and $(u_0^N) $ be a sequence that converges to $u_0 $ in $H^{s}$, where $u_0^N \in E_N$. Then for any $ \sigma \in (s_c,s)$
\begin{align*}
\norm{\phi_t^N u_0^N - \phi_t u_0}_{X^{\sigma,b}_T} \to 0 \quad \text{ as } N \to \infty.
\end{align*} 
\end{prop}

\begin{proof}[Proof of Proposition \ref{Prop:CvgLocSol}]
Since $u_0^N \to u_0$, there exists an $R$ such that $u_0^N \in B_R(H^{\gamma})$. Take the same $T_R$ as in \eqref{eq T_R}. Set $w_0^N = u_0^N - u_0$ and $w^N = \phi_t^N u_0^N - \phi_t u_0$, then we write
\begin{align*}
w^N & = e^{\i t\Delta} w_0^N - \i \int_0^t  e^{\i t \Delta} (\Pi^N \abs{\phi_t^N u_0^N}^{2q}\phi_t^N u_0^N - \abs{\phi_t u_0}^{2q}\phi_t u_0 ) \, ds\\
& = e^{\i t\Delta} w_0^N -  \i \int_0^t  e^{\i t \Delta} \Pi^N \parenthese{ \abs{\phi_t^N u_0^N}^{2q}\phi_t^N u_0^N  - \abs{\phi_t u_0}^{2q}\phi_t u_0} - (1-\Pi^N ) (\abs{\phi_t u_0}^{2q}\phi_t u_0 ) \, ds .
\end{align*}
Using H\"older inequality, Lemma \ref{lem Duhamel}, Proposition \ref{prop nonlinear est} and the choice of $T_R$ as in \eqref{eq T_R}, we have
\begin{align*}
\norm{w^N}_{X^{\sigma,b}_T} & \leq \norm{w_0^N}_{H^s} + C \norm{w^N}_{X^{\sigma,b}_T} (\norm{u^N}_{X^{\sigma,b}_T}^{2q} + \norm{u}_{X^{\sigma,b}_T}^{2q} ) +   \norm{(1-\Pi^N) \abs{u}^{2q} u}_{X^{\sigma,b}_T} \\
& \leq \norm{w_0^N}_{H^s} + C\norm{w^N}_{X^{\sigma,b}_T} (\norm{u^N}_{X^{\sigma,b}_T}^{2q} + \norm{u}_{X^{\sigma,b}_T}^{2q} ) + C_1  z_N^{\frac{\sigma-s}{2}} \norm{u}_{X^{s,b}_T}^{2q+1} .
\end{align*}
Therefore,
\begin{align*}
\norm{w^N}_{X^{\sigma,b}_T} & \leq \norm{w_0^N}_{H^s} + \frac{1}{2} \norm{w^N}_{X^{\sigma,b}_T} + C_2 z_N^{\frac{\sigma-s}{2}}\\
\norm{w^N}_{X^{\sigma,b}_T} & \leq 2  \norm{w_0^N}_{H^s} + 2C_2 z_N^{\frac{\sigma-s}{2}}.
\end{align*}
Recall that $z_N = (\pi N)^2$, then 
\begin{align*}
\norm{w^N}_{X^{\sigma,b}_T} \to 0 \quad \text{ as } N \to \infty.
\end{align*}
\end{proof}

\section{Invariant measure for Galerkin approximation}\label{Sect:Galerk}\label{Sect:Gal}
We start by introducing some settings and notations that will  be used from this section on. First, $(\Omega, \mathcal{F}, \P)$ is a complete probability space. If $E$ is a Banach space, we can define random variables $X:\Omega\to E$ as Bochner measurable functions with respect to $\mathcal{F}$ and $\mathcal{B}(E)$, where $\mathcal{B}(E)$ is the Borel $\sigma-$algebra of $E.$

For every positive integer $N$, we define the $N-$dimensional Brownian motion
\begin{align}
W^N (t,r) = \sum_{n=1}^N a_n e_n (r) \beta_n (t) ,\label{Noise}
\end{align}
where $(\beta_n (t))$ is a fixed sequence of independent one dimensional Brownian motions with filtration $(\mathcal{F}_t)_{t\geq 0}$. The numbers $(a_m)_{m\geq 1}$ are complex numbers such that $|a_m|$ decreases sufficiently fast to $0.$ More precisely we assume that
\begin{align}
A_r := \sum_{n \geq 1} z_n^{2r}|a_n|^2 < + \infty\quad\text{for $r\leq 1$,}\label{Noise:Cond}
\end{align}
where $(z_n^2)$ are the eigenvalues of $-\Delta$ associated to $(e_n)$. Set 
\begin{align}\label{eqA_r}
A_r^N = \sum_{n=1}^N z_n^{2r}|a_n|^2.
\end{align}

Throughout the sequel, we use $\langle \cdot,\cdot\rangle$ as the real dot product in $L^2(\Theta)$:
\begin{align*}
\langle u,g\rangle =\re\int_{\Theta}u\bar{v} \, dx, \quad \text{for all $u,v\in L^2(\Theta)$}.
\end{align*}
And for a smooth enough functional $F: E^N\to \R\ ;\ u\mapsto F(u)$, we denote the first derivative of $F$ by $F'(u,v)$ and the second derivative of $F$ by $F''(u;v,w)$.
 Occasionally this notation is used to write a duality bracket when the context is not confusing.

The rest of this section is organized as follows:
\begin{enumerate}
\item[Step 1] 
In Section \ref{ssec 4.1}, we introduce fluctuation-dissipation  equations \eqref{SNLS}. We obtain the formal estimates on the dissipation rate of the mass and the energy functionals $M(u)$ and $E(u)$ (given in \eqref{eq Intro mass} and \eqref{eq Intro energy}) where $u$ is the (formal) solution of \eqref{SNLS}. These dissipation rates play a crucial in the globalization argument. They, first, provide invariant measures with uniform large deviation estimates needed in the Bourgain's argument. They also provide information about the infinite-dimensional invariant measures, specially using $\mathcal{M}(u)$ we derive large size property of the data covered in the support of this measure.

\item[Step 2]
In Section \ref{ssec 4.2}, we present GWP result for \eqref{SNLS}.

\item[Step 3]
In Section \ref{ssec 4.3}, we define the Markov semi-group associated to the solutions, and in Section \ref{ssec 4.4}, we construct stationary measures for \eqref{SNLS} and derive uniform in $(\al,N)$ bounds for $\mathcal{M}(u)$ and $\mathcal{E}(u)$. 

\item[Step 4]
In Section \ref{ssec 4.5}, We then pass to the limit $\al\to 0$.
\end{enumerate}

\subsection{Introduction of stochastic equations and their estimates}\label{ssec 4.1}

Let us introduce the following fluctuation-dissipation of the Galerkin projections of \eqref{NLS}.
\begin{align}\label{SNLS}
du = \i (\Delta u - \Pi^N \abs{u}^{2q} u) \, dt - \alpha \mathcal{L}(u) \, dt +  \sqrt{\alpha} \, d W^N
\end{align}
where $\alpha \in (0,1)$, $\rho:\R_+\to\R_+$ is any increasing function and
\begin{align}
\mathcal{L}(u)=e^{\rho(\|u\|_{H^{\beta-}})}[(-\Delta)^{\beta-1}u+\Pi^N|u|^{2q}u], \quad \beta \in [1,\frac{3}{2}].
\end{align}
 
We have the following
\begin{align}\label{eqMcal}
M'(u,e^{\rho(\|u\|_{H^{\beta-}})}[(-\Delta)^{\beta-1}+\Pi^N|u|^{2q}]u)=e^{\rho(\|u\|_{H^{\beta -}})}\left(\|u\|_{H^{\beta-1}}^2+\|u\|_{L^{2q+2}}^{2q+2}\right)=:\mathcal{M}(u).
\end{align}
\begin{align}\label{eqEcal}
E'(u,e^{\rho(\|u\|_{H^{\beta-}})}[(-\Delta)^{\beta-1}+\Pi^N|u|^{2q}]u) &=  e^{\rho(\|u\|_{H^{\beta -}})}\left(\|u\|_{H^\beta}^2+\langle\Pi^N|u|^{2q}u,-\Delta u+|u|^{2q}u\rangle +\langle(-\Delta)^{\beta-1}u,|u|^{2q}u\rangle\right)=:\mathcal{E}(u).
\end{align}
Let us treat the term $\langle(-\Delta)^{\beta-1}u,|u|^{2q}u\rangle$: We have, using that $1\leq \beta\leq \frac{3}{2}$ and noting that $2\beta -2<\beta$, that
\begin{align*}
|\langle(-\Delta)^{\beta-1}u,\Pi^N|u|^{2q}u\rangle|\leq \frac{1}{2}\|u\|_{H^{2\beta -2}}^2+\frac{1}{2}\|\Pi^N|u|^{2q}u\|_{L^{2}}^{2}\leq \frac{1}{2}\|u\|_{H^{\beta}}^2+\frac{1}{2}\|\Pi^N|u|^{2q}u\|_{L^{2}}^{2}.
\end{align*}
Also, noticing that $\Pi^N\Delta u=\Delta u,$ we have
\begin{align*}
\langle\Pi^N|u|^{2q}u,-\Delta u+|u|^{2q}u\rangle =\|\Pi^N|u|^{2q}u\|_{L^{2}}^{2}+\langle|\nabla u|^2,|u|^{2q}\rangle+\langle u\nabla (u^{q}\bar{u}^{q}),\nabla u  \rangle.
\end{align*}
Let us remark that
\begin{align*}
\langle u\nabla (u^{q}\bar{u}^{q}),\nabla u  \rangle &= 2q \langle uu^{q-1}(\nabla u)\bar{u}^{q},\nabla u  \rangle+q \langle u\bar{u}^{q-1}(\nabla \bar{u})u^{q},\nabla u  \rangle\\
& \geq q \langle |u|^{2q}, |\nabla u|^2 \rangle - q \langle |u|^{2q}, |\nabla u|^2\rangle= 0.
\end{align*}
Hence we obtain
\begin{align*}
\langle\Pi^N|u|^{2q}u,-\Delta u+|u|^{2q}u\rangle \geq \frac{1}{2}\|\Pi^N|u|^{2q}u\|_{L^{2}}^{2}+\langle|\nabla u|^2,|u|^{2q}\rangle.
\end{align*}
Overall,
\begin{align*}
 \mathcal{E}(u) &\geq e^{\rho(\|u\|_{H^{\beta -}})}\left(\|u\|_{H^\beta}^2-\frac{1}{2}\|u\|_{H^{\beta}}^2-\frac{1}{2}\|\Pi^N|u|^{2q}u\|_{L^{2}}^{2}+\|\Pi^N|u|^{2q}u\|_{L^{2}}^{2}+\langle|\nabla u|^2,|u|^{2q}\rangle\right).
\end{align*}

We arrive at
\begin{align*}
 \mathcal{E}(u) &\geq e^{\rho(\|u\|_{H^{\beta -}})}\left(\frac{1}{2}\|u\|_{H^{\beta}}^2+\frac{1}{2}\|\Pi^N|u|^{2q}u\|_{L^{2}}^{2}+\langle|\nabla u|^2,|u|^{2q}\rangle\right).
\end{align*}
We remark that
\begin{align}
\|u\|_{L^2}^2\leq \mathcal{M}(u)\leq E(u)e^{\rho(\norm{u}_{H^{\beta -}})}\lleq \mathcal{E}(u), \quad \text{(since $\beta\leq 3/2$, we have $\beta-1\leq 1/2< 1.$)}\label{Mcal:leq:Ecal:Subcrit}
\end{align}
and
\begin{align}\label{Est_bound:on:E}
 \|u^{q+1}\|_{\dot{H}^1}^2+\|u\|_{H^\beta}^2\leq (\|u^{q+1}\|_{\dot{H}^1}^2+\|u\|_{H^\beta}^2)e^{\rho(\|u\|_{H^{\beta -}})}\leq 2\mathcal{E}(u).
\end{align}
\begin{rmq}
The role of the term $|u|^{4q}u$ in the damping allowed to manage the `interaction' between $(-\Delta)^{\beta-1}$ and the nonlinearity, as in the term $\langle(-\Delta)^{\beta-1}u,|u|^{2q}u\rangle$. On the other hand the term $e^{\rho(\|u\|_{\beta-})}$ allows to make use of a Bourgain globalization argument and provides a control on the resulting solutions (see Section \ref{Sect:Alg:pow}). 
\end{rmq}

\subsection{Global well-posedness for stochastic equations}\label{ssec 4.2}

Below we present a probabilistic global wellposedness result for \eqref{SNLS}.
\begin{prop}\label{prop u_alpha^N}
For every $\mathcal{F}_0$-measurable random variable $u_0$ in $E_N$, we have
\begin{enumerate}
\item
for $\mathbb{P}$-almost all $\omega \in \Omega$, \eqref{SNLS} with initial condition $u(0) = \Pi^N  u_0^{\omega}$ has a unique global solution $u_{\alpha}^N (t; u_0^{\omega})$;

\item
if $u_{0,n}^{\omega} \to u_0^{\omega}$ then $u_{\alpha}^N (\cdot; u_{0,n}^{\omega}) \to u_{\alpha}^N(\cdot; u_0^{\omega})$ in $C_t L^2$;

\item
the solution $u^N_\al$  is adapted to $(\mathcal{F}_t)$.
\end{enumerate}
\end{prop}

\begin{proof}[Proof of Proposition \ref{prop u_alpha^N}]
\begin{enumerate}
\item
The proof is rather classical. We can follow for instance \cite{syNLS7}. It consists in splitting the equation \eqref{SNLS} into the following two complementary problems:
\begin{align}\label{Equ:z}
\begin{cases}
dz =\i\Delta z \, dt+\sqrt{\alpha} \, d W^N,\\
z(0)=0;
\end{cases}
\end{align}
and
\begin{align}\label{Equ:v}
\begin{cases}
\dt v =\i(\Delta v-|v+z|^{2q}(v+z))-\alpha e^{\rho(\|v+z\|_{H^{\beta -}})}\left[(-\Delta)^{\beta-1}(v+z)+\Pi^N|v+z|^{2q}(v+z)\right],\\
v(0)= u_0.
\end{cases}
\end{align}
The problem \eqref{Equ:z} is solved by the stochastic convolution
\begin{align*}
z(t)=\sqrt{\alpha}\int_0^te^{\i(t-s)\Delta} \, dW^N(s).
\end{align*}
Using the Doob's maximal inequality (see Theorem \ref{Thm:Doob}) and the It\^o formula, we see that $z$ satisfies the estimate
\begin{align}
\E\left(\sup_{t\in [0,T]}\|z(t)\|_{L^2}\right)^m\lleq C_m\alpha\label{Est:z}\quad \text{for any $T>0$},
\end{align}
for any $m>1.$

The equation \eqref{Equ:v} can be solved locally almost surely by a fixed point argument, clearly the vector field involved in the equation is smooth. In order to show that the local solution $v$ is global almost surely, we first claim the following almost sure type control on $v$:
\begin{align}
\sup_{t\in[0,T]}\|v\|_{L^2}^2\leq \|u_0\|_{L^2}^2+C^3_\alpha(\omega,T,N)\quad\text{for any $T>0$}. \label{Est:v:point}
\end{align}
To prove \eqref{Est:v:point}, let us compute the derivative of $\frac{\|v\|_{L^2}^2}{2}$ and use the equation \eqref{Equ:v}, we obtain, by adding $-z+z$,
\begin{align*}
\frac{d}{dt}\left[\frac{\|v\|_{L^{2}}^{2}}{2}\right] &\leq |\langle z,|v+z|^{2q}(v+z)\rangle|-\al e^{\rho(\|v+z\|_{H^{\beta -}})}\left(\|v\|_{H^{\beta -1}}^2+\|v+z\|_{L^{2q+2}}^{2q+2}\right)\\
& \quad +\al e^{\rho(\|v+z\|_{H^{\beta -}})}\left(\langle|v+z|^{2q}(v+z),z\rangle-\langle (-\Delta)^{\beta-1}v,z\rangle \right)\\
&\leq \frac{\|z\|_{L^{2q+2}}^{2q+2}}{2\al}+\frac{\al\|v+z\|_{L^{2q+2}}^{2q+2}}{2}-\al e^{\rho(\|v+z\|_{H^{\beta -}})}\left(\|v\|_{H^{\beta -1}}^2+\|v+z\|_{L^{2q+2}}^{2q+2}\right)\\
& \quad +\al e^{\rho(\|v+z\|_{H^{\beta -}})}\left(\frac{1}{2}\|v\|_{H^{\beta -1}}^2+\frac{1}{2}\|v+z\|_{L^{2q+2}}^{4q+2}+\frac{1}{2}\|z\|_{L^{2q+2}}^{4q+2}+\frac{1}{2}\|z\|_{H^{\beta-1}}^2\right)\\
&\leq \frac{\|z\|_{L^{2q+2}}^{2q+2}}{2\al}+C\alpha\|z\|_{L^{2q+2}}^{2q+2}+2{\al} e^{\rho(\|v+z\|_{H^{\beta -}})}\left(\|z\|_{L^{2q+2}}^{2q+2}+\|z\|_{H^{\beta-1}}^2-\|v\|_{H^{\beta -1}}^2\right)\\
&\leq C(\al,N)\|z\|_{L^{2}}^{2q+2}+2{\al} e^{\rho(\|v+z\|_{H^{\beta -}})}\left(c(N)(\|z\|_{L^2}^{2q+2}+\|z\|_{L^2}^2)-\|v\|_{H^{\beta -1}}^2\right) .
 \end{align*} 
Now we have that for $\P-$almost all $\omega\in\Omega$ for all $T$, there is a constant $C_\al(\omega,T)$ such that 
\begin{align}
\sup_{t\in[0,T]}\|z(\omega,t)\|_{L^2}\leq C_\al(\omega,T).\label{Controlonz}
\end{align}
Now, for a fixed $\omega$ such that \eqref{Controlonz} holds, fix any $T>0$. Let $t\in[0,T]$, we have the following two complementary scenarios:
\begin{enumerate}
\item either $\|v(\omega,t)\|_{H^{\beta-1}}^2 \leq c(N)[C_\al(\omega,T)^{2q+2}+C_\al(\omega,T)^2]$; in this case there is $C^0_\al(\omega,T,N)$ such that
\begin{align*}
\frac{d}{dt}\left[\frac{\|v\|_{L^2}^2}{2}\right]\leq C^0_\al(\omega,T,N),
\end{align*}
\item or $\|v(\omega,t)\|_{H^{\beta-1}}^2 > c(N)[C_\al(\omega,T)^{2q+2}+C_\al(\omega,T)^2]$. In this case $$e^{\rho(\|v+z\|_{H^{\beta -}})}\left(c(N)(\|z\|_{L^2}^{2q+2}+\|z\|_{L^2}^2)-\|v\|_{H^{\beta -1}}^2\right)<0$$ and
\begin{align*}
\frac{d}{dt}\left[\frac{\|v\|_{L^2}^2}{2}\right]\leq C(\al,N)\|z\|_{L^{2}}^{2q+2}\leq C^1_\al(\omega,T,N).
\end{align*}
\end{enumerate}
Therefore, taking $C^2_\alpha=\max(C^0_\alpha,C^1_\alpha)$, we obtain
\begin{align*}
\frac{d}{dt}\left[\frac{\|v\|_{L^2}^2}{2}\right]\leq C^2_\alpha(\omega,T,N)\quad\text{for any $t\in[0,T]$}.
\end{align*}
We then obtain the claimed estimate \eqref{Est:v:point} 

This gives the globalization for \eqref{SNLS}.

\item
For the continuity with respect to the initial data, let us set 
\begin{align*}
F(u)=\i(\Delta u-|u|^{p-1}u)-\alpha e^{\rho(\|u\|_{H^{\beta-}})}[(-\Delta)^{\beta-1}+\Pi^N|u|^{2q}]u .
\end{align*}
Using the mean value theorem, we have
\begin{align*}
\dt(u-v)=(u-v)\int_0^1F'(ru+(1-r)v) \,dr.
\end{align*}
Now, multiplying by $u-v$ and integrating in $x$, we can use the Gronwall inequality to obtain
\begin{align*}
\|u(t)-v(t)\|_{L^2}\leq \|u(0)-v(0)\|_{L^2}e^{\sup_{x\in\Theta}\int_0^1|F'(ru+(1-r)v)| \,dr}.
\end{align*}
Hence the desired property.

\item The adaptation property comes from the iteration hidden in the fixed point argument.
\end{enumerate}
\end{proof}

\subsection{The Markov semi-group}\label{ssec 4.3}
Let us introduce the transition probability $T_{t,\alpha} (v, \Gamma) = \mathbb{P} (u_{\alpha}^N (t ; v) \in \Gamma)$, where $v \in L^2$, $\Gamma \in \mathcal{B}(L^2)$, and define the Markov operators:
\begin{align*}
P_{t,\alpha}^N f (v) & = \int_{L^2} f (u) T_{t,\alpha}^N (v ; du) : L^{\infty} (L^2 ; \R) \to L^{\infty} (L^2 ; \R)\\
P_{t,\alpha}^{N*} \lambda (\Gamma) & = \int_{L^2} \lambda (du) T_{t,\alpha}^N (u,\Gamma) : \mathfrak{p} (L^2) \to \mathfrak{p} (L^2) .
\end{align*}
From the continuity of the solution with respect to initial data, we obtain the Feller property
\begin{align*}
P_{t, \alpha}^N C_b (L^2)  \subset  C_b (L^2) .
\end{align*}
Also,  the following duality relation holds
\begin{align*}
\langle P_{t,\alpha}^N f , \lambda\rangle = \langle f, P_{t, \alpha}^{N *} \lambda\rangle
\end{align*}
where
\begin{align*}
\langle f,\lambda\rangle : = \int_{L^2} f(u) \lambda (du)\quad \text{where $f\in C_b(L^2)$ and $\mu \in \mathfrak{p}(L^2)$}.
\end{align*}

\subsection{Stationary measures}\label{ssec 4.4}
Let us start with recalling classical results that are essential in the arguments of the stationary measures existence. We will  present the proofs of  Lemma \ref{Lem:KBarg} and Lemma \ref{Lem:MeanMeas} in Appendix \ref{Apx B}.
\begin{lem}[Krylov-Bogoliubov argument]\label{Lem:KBarg}
Let $(P_t)_{t \geq 0}$ be a Feller semi-group (a Markov semi-group satisfying the Feller property) on Banach space $X$ if there exists $t_n \to \infty$ and $\mu \in \mathfrak{p} (X)$ such that $\frac{1}{t_n} \int_0^{t_n} P_{t}^* \delta_0 \, dt \rightharpoonup \mu $ in $X$, then $P_t^* \mu = \mu$ for all $t \geq 0$. $\delta_0$ is the Dirac measure at 0.
\end{lem}

\begin{lem}\label{Lem:MeanMeas}
Set 
\begin{align*}
\nu_n = \frac{1}{t_n} \int_0^{t_n} P_t^* \delta_0 \, dt.
\end{align*}
Assume that $X$ is compactly embedded in $X_0$ and 
\begin{align*}
\int_{X} \norm{u}_{X}^r \nu_n (du) \leq C
\end{align*}
for some $r>0$, where $C$ does not depend on $n$. Then there exists $\mu \in \mathfrak{p} (X)$ such that $\nu_n \rightharpoonup \mu$ weakly on $X_0$.
\end{lem}

In order to  obtain an stationary measure, we also need the following proposition.
\begin{prop}\label{prop Est:M:General}
Let $u_0$ be an $\mathcal{F}_0-$ measurable, $E_N-$valued random variable such that $\mathbb{E} M (u_0) < + \infty$. Then
\begin{align}
\mathbb{E} M(u) + \alpha \int_0^t \mathbb{E} \mathcal{M}(u) \, ds = \mathbb{E} M(u_0) + \frac{\alpha}{2} A_0^N t\label{Est:M:General}
\end{align}
where $\mathcal{M}$ is defined as in \eqref{eqMcal}.
\end{prop}

\begin{proof}[Proof of Proposition \ref{prop Est:M:General}]
Let us recall the equation
\begin{align*}
du = \square{ \i (\Delta u - \Pi^N \abs{u}^{2q} u) - \alpha e^{\rho(\|u\|_{H^{\beta -}})}\left((-\Delta)^{\beta -1}u+\Pi^N|u|^{2q}u\right)} \,dt + \sqrt{\alpha} \, dW^N .
\end{align*}
Applying the It\^o formula (see Theorem \ref{Thm:Ito}), we write
\begin{align*}
dF = F' (u, du) + \frac{\alpha}{2} \sum_{n=1}^N |a_n|^2 F'' (u; e_n , e_n ) \, dt. 
\end{align*}
Taking $F=M =\frac{1}{2} \norm{u}_{L^2}^2$, we obtain
\begin{align*}
dM(u) = \langle u , -\alpha e^{\rho(\|u\|_{H^{\beta -}})}\left((-\Delta)^{\beta -1}u+\Pi^N|u|^{2q}u\right)\rangle \, dt + \frac{\alpha }{2} \sum_{n=1}^N |a_n|^2 \, dt + \sqrt{ \alpha} \langle u,dW^N\rangle.
\end{align*}
That is,
\begin{align*}
M(u) + \alpha \int_0^t \langle u, e^{\rho(\|u\|_{H^{\beta -}})}\left((-\Delta)^{\beta -1}u+\Pi^N|u|^{2q}u\right)\rangle \, ds = M(u_0) + \frac{\alpha }{2} A_0^N t + \int_0^t \sqrt{\alpha } \langle u, dW^N\rangle.
\end{align*}
Therefore,
\begin{align*}
\mathbb{E} M(u) + \alpha \int_0^t \mathbb{E} \mathcal{M}(u) \, ds = \mathbb{E} M(u_0) + \frac{\alpha}{2} A_0^N t.
\end{align*}
Hence we finish the proof of Proposition \ref{prop Est:M:General}.
\end{proof}

\begin{cor}\label{cor Est:M:General}
Assume $\mathcal{D} (u_0) =\delta_0$ and set $\nu_t = \frac{1}{t} \int_0^t P_{s, \alpha}^{N*} \delta_0 \, dt$, then 
\begin{align*}
\int_{L^{2}} \mathcal{M}(u) \nu_t (dv) \leq \frac{1}{2} A_0^N.
\end{align*}
\end{cor}

\begin{proof}[Proof of Corollary \ref{cor Est:M:General}]
Using \eqref{Est:M:General} with $u_0=0$, we have
\begin{align*}
\mathbb{E} \norm{u}_{L^2}^2 + \alpha \int_0^t \mathbb{E} \mathcal{M}(u) \, ds = \frac{\alpha}{2} A_0^N t.
\end{align*}
Let us write
\begin{align*}
\frac{1}{t} \int_0^t \mathbb{E} \mathcal{M}(u) \, ds  = \frac{1}{t} \int_0^t \int_{L^{2}} \mathcal{M}(u) P_{s, \alpha}^{N*} \delta_0 (dv) \, ds= \int_{L^{2}} \mathcal{M}(u) \frac{1}{t} \int_0^t P_{s,\alpha}^{N*} \delta_0 \, ds (dv) = \int_{L^{2}} \mathcal{M}(u) \nu_t (dv).  
\end{align*}
Therefore,
\begin{align*}
\int_{L^{2}} \mathcal{M}(u) \nu_t (dv) \leq \frac{1}{2} A_0^N.
\end{align*}
\end{proof}

As a direct consequence of Lemmas \ref{Lem:KBarg}, \ref{Lem:MeanMeas} and Corollary \ref{cor Est:M:General}, we arrive at the following statement:
\begin{cor}
For any $\alpha \in (0,1)$ and $N \in \N $, \eqref{SNLS} has an stationary measure $\mu_{\alpha}^N$. 
\end{cor}
Now let us focus on statistical estimates on the stationary measures $\mu_{\alpha}^N$.
\begin{prop}\label{prop Est:M:N,al}
We have that
\begin{align}
\int_{L^{2}} \mathcal{M}(u) \mu_{\alpha}^N (dv) = \frac{A_0^N}{2}. \label{Est:M:N,al}
\end{align}
\end{prop}

\begin{proof}[Proof of Proposition \ref{prop Est:M:N,al}]
It suffices to show that 
\begin{align*}
\int_{L^{2}} \norm{v}_{L^2}^2 \mu_{\alpha}^N (dv) < + \infty,
\end{align*}
since then by  \eqref{Est:M:General} and invariance, we obtain the identity.

Now, let $\chi$ be a smooth bump function having value $1$ on $[0,1]$ and $0$ on $[2,\infty)$ and set $\chi_R(x)=\chi(\frac{x}{R})$.
Then we have by definition of $\rho$
\begin{align*}
\int_{L^{2}} \chi_R (v) \norm{v}_{L^2}^2 \nu_t (dv) \leq \int_{L^{2}} \norm{v}_{L^2}^2e^{\rho(\norm{v}_{H^{\beta -}})} \nu_t (dv)\leq\int_{L^{2}} \mathcal{M}(v) \nu_t (dv) \leq \frac{A_0^N}{2}.
\end{align*}
Go to the limit along subsequence $t_n \to \infty$, 
\begin{align*}
\int_{L^{2}} \norm{v}_{L^2}^2 \chi_R (v) \mu_{\alpha}^N (dv) \leq \frac{A_0^N}{2}.
\end{align*}
Take $R \to \infty$, using Fatou's lemma
\begin{align*}
\int_{L^{2}} \norm{v}_{L^2}^2 \mu_{\alpha}^N (dv) \leq \frac{A_0^N}{2} < + \infty.
\end{align*}
Now we finish the proof of Proposition \ref{prop Est:M:N,al}.
\end{proof}

We have the following estimates for the stationary measures:
\begin{prop}\label{prop StatMeas}
We have that
\begin{align}
\int_{L^2}\mathcal{E}(u)\mu^N_\alpha(du) &\leq C_1,\label{Est:E:N,al}\\
\int_{L^2}E(u)\mathcal{E}(u)\mu^N_\alpha(du) &\leq C_2,\label{Est:EEcal:N,al}
\end{align}
where $C_i,\ i=1,2$ are independent of $(N,\alpha)$ and $\mathcal{E}(u)$ is defined as in \eqref{eqEcal}.
\end{prop}
\begin{proof}[Proof of Proposition \ref{prop StatMeas}]
For technical reasons, the estimate \eqref{Est:E:N,al} is simpler than \eqref{Est:EEcal:N,al}. We need just prove the second inequality. Let us apply the It\^o formula to $E^2$:
\begin{align*}
dE^2 &=2E \, dE(u)+\frac{\al}{2}\sum_{m=0}^N(2|a_m|^2E(u)E''(u,e_m,e_m)+2|E'(u,e_m)|^2) \, dt\\
&=-2E(u)\mathcal{E}(u) \,dt+\frac{\al}{2}\sum_{n=0}^N(2|a_m|^2E(u)E''(u,e_m,e_m)+2|E'(u,e_m)|^2) \, dt+2\sqrt{\al}\sum_{m=0}^N E(u)E'(u,a_me_m) \, d\beta_m.
\end{align*}
Now,
\begin{align*}
EE''(u,e_m,e_m)&= \langle-\Delta e_m,e_m\rangle E+\langle|u|^{2q},e_m^2\rangle E\leq (z_m^2 + \|u\|_{L^{2q}}^{2q}\|e_m\|_{L^\infty}^2)E,
\end{align*}
Now, recall that $\|e_m\|_{L^\infty}^2\lleq z_m,$ where $z_m^2$ are the eigenvalues of radial $-\Delta$. We then have
\begin{align*}
EE''(u,e_m,e_m)&\leq \epsilon z_m\|u\|_{L^{2q+2}}^{2q+2}E+(z_m^2+C(\epsilon) z_m)E,
\end{align*}
for any $\epsilon>0$, and $C(\epsilon)$ is a corresponding number. Then
\begin{align*}
\sum_{m=0}^N|a_m|^2E(u)E''(u,e_m,e_m)&\leq \epsilon A_{\frac{1}{2}}^NE(u)\mathcal{E}(u)+ (C(\epsilon) A_{\frac{1}{2}}^N+A_1^N)E(u),
\end{align*}
where $A_{r}^N$ is defined as in \eqref{eqA_r}. 
Now take $\epsilon=(A_{\frac{1}{2}}^N)^{-1}$ and obtain
\begin{align*}
\sum_{m=0}^N|a_m|^2E(u)E''(u,e_m,e_m)&\leq E(u)\mathcal{E}(u)+ (C A_{\frac{1}{2}}^N+A_1^N)E(u) .
\end{align*}
On the other hand,
\begin{align*}
|E'(u,e_m)|^2=|\langle-\Delta u+|u|^{2q}u,e_m\rangle| &\leq 2|\langle-\Delta u,e_m\rangle|^2+2|\langle|u|^{2q}u,e_m\rangle|^2 \leq 2\|u\|_{\dot{H}^1}^2z_m^2+2\|\Pi^N |u|^{2q}u\|_{L^2}^2,
\end{align*}
therefore
\begin{align*}
\sum_{m=0}^N|a_m|^2|E'(u,e_m)|^2\leq A_1^NE(u)+A_0^N\|\Pi^N |u|^{2q}u\|_{L^2}^2\leq (C A_{\frac{1}{2}}^N+A_1^N)E(u)+A_0^N\mathcal{E}(u).
\end{align*}
It remains to integrate in $t$, to take the expectation, to use the stationarity and the estimate \eqref{Est:E:N,al}.
\end{proof}

Now, let $\chi$ be a smooth bump function having value $1$ on $[0,1]$ and $0$ on $[2,\infty).$ Clearly the function $\chi$ and its derivative are bounded, we can take a universal constant that bounds $\chi$ and its first two derivative. Set 
\begin{align}\label{eq chi}
\chi_R(x)=\chi(\frac{x}{R}).
\end{align}
We then have the following:
\begin{align*}
|\chi_R^{(m)}(x)| &\leq CR^{-m}.
\end{align*}

Using these functions, we establish a useful estimate on the `tail' of the measures: 
\begin{prop}\label{prop Est:MR:N,al}
For any $R>0$, the following estimate holds
\begin{align}\label{Est:MR:N,al}
\int_{L^2}\mathcal{M}(u)(1-\chi_R(\|u\|_{L^2}^2))\mu_\al^N(du)\leq CR^{-1},
\end{align}
where $C$ is independent of $(\al,N)$.
\end{prop}
\begin{proof}[Proof of Proposition \ref{prop Est:MR:N,al}]
Let $F_R(u)=\|u\|_{L^2}^2(1-\chi_R)(\|u\|_{L^2}^2)$. Applying It\^o's formula we see the following
\begin{align*}
dF_R +2\alpha \mathcal{M}(u)(1-\chi_R(\|u\|_{L^2}^2)) & =-2\alpha\mathcal{M}(u)\chi'_R(\|u\|_{L^2}^2) +\frac{\al}{2}\left((1-\chi_R(\|u\|_{L^2}^2))A_0^N+2\chi_R'(\|u\|_{L^2}^2)\sum_{|m|=0}^N|a_m|^2\langle u,e_m\rangle^2\right)\\
& \quad +\frac{\al}{2}\|u\|_{L^2}^2\left[A_0^N\chi_R'(\|u\|_{L^2}^2)+\chi_R''(\|u\|_{L^2}^2)\sum_{|m|=0}^N|a_m|^2\langle u,e_m\rangle^2\right].
\end{align*}
Let us make use of the invariance and \eqref{Est:M:N,al}, we obtain
\begin{align*}
\E\mathcal{M}(u)(1-\chi_R(\|u\|_{L^2}^2))\leq A_0^N\E(1-\chi_R(\|u\|_{L^2}^2))+\frac{C(A_0)}{R},
\end{align*}
where $C(A_0)>0$ is a constant depending only on $A_0$.

Now we use the Markov inequality and \eqref{Est:M:N,al} to obtain
\begin{align*}
\int_{L^2}(1-\chi_R(\|u\|_{L^2}^2))\mu_\al^N\leq C R^{-1} 
\end{align*}
where $C$ is independent of $(\al, N)$. We arrive at
\begin{align*}
\int_{L^2}\mathcal{M}(u)(1-\chi_R(\|u\|_{L^2}^2))\mu_\al^N\leq C R^{-1},
\end{align*}
hence the proof of Proposition \ref{prop Est:MR:N,al}] is complete.
\end{proof}

\subsection{Inviscid limit}\label{ssec 4.5} In this subsection we show that the fluctuation-dissipation equation \eqref{SNLS} converges as $\al\to 0$ to the following  $N-$Galerkin projection of the equation \eqref{NLS}:
\begin{align}
\dt u = \i (\Delta u - \Pi^N \abs{u}^{2q} u), \quad u(0)= \Pi^N u_0\label{NLS-projN}\quad N<\infty.
\end{align}
We obtain invariant measures for \eqref{NLS-projN} satisfying some crucial bounds that will be central in the passage to the limit $N\to\infty$.

We use the conservation of the mass to see that \eqref{NLS-projN} is globally well-posed on  $L^2$.  Uniqueness and continuity follow through usual methods. Since we are in finite dimensions the nonlinearity is regular, we therefore obtain global well-posedness. This being set, let us define the associated global flow by 
\begin{align*}
\phi^{N}_t:E_N\to E_N,\ \ u_0\mapsto \phi^N_tu_0
\end{align*}
where $\phi^N_t(u_0)=:u(t,u_0)$ represents the solution to \eqref{NLS-projN} starting at $u_0$. Let us set the corresponding Markov groups 
\begin{align*}
P^N_tf(v) &=f(\phi^{N}_t(\Pi^Nv));\quad C_b(L^2)\to C_b(L^2),\\
P^{N*}_{t}\lambda(\Gamma) &=\lambda({\phi^N_{-t}}(\Gamma)); \quad \mathfrak{p}(L^2)\to \mathfrak{p}(L^2).
\end{align*} 
From the estimate \eqref{Est:M:N,al}, we have the weak compactness of $(\mu_{\al_k}^N)$ with respect to the topology of $H^{\beta -\epsilon},$ therefore there exists a subsequence $(\mu_{\al_{k}}^N)=:(\mu_{k}^N),$ converging to a measure $\mu^N$ on $L^2$.

\begin{prop}\label{prop mu^N}
For any $N$, as $\alpha \to 0$, there exists a subsequence $( \alpha_k) \to 0$ and $\mu^N \in \mathfrak{p}(L^2)$ such that
\begin{enumerate}
\item
$\mu_{k}^N \rightharpoonup \mu^N$ on $H^{\beta -}$

\item
$\mu^N$ is invariant for $\phi_t^N$

\item We have the estimates
\begin{align}
&\int_{L^2} \mathcal{M}(v) \mu^N (dv) = \frac{ A_0^N}{2} ; \label{Est:M:N}\\
& \int_{L^2}\mathcal{M}(u)(1-\chi_R(\|u\|_{L^2}^2))\mu^N(du) \leq CR^{-1};\label{Est:MR:N}\\
& \int_{L^2}\mathcal{E}(u)\mu^N(du) \leq C_1,\label{Est:E:N}\\
&\int_{L^2}E(u)\mathcal{E}(u)\mu^N(du) \leq C_2,\label{Est:EEcal:N}
\end{align}
where $C_i,\ i=1,2$ are independent of $N.$
\end{enumerate}
\end{prop}

\begin{proof}[Proof of Proposition \ref{prop mu^N}]
This proof will consist the following two parts.
\begin{enumerate}
\item {\bfseries Estimates.}
The estimates \eqref{Est:E:N}, \eqref{Est:EEcal:N,al} and \eqref{Est:MR:N} follow respectively from \eqref{Est:E:N,al}, \eqref{Est:EEcal:N} and  \eqref{Est:MR:N,al}  and the lower semicontinuity of $E(u)$, $\mathcal{E}(u)$ and $\mathcal{M}(u)$. 

Now let us focus on the proof of \eqref{Est:M:N}, it suffices to show the middle equality: for $\chi_R$ a $C^\infty$ function on $\R$ having the value $1$ on $[0,1]$ and the value $0$ on $[2,\infty),$ we write
\begin{align*}
\frac{A_{0}^N}{2}-\int_{L^2}(1-\chi_R(\|u\|_{L^2}^2))\mathcal{M}(u)\mu_k^N(du)\leq \int_{L^2}\chi_R(\|u\|_{L^2}^2)\mathcal{M}(u)\mu_k^N(du)\leq \frac{A_0^N}{2}.
\end{align*}
Using \eqref{Est:MR:N,al},
\begin{align*}
\frac{A_{0}^N}{2}-C_2R^{-1}\leq \int_{L^2}\chi_R(\|u\|_{L^2}^2)\mathcal{M}(u)\mu_k^N(du)\leq \frac{A_0^N}{2}.
\end{align*}
We pass to the limits $k\to\infty$, then $R\to\infty$, and we obtain the claim \eqref{Est:M:N}. 
\item {\bfseries Invariance.} It suffices to show the invariance under $\phi_N^t,\ t>0$. Indeed
For $t<0,$ we have, using the invariance for positive times, that
\begin{align*}
\mu^N(\Gamma)=\mu^N(\phi^N_{t}\Gamma)=\mu^N(\phi^N_{2t}\phi^N_{-t}\Gamma)=\mu^N(\phi^N_{-t}\Gamma),
\end{align*}
as wished. Now the proof of the invariance for positive times is summarized in the following diagram.
$$
\hspace{10mm}
\xymatrix{
  P_{t,k}^{N*}\mu_k^N \ar@{=}[r]^{(I)} \ar[d]^{(III)} & \mu_k^N \ar[d]^{(II)} \\
    P^{N*}_{t}\mu^N \ar@{=}[r]^{(IV)} & \mu^N
  }
$$
\end{enumerate}
The equality $(I)$ represents the stationarity  of $\mu_k^N$ under $P_{t,k}^N$, $(II)$ is the weak convergence of $\mu_k^N$ towards $\mu^N$. The equality $(IV)$ represents the (claimed) invariance of $\mu^N$ under $\phi^N$, that follows once we prove the convergence $(III)$ in the weak topology of $L^2.$ To this end, let $f: L^2\to\R$ be a Lipschitz function that is also bounded by $1.$ We have
\begin{align*}
\langle P_{t,k}^{N*}\mu_k^N,f\rangle-\langle P^{N*}_{t}\mu^N,f\rangle &=\langle\mu_k^N,P_{t,k}^Nf\rangle-\langle\mu^N,P^N_tf\rangle\\
&=\langle\mu_k^N,P_{t,k}^Nf-P^N_tf\rangle-\langle\mu^N-\mu_k^N,P^N_tf\rangle\\
&=\mathfrak{a}-\mathfrak{b}.
\end{align*}
Since $P^N_t$ is Feller, we have that $\mathfrak{b} \to 0$ as $k\to\infty.$ Now, using the boundedness property of $f$, we have
\begin{align*}
|\mathfrak{a}|\leq \int_{B_R(L^2)}|P^N_tf(u)-P_{t,k}^Nf(u)|\mu_k^N(du)+2\mu_k^N(L^2\backslash B_R(L^2))=:\mathfrak{a}_1+\mathfrak{a}_2.
\end{align*}
Here $C_f$ is the Lipschitz constant of $f$ and $u_k(t,\Pi^Nu)$ is the solution to \eqref{NLS-projN} starting from $\Pi^Nu.$ Now from \eqref{Est:M:N,al}, we have
\begin{align*}
\mathfrak{a}_2\leq \frac{C}{R^2}.
\end{align*}
For the term $\mathfrak{a}_1$, define the set 
\begin{align*}
S_r=\left\{\omega\in\Omega|\ \ \max\left(\left|\sqrt{\al_k}\sum_{|m|\leq N}a_m\int_0^t\langle u,e_m\rangle \,  d\beta_m\right|, \|z_k\|_{L^2}\right)\leq r\sqrt{\al_k}t\right\}\quad r>0,
\end{align*}
we have the following:
\begin{lem}\label{LemConvUnifInvis}
For any $R>0,$ any $r> 0,$
\begin{align*}
\sup_{u_0\in B_R(L^2)}\E(\|\phi^N_t\Pi^Nu_0- u_k(t,\Pi^Nu_0)\|_{L^2}\Bbb 1_{S_r})\to 0,\quad as\ k\to\infty.
\end{align*}
\end{lem}

Assuming Lemma \ref{LemConvUnifInvis}, we continue the computation. 
Now let us use the Lispschitz and boundedness properties of $f$,
\begin{align*}
\mathfrak{a}_1\leq C_f\int_{B_R}\E\|\phi^N_t\Pi^Nu_0-u_k(t,\Pi^Nu_0)\|_{L^2}\Bbb 1_{S_r}\mu_k^N(du_0)+2\int_{B_R}\E(\Bbb 1_{S_r^c})\mu_k^N(du_0) =:\mathfrak{a}_{1,1}+\mathfrak{a}_{1,2}.
\end{align*}
It follows from the Lemma \ref{LemConvUnifInvis} that for any fixed $R>0$ and any $r>0,$ $\lim_{k\to\infty}\mathfrak{a}_{1,1}=0.$

Using the  It\^o isometry and \eqref{Est:M:N,al}, we have
\begin{align*}
\E\left|\sqrt{\al_k}\sum_{|m|\leq N}a_m\int_0^t\langle u,e_m\rangle \, d\beta_m\right|^2=\al_k\sum_{|m|\leq N}|a_m|^2\int_0^t\langle u,e_m\rangle^2 \, dt\leq \al_k A_0\E\|u\|_{L^2}^2t\leq C\al_kt,
\end{align*}
where $C$ does not depend on $k.$ Also, from \eqref{Est:z},
\begin{align*}
\E\|z_k\|_{L^2}^2\leq C\al_k t,
\end{align*}
where $C$ is independent of $k.$ Therefore, using the Chebyshev inequality, we have
\begin{align*}
\E(\Bbb 1_{S^c_r})=\P\left\{\omega\ | \  \max\left(\left|\sqrt{\al_k}\sum_{|m|\leq N}a_m\int_0^t\langle u,e_m\rangle \,  d\beta_m\right|, \|z_k\|_{L^2}\right)\geq r\sqrt{\al_k}t \right\}\leq \frac{C\al_k}{r^2\al_k}=\frac{C}{r^2}.
\end{align*}
We pass to the limits $k\to\infty,\ R\to\infty,\ r\to\infty$ (respecting this order), we obtain $(III)$, and hence $(IV).$

Hence the proof of Proposition \ref{prop mu^N} is finished.
\end{proof}

Now we are left to prove Lemma \ref{LemConvUnifInvis}. 
\begin{proof}[Proof of Lemma \ref{LemConvUnifInvis}]
Set $w_k=u-v_k:=\phi^N_t\Pi^Nu_0-v_k(t,\Pi^Nu_0),$ where $v_k(t,\Pi^Nu_0)$ is the solution of \eqref{Equ:v}, with $\alpha =\al_k$ and that starts from $\Pi^Nu_0.$ We recall that $u_k=v_k+z_k,$ where $z_k$ solves the problem \eqref{Equ:z} with $\al=\al_k.$ Now, thanks to \eqref{Est:z}, we have that $\E\|z_k\|_{L^2}^2\to 0$ as $k\to\infty.$
Therefore, it suffices to show the following to complete the proof of the Lemma \ref{LemConvUnifInvis}.
\begin{align*}
\sup_{u_0\in B_R(L^2)}\E(\|w_k\|_{L^2}\Bbb 1_{S_r})\to 0,\quad as\ k\to\infty,
\end{align*}
Let us take the difference between the equations \eqref{NLS-projN} and \eqref{Equ:v}:
\begin{align}\label{eq wk}
\begin{aligned}
\dt w_k&= \i[\Delta w_k-\Pi^N(w_kf_{q}(u,v_k))]+\i \Pi^N(g_{q}(u,v_k,z_k)z_k)\\
& \quad -\al_k e^{\rho(\|v_k+z_k\|_{H^{\beta -}})}\left((-\Delta)^{\beta -1}(v_k+z_k)+\Pi^N|v_k+z_k|^{2q}(v_k+z_k)\right),
\end{aligned}
\end{align}
where $f_{q}$ and $g_{q}$ are polynomials of degree $2q$ in the given variables. We observe that 
\begin{align*}
|v_k+z_k|^{2q}(v_k+z_k)-|u|^{2q}u=|v_k|^{2q}v_k-|u|^{2q}u+z_kg_{q}(v_k,z_k)=w_kf_{q}(u,v_k)+z_kg_{q}(v_k,z_k).
\end{align*}
Taking the inner product with $w_k$ on both sides of \eqref{eq wk}, we see that
\begin{align*}
\dt\|w_k\|_{L^2}^2 &\leq 2\|w_k\|_{L^2}^2(1+\lambda_N^2+\|f_{q}(u,v_k)\|_{L^\infty_{t,x}})+2\|z_k\|_{L^2}^2\|g_{q}(v_k,z_k)\|^2_{L^\infty_{t,x}}\\
& \quad +\al_kC_0(N)\|w_k\|e^{\rho(c(N)(\|v_k+z_k\|_{L^2}))}\left[\|v_k\|_{L^2}+\|z_k\|_{L^2}+\|v_k\|_{L^2}^{2q+1}+\|z_k\|_{L^2}^{2q+1}\right]\\
&\leq C_1(N)\|w_k\|_{L^2}^2(1+\lambda_N^2+\|u\|_{L^\infty_{t}L^2_x}^{2q}+\|v_k\|^{2q}_{L^\infty_tL^2_x})+C_2(N)\|z_k\|_{L^2}^2\left(\|v_k\|^{4q+2}_{L^\infty_tL^2_x}+\|z_k\|^{4q+2}_{L^\infty_tL^2_x}\right)\\
& \quad +\alpha_kC_3(N)e^{2\rho(c(N)(\|v_k+z_k\|_{L^2}))}\left[\|v_k\|_{L^2}^2+\|z_k\|_{L^2}^2+\|v_k\|_{L^2}^{4q+2}+\|z_k\|_{L^2}^{4q+2}\right],
\end{align*} 
where $C_i(N)$ are some constant depending on $N$ (and in particular not on $\al_k$).
Using the Gronwall lemma and the fact that $w_k(0)=0,$  we arrive at 
\begin{align}
& \quad \|w_k(t)\|_{L^2}^2 \\
& \leq C_4(N)e^{C_1(N)\int_0^t(1+\lambda_N^2+\|u\|_{L^\infty_{t}L^2_x}^{2q}+\|v_k\|^{2q}_{L^\infty_tL^2_x}) \, d\tau}\left(\int_0^t\|z_k\|_{L^2}^2 \, d\tau+\al_kt\right)\left[1+e^{3\rho(c(N)(\|v_k\|_{L^\infty_t L^2_x}+\|z_k\|_{L^\infty_t L^2_x}))}\right]\label{ConvGENERALEst}
\end{align}
and the estimate \eqref{Est:z}, we have that, up to a subsequence, 
\begin{align*}
\lim_{k\to\infty}\sup_{t\in[0,T]}\|w_k\|_{L^2}=0,\ \ \P-almost\ surely.
\end{align*}
Now, writing the It\^o formula for $\|u_k\|_{L^2}^2$, we have
\begin{align*}
\|u_k\|_{L^2}^2+2\al_k\int_0^t\mathcal{M}(u_k) \, d\tau =\|\Pi^Nu_0\|_{L^2}^2+\al_k \frac{A_0^N}{2}t+2\sqrt{\al_k}\sum_{|m|\leq N}a_m\int_0^t\langle u_k,e_m\rangle \, d\beta_m.
\end{align*}
Therefore, recalling that $\al_k\leq 1,$ we have that, on the set $S_r,$
\begin{align*}
\|u_k\|_{L^2}^2\leq \|\Pi^Nu_0\|_{L^2}^2+C(r,N)t,
\end{align*}
where $C(r,N)$ does not depend on $k.$ Hence we see that, on $S_r$,
\begin{align*}
\|w_k\|_{L^2}\leq \|v_k\|_{L^2}+\|z_k\|_{L^2}\leq \|u_k\|_{L^2}+2\|z_k\|_{L^2}\leq \|u_0\|_{L^2}+3C(r,N)t.
\end{align*}
In particular, we have the following two estimates:
\begin{align}
\sup_{u_0\in B_R}\|w_k\|_{L^\infty_tL^2_x}\Bbb 1_{S_r}\leq R+3C(r,N)T,\label{ControlTronqueeStoch}
\end{align}
and
\begin{align}
\sup_{k\geq 1}\sup_{u_0\in B_R}\|w_k\|_{L^\infty_tL^2_x}\Bbb 1_{S_r}\leq R+3C(r,N)T.\label{CONVUnifin_k_bound}
\end{align}
Hence coming back to \eqref{ConvGENERALEst} and using the (deterministic) conservation $\|u(t)\|_{L^2}=\|P_N u_0\|_{L^2}$ and the estimate \eqref{ControlTronqueeStoch}, we obtain
\begin{align*}
\sup_{u_0\in B_R}\|w_k\|_{L^\infty_tL^2_x}^2\Bbb 1_{S_r}\leq A(R,N,r,T)\|z_k\|_{L^1_tL^2_x}.
\end{align*}
Therefore, using again the bound \eqref{Est:z}, we obtain the almost sure convergence $\|z_k\|_{L^2}\to 0$ (as $k\to\infty$, up to a subsequence), we obtain then the almost sure convergence
\begin{align*}
\lim_{k\to\infty}\sup_{u_0\in B_R}\|w_k\|_{L^\infty_tL^2_x}^2\Bbb 1_{S_r}=0.
\end{align*}
Now, taking into account the bound \eqref{CONVUnifin_k_bound}, we can then use the Lebesgue dominated convergence theorem to obtain
\begin{align*}
\E\sup_{u_0\in B_R}\|w_k\|_{L^\infty_tL^2_x}\Bbb 1_{S_r}\to 0,\quad as\ k\to\infty.
\end{align*}

Now, for $u_0\in B_R,$ we have
\begin{align*}
\|w_k(t,\Pi^Nu_0)\|_{L^2}\Bbb 1_{S_r}\leq \sup_{u_0\in B_R}\|w_k(t,\Pi^Nu_0)\|_{L^2}\Bbb 1_{S_r},
\end{align*}
then
\begin{align*}
\E\|w_k(t,\Pi^Nu_0)\|_{L^2}\Bbb 1_{S_r}\leq \E \sup_{u_0\in B_R}\|w_k(t,\Pi^Nu_0)\|_{L^2}\Bbb 1_{S_r},
\end{align*}
and finally,
\begin{align*}
\sup_{u_0\in B_R}\E\|w_k(t,\Pi^Nu_0)\|_{L^2}\Bbb 1_{S_r}\leq \E \sup_{u_0\in B_R}\|w_k(t,\Pi^Nu_0)\|_{L^2}\Bbb 1_{S_r}.
\end{align*}
The proof of Lemma \ref{LemConvUnifInvis} is finished. 
\end{proof}
Now we finished discussion on the Galerkin approximation, and in the next section we will take the dimension $N \to \infty$.

\section{Almost sure GWP: All algebraic powers with subcritical regularities}\label{Sect:Alg:pow}
This section is devoted to the proof of the content of Theorem \ref{main thm1}. In the previous section we obtain invariant measures and strong bounds for the dynamics \eqref{NLS-projN}. This is going to be combined with the local well-posedness result obtained in Theorem \ref{thm LWP} and a globalization argument of Bourgain type (see the seminal work \cite{bourg94}). Our construction here has some important differences coming from the fact that our measures are not of Gibbs type, see also \cite{syNLS7} and \cite{lat}. Here $q \in \Z^+$ so that Theorem \ref{thm LWP} holds. 

Let $\xi:\R_+\to\R_+$ be a one-to-one increasing concave function, let $\xi^{-1}$ be its inverse. We take $\rho=3\xi^{-1}.$ We claim that for some constant $C(\xi)>0$ not depending on $N$, the following bound holds:
\begin{align}
\int_{L^2} e^{3\xi^{-1}(\norm{v}_{H^{\beta -}})} \mu^N (dv) \leq C(\xi). \label{Est:Exp:N}
\end{align}
Indeed, using \eqref{Est_bound:on:E} and \eqref{Est:E:N}, we have
\begin{align*}
\int_{L^2} e^{3\xi^{-1}(\norm{v}_{H^{\beta -}})} \mu^N (dv) &\leq \int_{\|v\|_{H^{\beta-}}\leq 1} e^{3\xi^{-1}(\norm{v}_{H^{\beta -}})} \mu^N (dv) +\int_{\|v\|_{H^{\beta-}}\geq 1} e^{3\xi^{-1}(\norm{v}_{H^{\beta -}})}\|v\|_{H^{\beta-}}^2 \mu^N (dv) \\
&\leq e^{3\xi^{-1}(1)}+\int_{L^2} \mathcal{E}(v) \mu^N (dv)\leq C_1(\xi)+C=:C(\xi),
\end{align*}
where $C$ is the constant $C_1$ in \eqref{Est:E:N}.

\subsection{The GWP result.}
In what follows the constant $\tilde{C}$ in \eqref{Local:increase} will be taken to be $2$ for simplicity. 

\begin{prop}\label{prop growth}
Let $s_c< s \leq \beta -$, $N , i \in \N^+$, there exists a set $\Sigma_{s,i}^N$ such that
\begin{enumerate}
\item
$\mu^N (E_N \setminus \Sigma_{s,i}^N) \leq C e^{-2i}$ for some universal constants $C> 0$ and $\delta_0 > 0$;

\item
For any $u_0 \in \Sigma_{s,i}^N$
\begin{align}
\norm{\phi_t^N u_0}_{H^s} \leq 2 \xi(1+i+\ln(1+|t|))\label{Est:Bourg-type}
\end{align}
for all $t \in \R$.
\end{enumerate}
\end{prop}

\begin{proof}[Proof of Proposition \ref{prop growth}]
\begin{enumerate}
\item
Without loss of generality, we can consider $t\geq 0$.
Set 
\begin{align*}
B_{s,i,j}^{N} := \{ u \in E_N \big| \norm{u}_{H^s} \leq \xi(i+j) \}
\end{align*}
Note that
\begin{align*}
\phi_t^N B_{s,i,j}^N \subset \{ u \in E_N \big|   \norm{u}_{H^s} \leq 2 \xi(i+j) \}
\end{align*}
for $t \leq T_R$, where $T_R$ is the local time of existence on $B_{s,i,j}^N$ (see Theorem \ref{thm LWP}). Recall that $T_R \gtrsim R^{-2q\beta} = [\xi(i+j)]^{-2q\beta}$.

Set 
\begin{align*}
\Sigma_{s,i,j}^N : = \cap_{k=0}^{ \square{\frac{e^j}{T_R}}} \phi_{-kT_R}^N (B_{s, i,j}^N).
\end{align*}
We obtain, using permutation between inverse mapping and complementary operation,
\begin{align*}
\mu^N (E_N \setminus \Sigma_{s,i,j}^N) = \mu^N (\cup_{k=0}^{ \square{\frac{e^j}{T_R}}} \phi_{-kT_R}^N (E_N \setminus B_{s, i,j}^N)) \leq \sum_{k=0}^{ \square{\frac{e^j}{T_R}}} \mu ( \phi_{-kT_R}^N (E_N \setminus B_{s, i,j}^N). 
\end{align*}
We use the invariance of $\mu^N$ under $\phi_t^N$ and the estimate \eqref{Est:Exp:N} and Chebyshev's inequality, we obtain
\begin{align*}
\mu^N (E_N \setminus \Sigma_{s,i,j}^N) & \leq \left(\square{\frac{e^j}{T_R}}+1\right) \mu^N (E_N \setminus B_{s, i,j}^N)  \leq C  \square{\frac{e^j}{T_R}} \mu^N \left(\norm{u}_{H^s} \geq \xi(i+j)\right) \\
& \leq   2 \square{\frac{e^j}{T_R}} e^{-3\xi^{-1}(\xi(i+j))} C(\xi)  \leq 2C(\xi) e^j(i+j)^{-2q\beta} e^{-3(i+j)}\\
& \lleq e^{-2i}e^{-j}.
\end{align*}
Let us define
\begin{align*}
\Sigma_{s,i}^N = \cap_{j \geq 1} \Sigma_{s,i,j}^N,
\end{align*}
 then we obtain
\begin{align*}
\mu^N (E_N \setminus \Sigma_{s,i}^N) \leq C \sum_{j \geq 0} e^{-2i} e^{-j} \leq C e^{-2i}. 
\end{align*}

\item
Now, since $u_0 \in \Sigma_{s,i,j}^N$, then we get
\begin{align}
\norm{\phi_t^N u_0}_{H^s} \leq 2\xi(i+j) \quad \text{for all $t \leq e^j$}.\label{Est:Sigma}
\end{align}
 Indeed we write $t = k T_R + \tau$ where $k$ is an integer in $[0, \frac{e^j}{T_R}]$ and $\tau \in [0, T_R].$
 We have
\begin{align*}
\phi_t^N u_0 = \phi_{\tau+kT_R}^N u_0 = \phi_{\tau}^N (\phi_{kT_R}^N u_0)
\end{align*}
But since $u_0 \in \phi_{-kT_R}^N (B_{s,i,j}^N)$, then $u_0 = \phi_{-kT_R}^N w$, for some $w \in B_{s,i,j}^N$, then
$\phi_t^N u_0 = \phi_{\tau}^N w$ and we obtain the claim by invoking the local existence time bound (Theorem \ref{thm LWP}).

Now, remark that for any $t>0$, there exists $j$ such that
\begin{align*}
e^{j-1} \leq 1 +t \leq e^{j},
\end{align*}
hence
\begin{align*}
j\leq1+\ln(1+t),
\end{align*}
and using \eqref{Est:Sigma}, we arrive at
\begin{align*}
\norm{\phi_t^N u_0}_{H^s} \leq 2\xi(1+i+\ln(1+t)).
\end{align*}
This finishes the proof of Proposition \ref{prop growth}.
\end{enumerate}
\end{proof}

Now, for any $s<\beta$, we introduce the sets
\begin{align*}
\Sigma_{s,i} : = \left\{ u \in H^s\ \big|\ \exists (u^{N_k}) , u^{N_k} \to u, \text{ where } u^{N_k} \in \Sigma_{s,i}^{N_k}  \right\},
\end{align*}
we define the following conditional measures
\begin{align*}
\mu_{s,i}^N(\Gamma) = \mu^N (\Gamma \big| \Sigma_{s,i}^N) = \frac{\mu^N (\Gamma \cap \Sigma_{s,i}^N)}{\mu^N (\Sigma_{s,i}^N)}.
\end{align*}

Now we are ready to take the dimension $N \to \infty$.
\begin{prop}\label{prop mu}
We have the following
\begin{enumerate}
\item
As $N \to \infty$, there exists $N_k \to \infty$ and $\mu \in \mathfrak{p}(H^{s})$, such that $\mu^{N_k} \rightharpoonup \mu$ on $H^{s}$;

\item

For any $i,s$, as $N \to \infty$, there exists $N_k \to \infty$ and $\mu_{s,i} \in \mathfrak{p}(H^{s})$, such that $\mu_{s,i}^{N_k} \rightharpoonup \mu_{s,i}$ on $H^{s}$.
\item We have the estimates
\begin{align}
\int_{L^2} \mathcal{M}(v) \mu (dv) =\frac{ A_0}{2}; \label{Est:M}\\
\int_{L^2}\|u\|_{H^\beta}^2\mu(du) \leq C.\label{Est:E}
\end{align}
Consequently,  from \eqref{Est:E}, we obtain that $\mu$ is concentrated on $H^\beta$.
\end{enumerate}
\end{prop}

\begin{proof}[Proof of Proposition \ref{prop mu}]
\begin{enumerate}
\item
First, combining \eqref{Est:E:N} and \eqref{Est_bound:on:E} we derive
\begin{align}
\int_{L^2} \norm{u}_{H^\beta}^{2} \mu^N (dv) \leq C.\label{Est:H2:N}
\end{align}
Since the constant $ C$ does not depend on $N$, we have Prokhorov theorem that $(\mu^N)_{N \geq 1}$ is tight on $H^{\beta-}$. Therefore, there exists $\mu$ such that $\mu^{N_k}  \rightharpoonup \mu$ on $H^{\beta -}$ for some subsequence $N_k \to \infty$. 

\item
Second, we write
\begin{align*}
\mu_{s,i}^N (dv) \leq \frac{\mu (dv)}{\mu^N (\Sigma_{s,i}^N)} \leq \frac{1}{1- Ce^{-2 i}} \mu^N (dv).
\end{align*}
Hence
\begin{align*}
\int_{L^2} \norm{u}_{H^{\beta}}^{2} \mu_{s,i}^N (dv) \leq \frac{1}{1-Ce^{-2i}} \int_{L^2} \norm{u}_{H^{\beta}}^{2} \mu^N (dv) \leq \frac{A_0}{2 (1- C e^{- 2 i})}.
\end{align*}
Then we proceed as above.
\item
Let us prove the estimate \eqref{Est:E} first. Let $M, R>0$  and $\chi_R$ be a smooth cut-off function on $[0,R]$ defined as in \eqref{eq chi}. We use \eqref{Est:H2:N} and write
\begin{align}
\int_{L^2}\|\Pi^Mu\|_{H^2}^2\chi_R(\|u\|_{L^2})\mu^N(du)\leq C.
\end{align}
Recall that $C$ does not depend on $N$. Thanks to the cut-off on the frequency and on the $L^2-$norm,   we have that $u\mapsto\|\Pi^Mu\|_{H^2}^2\chi_R(\|u\|_{L^2})$ is bounded continuous on $L^2$. Recall that the weak convergence of $(\mu^N)$ to $\mu$ holds in particular on $L^2$. Therefore we pass to the limit $N\to\infty$ and obtain
\begin{align}
\int_{L^2}\|\Pi^Mu\|_{H^2}^2\chi_R(\|u\|_{L^2})\mu(du)\leq C.
\end{align}
It remains to pass to the limits $R\to\infty$ and $M\to\infty$ with the use of Fatou's lemma to obtain \eqref{Est:E}.

The proof of \eqref{Est:M} is more difficult.
From \eqref{Est_bound:on:E} and \eqref{Est:E:N}, we have that $\E\langle|u^N|^{2q},|\nabla u^N|^2\rangle \leq C$, where $C$ does not depend on $N.$ Hence $\E\||u^N|^{q}u^N\|_{\dot{H}^1}^2\leq C$. Since $\dot{H}^1$ is compactly embedded in $L^{6-}$, we have that the random variables $|u^N|^{q}u^N$ converges weakly in $L^{6-}$ (up to a subsequence that we write without changing the index). Using the Skorokhod theorem (see Theorem \ref{Thm:Skorokhod}) this convergence is an almost sure one up to a modification of the probability space (that we denote by $(\Omega, \P)$ again, and we use the same names for the involved random variables as well). This almost sure convergence being strong, i.e. in the $L^{6-}$ norm, we have by writing $u$ the limiting random variable that
\begin{align}
\left|\||u^N|^{q}u^N\|_{L^{2}}-\||u|^{q}u\|_{L^{2}}\right|\leq \||u^N|^{q}u^N-|u|^{q}u\|_{L^{2}}\leq C \||u^N|^{q}u^N-|u|^{q}u\|_{L^{6-}}\to 0\quad \text{as $N\to\infty$}.\label{Conv:u^q}
\end{align} 
Hence we arrive at 
\begin{align}
\lim_{N\to\infty}\|u^N\|_{L^{2q+2}}=\|u\|_{L^{2q+2}}.\label{Conv:L4q+2norm}
\end{align}
Now let us write, from the identity \eqref{Est:E:N}, 
\begin{align*}
\frac{A_0^N}{2}\leq \E [\mathcal{M}(u^N)1_{\{E(u^N)+\|u^N\|_{H^{\beta-}}^2\leq R\}}]+\E [\mathcal{M}(u^N)1_{\{E(u^N)+\|u^N\|_{H^{\beta-}}^2\geq R\}}]=: \mathfrak{a}+\mathfrak{b}.
\end{align*}
We use the Lebesgue dominated convergence theorem, \eqref{Conv:L4q+2norm} and the convergence in $H^{\beta-}$ to see that 
\begin{align*}
\lim_{N\to\infty} \mathfrak{a}=\E [\mathcal{M}(u)1_{\{E(u^N)+\|u^N\|_{H^{\beta-}}^2\leq R\}}].
\end{align*}
Using \eqref{Mcal:leq:Ecal:Subcrit} and \eqref{Est:EEcal:N}, we have
\begin{align*}
\mathfrak{b} &=E \left[\mathcal{M}(u^N)\frac{E(u^N)+\|u^N\|_{H^{\beta-}}^2}{E(u^N)+\|u^N\|_{H^{\beta-}}^2}1_{\{E(u^N)+\|u^N\|_{H^{\beta-}}^2\geq R\}}\right] \leq \frac{1}{R}\E [\mathcal{M}(u^N)(E(u^N)+\|u^N\|_{H^{\beta-}}^2)1_{\{E(u^N)+\|u^N\|_{H^{\beta-}}^2\geq R\}}]\\
& \leq \frac{2}{R}E [\mathcal{E}(u^N)E(u^N)]\leq \frac{2C_1}{R}.
\end{align*}
Therefore
\begin{align*}
\frac{A_0}{2}\leq \E [\mathcal{M}(u)1_{\{E(u)+\|u^N\|_{H^{\beta-}}^2\leq R\}}] + \frac{2C_1}{R}.
\end{align*}
Letting $R\to\infty$, and then observing the fact that $\E [\mathcal{M}(u)]\leq \frac{A_0}{2}$, we arrive at \eqref{Est:M}.
\end{enumerate}
Then the proof of Proposition \ref{prop mu} is complete.
\end{proof}

Set 
\begin{align*}
\Sigma_s : = \cup_{i \geq 1} \overline{\Sigma_{s,i}}.
\end{align*}
We see that
\begin{prop}\label{prop mu bdd}
For any $s_c<s \leq \beta -$, we have that
\begin{enumerate}
\item
$\mu (\overline{\Sigma_{s,i}}) \geq 1 - C e^{-2 i }$;

\item
$\mu(\Sigma_s) =1$.
\end{enumerate}
\end{prop}

\begin{proof}[Proof of Proposition \ref{prop mu bdd}]
\begin{enumerate}
\item
For the first point, let us observe that
\begin{align*}
\Sigma_{s,i}^N \subset \Sigma_{s,i}
\end{align*}
because for any $u \in \Sigma_{s,i}^N $ the constant sequence $u_N = u$ converges to $u$ and $u \in \Sigma_{s,i}$.

By Portmanteau theorem (see Theorem \ref{Thm:Portm})
\begin{align*}
\mu(\overline{\Sigma_{s,i}} ) \geq \lim_{N_k \to \infty} \mu^{N_k} (\overline{\Sigma_{s,i}}) \geq \lim_{N_k \to \infty} \mu^{N_k} (\overline{\Sigma^{N_k}_{s,i}}) \geq \lim_{N_k \to \infty} \mu^{N_k} (\Sigma_{s,i}^{N_k}) \geq \lim_{N_k \to \infty} (1 - C e^{-2 i})  = 1 - C e^{-2 i}.
\end{align*}

\item
Next, for the second statement, we first claim that $(\overline{\Sigma_{s,i}})_{i \geq 1}$ is non-decreasing. Indeed, let us remark that $(\Sigma_{s,i,j}^N)_{j \geq 1}$ is non-decreasing because $\Sigma_{s,i,j}^N = \cup_{k=0}^{\square{\frac{e^j}{T_R}}} \phi_{-kT_R}^N (B_{s,i,j}^N)$ and $(\Sigma_{s,i,j}^N)_{j \geq 1}$ is non-decreasing for any fixed $i$. Now  the following implications gives the claim:
\begin{align*}
\text{$(\Sigma_{s,i}^N)_{i \geq 1}$ is non-decreasing} & \implies \text{$(\Sigma_{s,i})_{i \geq 1}$ is non-decreasing} \implies \text{$(\overline{\Sigma_{s,i}})_{i \geq 1}$ is non-decreasing}.
\end{align*}
In particular, 
\begin{align*}
\mu (\cup_{i \geq 1} \overline{\Sigma_{s,i}}) = \lim_{i \to \infty} \mu (\overline{\Sigma_{s,i}}).
\end{align*}
Therefore
\begin{align*}
\mu (\cup_{i \geq 1} \overline{\Sigma_{s,i}}) = \lim_{i \to \infty} \mu (\overline{\Sigma_{s,i}}) \geq \lim_{i \to \infty} (1 - C e^{-2 i})  =1.
\end{align*}
Since $\mu$ is a probability measure, we obtain
\begin{align*}
1 \geq \mu(\Sigma_s ) \geq 1,
\end{align*}
which finishes the proof of Proposition \ref{prop mu bdd}.
\end{enumerate}
\end{proof}

\begin{prop}\label{prop phi control}
For any $s_c<s\leq \beta -$, for every $u_0 \in \Sigma_s \cap \supp (\mu) $ then the solution $\phi_t u_0$ of \eqref{NLS} given in Theorem \ref{thm LWP} is global and we have  for $t \in \R$
\begin{align}
\norm{\phi_t u_0}_{H^s} \leq 2 \xi(1+i+\ln(1+|t|))\label{Control:phi}
\end{align}
for some $i$ depending on $\|u_0\|_{H^s}$.
\end{prop}

\begin{proof}[Proof of Proposition \ref{prop phi control}]
For $u_0 \in \Sigma_s \cap \supp (\mu)$\footnote{Recall that $\mu(\Sigma_s)=1$.} we have
\begin{enumerate}
\item
$u_0 \in H^{\beta}$ since $\supp (\mu)\subset H^\beta$;
\item
there exists $i \geq 1$ such that $u_0 \in \overline{\Sigma_{s,i}}$.
\end{enumerate}
First, assume $u_0 \in \Sigma_{s,i}$. By construction, there exists $N_k \to \infty$ such that 
\begin{align*}
u_0 = \lim_{N_k \to \infty} u^{N_k}, \quad u^{N_k} \in \Sigma_{s,i}^{N_k}.
\end{align*}
For these $u^{N_k}$ we have by \eqref{Est:Bourg-type} that for all $t$
\begin{align*}
\norm{\phi_t^N u^{N_k}}_{H^s} \leq 2 \xi(1+i+\ln(1+|t|)) .
\end{align*}
Then at $t=0$
\begin{align*}
\norm{u^{N_k}}_{H^s} \leq 2 \xi(1+i).
\end{align*}
Passing to the limit $N^k\to \infty$, we obtain
\begin{align*}
\norm{u_0}_{H^s} \leq 2 \xi(1+i).
\end{align*}
Now let $T > 0$, $b = 2 \xi(1+i+\ln(1+|T|))$ and $R = b+1$, then $u^{N_k}, u_0 \in B_R (H^s)$. Let $T_R$ be the local time of existence of $B_R$. $\phi_t u_0$ and $\phi_t^{N_k} u_0^{N_k}$ exist for $\abs{t} \leq T_R$. As a consequence of Proposition \ref{Prop:CvgLocSol}, for $s_c<s \leq \beta -$, we have 
\begin{align*}
\norm{\phi_t u_0 - \phi_t^{N_k} u_0^{N_k}}_{L_t^{\infty} H_x^s} \to 0  \text{ as } k\to\infty.
\end{align*}
Next, let us write
\begin{align*}
\norm{\phi_t u_0}_{H^s} & = \norm{\phi_t u_0 - \phi_t^{N_k} u^{N_k} + \phi_t^{N_k} u^{N_k}}_{H^s}  \leq  \norm{\phi_t u_0 - \phi_t^{N_k} u^{N_k} }_{H^s} + \norm{ \phi_t^{N_k} u^{N_k}}_{H^s} \leq o(1) + b.
\end{align*}
Passing to the limit $N_k \to \infty$
we obtain
\begin{align*}
\norm{\phi_t u_0}_{H^s} \leq b \leq R,
\end{align*}
since $\phi_t$ stays in $B_R$ after time $T_R$. We can then iterate on intervals $[n T_R, (n+1) T_R]$ until `consuming' $T$. But $T$ is arbitrary, the we obtain global existence of $\phi_tu_0.$

In order to obtain the control \eqref{Control:phi}, let us use the estimate \eqref{Est:Bourg-type}:
$\norm{\phi_t^{N_k} u^{N_k}}_{H^s}  \leq 2\xi(1+i+\ln(1 + \abs{t}))$ for any $t$. Since $\phi_t u_0$ is well defined and $\phi_t^{N_k} u^{N_k} \to \phi_t u_0$ in $H^s$, we obtain \eqref{Control:phi} for any $u_0\in\Sigma_{s,i}.$

 Now, let $u_0 \in \partial  \Sigma_{s,i}$, there exists $u_0^k \in \Sigma_{s,i}$, $u_0^k \to u_0$ in $H^s$. We know that there is an $R$, such that $u_0,\ u_0^k \in B_R (H^s)$. Let $T_R$ to be the time existence for $B_R$. We have already showed that for $\abs{t} \leq T_R$
\begin{align*}
\norm{\phi_t u_0^k}_{H^s} \leq 2\xi(1+i+\ln(1 + \abs{t})) .
\end{align*} 
Passing to the limit $k \to \infty$, we arrive at
\begin{align*}
\norm{\phi_t u_0}_{H^s} \leq 2\xi(1+i+\ln(1 + \abs{t})).
\end{align*}
Then we use a standard iteration to finish the proof of Proposition \ref{prop phi control}.
\end{proof}
\begin{rmq}
From the proof of Proposition \ref{prop phi control} above, we have that for any $i\geq 1,$ any $u_0\in \Sigma_r^i,$ any $t\in \R$,
\begin{align}\label{Conv_phit-phitN}
\lim_{N\to\infty}\|\phi^tu_0-\phi_N^{t}u_{0,N}\|_{H^r}=0,
\end{align}
where $(u_{0,N})$ is a sequence in $\Sigma_{r,i}^N$ that converges to $u_0$ in $H^r.$
\end{rmq}

Now let us finish the construction of the statistical ensemble. Let $l=(l_k)$ be a sequence such that $\frac{1}{2} < l_k \nearrow \beta -$, and set
\begin{align*}
\Sigma = \cap_{k \geq 0} (\Sigma_{l_k} \cap \supp (\mu)).
\end{align*}

\subsection{Invariance}
We show the invariance of the statistical ensemble $\Sigma$ and the measure $\mu$.
\begin{rmq}
We have that $\phi_t$ is global on $\Sigma$ as a consequence of our analysis above.
\end{rmq}

\begin{prop}\label{Prop:Inv:Sig}
We have the following
\begin{enumerate}
\item
$\mu(\Sigma) =1$.
\item
For all $t$, $\phi_t \Sigma = \Sigma$.
\end{enumerate}
\end{prop}

Before proving Proposition \ref{Prop:Inv:Sig}, let us first prepare the following lemma.
\begin{lem}\label{Lem:Inv}
Let $s_c < s \leq \beta -$, for any $s_1 \in (s_c , s)$ and any $t$, there exists $i_1$, such that for any $i$, if $u_0 \in \Sigma_{s,i}^N$, then $\phi_t^N u_0 \in \Sigma_{s_1 , i+i_1}^N$.
\end{lem}
\begin{proof}[Proof of Lemma \ref{Lem:Inv}]
Assume $t \geq 0$. If $u_0 \in \Sigma_{s, i}^N$, then for any $j \geq 1$ and any $t_1 \leq e^j$
\begin{align*}
\norm{\phi_{t_1}^N u_0}_{H^s} \leq 2\xi(i+j).
\end{align*}
Now, take $i_1 (t)$ such that $e^j + t \leq e^{j+ i_1}$ for any $j$ and any $t_1+t \leq e^{j+i_1}$ 
\begin{align*}
\norm{\phi_{t_1 + t}^N u_0}_{H^s} \leq 2\xi(i + j +i_1) .
\end{align*}
Now if $t_1 + t \leq e^{j +i_1}$, then $t_1 \leq e^{j + i_1} -t$. On the other hand, by definition of $i_1$, $e^j + t \leq e^{j+ i_1}$, therefore $e^{j+ i_1} - t \geq e^j$. In particular, for any $t_1 \leq e^j$
\begin{align*}
\norm{\phi_{t_1 + t}^N u_0}_{H^s} \leq 2 \xi(i +j + i_1) .
\end{align*}

For $u_0 \in \Sigma_{s,i}^N$ we use \eqref{Est:Bourg-type} to obtain
\begin{align*}
\norm{\phi_t^N u_0}_{H^s} & \leq 2\xi(1+i+\ln (1 + \abs{t}))\\
\norm{u_0}_{H^s} & \leq 2\xi(1+i).
\end{align*}
In particular
\begin{align*}
\norm{u_0}_{L^2} \leq 2\xi(1+i).
\end{align*}
By the $L^2$ conservation law, we obtain
\begin{align*}
\norm{\phi_{t+ t_1} ^N u_0}_{L^2} \leq 2\xi(1+i).
\end{align*}
For $s_1 \in (s_c , s)$, there exists $\theta\in (0,1)$ such that
\begin{align*}
\norm{\phi^N_{t+ t_1} u_0}_{H^{s_1}} \leq \norm{\phi^N_{t+ t_1} u_0}_{L^2}^{1-\theta} \norm{\phi^N_{t+ t_1} u_0}_{H^s}^{\theta} \leq 2^{1-\theta} [\xi(1+i)]^{(1-\theta)} 2^{\theta}[\xi(i + j + i_1)]^{\theta}.
\end{align*}
Since $\xi$ is increasing, let us take $i_1$ is big enough and then  obtain $\norm{\phi^N_{t+ t_1} u_0}_{H^{s_1}} \leq \xi(i+j+i_1)$ for all $t_1 \leq e^j.$
Thus $\phi^N_{t+ t_1} u_0 \in B_{s, i+i_1,j}$ for all $t_1 \leq e^j$, for all $j\geq 1$. 
Therefore $\phi^N_t u_0 \in \Sigma_{s , i+i_1}^N$. Hence we prove Lemma \ref{Lem:Inv}.
\end{proof}

Now we are ready to prove Proposition \ref{Prop:Inv:Sig}.
\begin{proof}[Proof of Proposition \ref{Prop:Inv:Sig}]
\begin{enumerate}
\item 
Since any $\Sigma_r$ is of full $\mu-$measure and the intersection is countable, we obtain the first statement.

\item
To prove the second statement, let us take $u_0\in \Sigma$, then $u_0$ belong to each $\Sigma_r$, $r\in l.$ First, consider $u_0\in \Sigma_{i_r}$. There is $i\geq 1$ such that $u_0\in \Sigma_{i_r}$, therefore $u_0$ is the limit of a sequence $(u_{0}^N)$ such that $u_{0}^N\in \Sigma_{i,r}^N$ for every $N$. Now thanks to Lemma \ref{Lem:Inv}, there is $i_1:=i_1(t)$ such that $\phi^N_t(u_{0}^N)\in \Sigma_{i+i_1,r_1}^{N}$.
Using the convergence \eqref{Conv_phit-phitN}, we see that $\phi^t(u_0)\in \Sigma_{i+i_1,r_1}$. Now if $u_0\in \partial\Sigma_{i,r}$, there is $(u_0^k)_k\subset\Sigma_{i,r}$ that converges to $u_0$ in $H^r$.  Since we showed that $\phi_t\Sigma_{i,r}\subset\Sigma_{i+i_1,r_1}$ and $\phi_t(\cdot) $ is continuous, we see that
\begin{align*}
\phi_t(u_0)=\lim_{k}\phi_t(u_0^k)\in \overline{\Sigma_{i+i_1,r_1}}.
\end{align*} 
We conclude that 
\begin{align*}
\phi_t\overline{\Sigma_{i,r}}\subset\overline{\Sigma_{i+i_1,r_1}}\subset \Sigma_{r_1}.
\end{align*}
It follows that $\phi_t\Sigma\subset\Sigma.$

Now, let $u$ be in $\Sigma$. Since $\phi_t$ is well-defined on $\Sigma$ we can set  $u_0=\phi_{-t}u$, we then have $u=\phi_tu_0$, hence $\Sigma\subset\phi_{t}\Sigma.$ 
 \end{enumerate}
 That finishes the proof of Proposition \ref{Prop:Inv:Sig}.
\end{proof}

Now we are left to verify the following property.
\begin{thm}\label{thm mu invariant}
The measure $\mu$ is invariant under $\phi_t$.
\end{thm}
\begin{proof}[Proof of Theorem \ref{thm mu invariant}]
We claim that it also suffices to show the invariance only on a fixed interval $[-\tau,\tau],$ where $\tau>0$ can be as small as we want. Indeed for $\tau\leq t\leq 2\tau$, one has 
\begin{align*}
\mu(\phi_{-t}S)=\mu(\phi_{-\tau}\phi_{t-\tau}S)=\mu(\phi_{t-\tau}S)=\mu(S)
\end{align*}
(using that $0\leq t-\tau\leq\tau$), and for greater values of $t$ we can iterate. A same argument works for negative values of $t$.

Now, it was shown in Proposition \ref{Prop:CvgLocSol} that for small times we have that 
\begin{align*}
\phi_tu_0=\lim_{k\to\infty}\phi_t^{N_k}u_0^{N_k} \quad \text{in $H^{\beta-}$}.
\end{align*}
Looking at $u_0$ and $u_0^{N_k}$ as random variables in $H^{\beta-}$, we have that the convergence above is then of course an almost sure one. Therefore 
\begin{align*}
\lim_{k\to\infty}\mathcal{L}(\phi_t^{N_k}u_0^{N_k})=\mathcal{L}(\phi_tu_0),
\end{align*}
where $\mathcal{L}(X)$ stands for the law of $X$. Now,  since $\mu^{N_k}=\mathcal{L}(\phi_t^{N_k}u_0^{N_k})$ and $\mu^{N_k}\rightharpoonup\mu$, we obtain that $\mu =\mathcal{L}(\phi_tu_0)$. Since $t$ is arbitrary (in some interval $[-\tau,\tau]$), we obtain the invariance. 

Hence we finish the proof of Theorem \ref{thm mu invariant}.
\end{proof}

\section{Almost sure GWP: All supercritical nonlinearities with critical and supercritical regularities}\label{sec supercritical}
In this section we consider all supercritical powers of the nonlinearity (without assuming $q \in \N^+$) and consider the data in critical and supercritical ranges of regularity. Notice that no deterministic GWP is known in the context of this range of powers even for smooth data. Recall the critical power for \eqref{NLS} is $s_c=\frac{3}{2}-\frac{1}{q}$ in \eqref{eq sc}. 

The rest of this section is dedicated to the proof of Theorem \ref{main thm2}, which is based on a compactness argument (see e.g. \cite{KS04,KS12}).

\subsection{Galerkin approximation and fluctuation-dissipation solutions}
We first  take in \eqref{SNLS} a dissipation defined as follows:
\begin{align}
\mathcal{L}(u)=\left[(1-\Delta)^{\beta-1}u+2e^{2\|\Pi^N|u|^{2q}u\|_{L^2}^2}\Pi^N|u|^{2q}u+\Pi^Ne^{|u|^2}u\right],\label{Newdiss}
\end{align}
where $\beta\in (1,s_c]$.
The fluctuation-dissipation equation considered here is hence
\begin{align}
du =\i(\Delta u-\Pi^Nu^{2q}u)-\al\mathcal{L}(u)+\sqrt{\al} \,dW^N. \label{Eq:Newdiss}
\end{align}
where $W^N$ is the Brownian motion defined in \eqref{Noise} with the assumption \eqref{Noise:Cond}.

 We still denote by $\mathcal{M}(u)$ and $\mathcal{E}(u)$ the dissipation rates of the mass and the energy under \eqref{Eq:Newdiss}. Let us state few properties satisfied by $\mathcal{M}(u)$ and $\mathcal{E}(u)$. First, we have the formula
\begin{align}
\mathcal{M}(u)=\|u\|_{H^{\beta-1}}^2+2e^{2\|\Pi^N|u|^{2q}u\|_{L^2}^2}\|u\|_{L^{2q+2}}^{2q+2}+\sum_{p\in\N}\frac{\|u\|_{L^{2p+2}}^{2p+2}}{p!}.
\end{align}
Second, the following lower bound holds for $\mathcal{E}(u)$
\begin{align*}
\mathcal{E}(u) &\geq\|u\|_{H^\beta}^2+2e^{2\|\Pi^N|u|^{2q}u\|_{L^2}^2}\langle |\nabla u|^2,|u|^{2q}\rangle+\sum_{p\in\N}\frac{\langle |\nabla u|^2,|u|^{2p}\rangle}{p!}\\
& \quad -\|u\|_{H^{2\beta-2}}\|\Pi^N|u|^{2q}u\|_{L^2}+2e^{2\|\Pi^N|u|^{2q}u\|_{L^2}^2}\|\Pi^N|u|^{2q}u\|_{L^2}^2-\sum_{p\in\N}\frac{\|\Pi^N|u|^{2q}u\|_{L^{2}}\|u\|_{L^{4p+2}}^{2p+1}}{p!}.
\end{align*}
In order to obtain a non-negative lower bound, we write the following estimates:
\begin{align*}
\|u\|_{H^{2\beta-2}}\|\Pi^N|u|^{2q}u\|_{L^2} &\leq \frac{1}{2}\|u\|_{H^{\beta}}^2+\frac{1}{2}\|\Pi^N|u|^{2q}u\|_{L^2}^2 \leq \frac{1}{2}\|u\|_{H^{\beta}}^2+\frac{1}{2}e^{2\|\Pi^N|u|^{2q}u\|_{L^2}^2}\|\Pi^N|u|^{2q}u\|_{L^2}^2,
\end{align*}
and
\begin{align*}
\sum_{p\in\N}\frac{\|\Pi^N|u|^{2q}u\|_{L^2}\|u\|_{L^{4p+2}}^{2p+1}}{p!} &\leq \sum_{p\in\N}\frac{\frac{2^{2p+2}}{2p+2}\|\Pi^N|u|^{2q}u\|_{L^2}^{2p+2}+\frac{1}{2^{\frac{2p+2}{2p+1}}}\frac{2p+1}{2p+2}\|u\|_{L^{4p+2}}^{2p+2}}{p!}\\
&\leq \|\Pi^N|u|^{2q}u\|_{L^{2}}^2e^{2\|\Pi^N|u|^{2q}u\|_{L^2}^2}+\frac{1}{2}\sum_{p\in\N}\frac{\langle |\nabla u|^2,|u|^{2p} \rangle}{p!}.
\end{align*}
Gathering all this, we arrive at the desired lower bound:
\begin{align*}
\mathcal{E}(u) &\geq\frac{1}{2}\left[\|u\|_{H^\beta}^2+3e^{2\|\Pi^N|u|^{2q}u\|_{L^2}^2}\langle|\nabla u|^2,|u|^{2q}\rangle+e^{2\|\Pi^N|u|^{2q}u\|_{L^2}^2}\|\Pi^N|u|^{2q}u\|_{L^2}^2+\sum_{p\in\N}\frac{\langle |\nabla u|^2,|u|^{2p}\rangle}{p!}\right].
\end{align*}
Let us make the following useful observations:
\begin{align*}
\sum_{p\in\N}\frac{\|u^{p+1}\|_{L^2}^2}{p!}&\leq \mathcal{M}(u);\\
\sum_{p\in\N}\frac{\|u^{p+1}\|_{\dot{H}^1}^2}{p!}=\sum_{p\in\N}\frac{\langle |\nabla u|^2,|u|^{2p}\rangle}{p!} &\leq 2\mathcal{E}(u).
\end{align*}
From these inequalities, we readily have the following
 \begin{align}
 \sum_{p\in\N}\frac{\|u^{p+1}\|_{{H}^1}^2}{p!} &\leq 2\mathcal{E}(u)+\mathcal{M}(u).\label{u^p:H^1}
 \end{align}
 Let us also notice the inequality
 \begin{align}
 \mathcal{M}(u)\lleq \mathcal{E}(u).\label{Mcal:leq:Ecal}
 \end{align}

We can perform similar computations and obtain the results of Section \ref{Sect:Galerk} where \eqref{SNLS} is replaced by \eqref{Eq:Newdiss}. Namely, after the inviscid limit $\al\to 0$, we obtain  a sequence of invariant measures $(\mu^N)$ associated to \eqref{NLS-projN} and satisfying
\begin{align}
\int_{L^2}\mathcal{M}(u)\mu^N(du) &=\frac{A_0^N}{2}, \notag\\
\int_{L^2}\mathcal{E}(u)\mu^N(du) &\leq C,\label{Est:E:Super}
\end{align}
where $C>0$ does not depend on $N$.
Then from \eqref{u^p:H^1}, we obtain
\begin{align}
 \int_{L^2}\sum_{p\in\N}\frac{\|u^{p+1}\|_{{H}^1}^2}{p!}\mu^N(du) &\leq C_1,\label{Exp:u^p:H^1}
 \end{align}
where $C_1$ does not depend on $N.$

Furthermore, we have the estimate
\begin{align}
\int_{L^2}E(u)\mathcal{E}(u)\mu^N(du)\leq C_2,\label{EEcal:Est}
\end{align}
where $C_2$ does not depend on $N.$ To see \eqref{EEcal:Est}, let us apply the It\^o formula to $E^2(u)$, where $u$ is the stationary solution to $\eqref{Eq:Newdiss}$: 
\begin{align*}
dE^2 &=2EE'(u)+\frac{\al}{2}\sum_{m=0}^N(2|a_m|^2E(u)E''(u,e_m,e_m)+2|E'(u,e_m)|^2)\\
&=-2E(u)\mathcal{E}(u)+\frac{\al}{2}\sum_{n=0}^N(2|a_m|^2E(u)E''(u,e_m,e_m)+2|E'(u,e_m)|^2)+2\sqrt{\al}\sum_{m=0}^N E(u)E'(u,a_me_m) \, d\beta_m.
\end{align*}
Now, we use the fact that $\|e_m\|_{L^p}$ is bounded for $p<3$ in \eqref{eq e_n bdd} and obtain
\begin{align*}
EE''(u,e_m,e_m)&= \langle-\Delta e_m,e_m\rangle E+\langle|u|^{2q},e_m^2\rangle E\leq (z_m^2 + \|u\|_{L^{6q+}}^{4q}+C)E,
\end{align*}
We then have
\begin{align*}
\sum_{m=0}^N|a_m|^2E(u)E''(u,e_m,e_m)&\leq (CA_0^N+A_1^N)E(u) +E(u)\mathcal{E}(u).
\end{align*}
Also
\begin{align*}
|E'(u,e_m)|^2=|\langle-\Delta u+|u|^{2q}u,e_m\rangle| &\leq 2|\langle-\Delta u,e_m\rangle|^2+2|\langle|u|^{2q}u,e_m\rangle|^2 \leq 2\|u\|_{\dot{H}^1}^2z_m^2+2\|\Pi^N |u|^{2q}u\|_{L^2}^2,
\end{align*}
therefore
\begin{align*}
\sum_{m=0}^N|a_m|^2|E'(u,e_m)|^2\leq A_1^NE(u)+A_0^N\|\Pi^N |u|^{2q}u\|_{L^2}^2\leq A_1^NE(u)+A_0^N\mathcal{E}(u).
\end{align*}
Overall, recalling that $A_0^N$ and $A_1^N$ are bounded in $N$, we arrive at
\begin{align*}
dE^2(u) +E(u)\mathcal{E}(u)\leq C(E(u)+\mathcal{E}(u))+2\sqrt{\al}\sum_{m=0}^N E(u)E'(u,a_me_m) \, d\beta_m,
\end{align*}
where $C$ does not depend on $(N,\al)$. We now integrate in $t$ and take the expectation with the use of invariance of the measure $\mu_{\alpha}^N$, we then obtain
\begin{align*}
\int_{L^2}E(u)\mathcal{E}(u)\mu_\al^N(du)\leq C_2,
\end{align*}
where $\mu^N_\al$ is the underlying stationary measure for \eqref{Eq:Newdiss} and $C_1$ does not depend on $(N,\al)$. Now passing to the limit $\al\to 0$, we arrive at \eqref{EEcal:Est}.
\subsection{Existence of global solutions}
For an arbitrary positive integer $k$, set the spaces
\begin{align*}
X_k &=L^2([0,k],H^\beta)\cap (H^1([0,k],H^{\beta-2})+H^{1}([0,k],L^2))=X_k^1+X_k^2;\\
Y_k &= L^2([0,k],H^{\beta-})\cap C([0,k],H^{(\beta-2)-}). 
\end{align*}
By classical Sobolev embedding we see that $H^1([0,k],H^{\beta-2})\subset C^{\frac{1}{2}}([0,k],H^{\beta-2})$ continuously, hence $X_k^1\subset C^{\frac{1}{2}}([0,k],H^{\beta-2}) $ continuously. We see also that $X_k^2\subset C^{\frac{1}{2}}([0,k],L^{2})$ continuously. Overall we have that $X_k^1$ and $X_k^2$ are compactly embedded in $C([0,k],H^{(\beta-2)-})$. Now  since $X_k^1$ and $X_k^2$ are compactly embedded in $L^2([0,k],H^{\beta-})$, we obtain that $X_k$ is compactly embedded in $Y_k$.

Denote by $\nu^N_k$ the distributions of the processes $(u^N(t))_{t\in [0,k)}$, and these are seen as random variables valued in $C([0,k],H^{(\beta-2)-})$. We see, using the invariance, the relation between $\mu^N$ and $\nu^N_k$: 
\begin{align}
\mu^N ={{\nu^N_k}_|}_{t=t_0}\quad \text{for any $t_0\in [0,k]$}.\label{Ms:Rest:ppt:N}
\end{align}

\begin{prop}\label{Prop:Tight}
We have that
\begin{align*}
\int_{X_k}\|u\|_{X_k}^2\nu_k^N(du)\leq C,
\end{align*}
where $C$ does not depend on $N$.
\end{prop}
\begin{rmq}
As a consequence of Proposition \ref{Prop:Tight}, we have, thanks to the Prokhorov's theorem, that there is a subsequence of $(\nu_k^N)$ that we denote again by $(\nu_k^N)$ converging to a mesure $\nu_k\in\mathfrak{p}(X_k)$.
\end{rmq}
\begin{proof}[Proof of Proposition \ref{Prop:Tight}]
From \eqref{Est:E:Super}, we derive that
\begin{align*}
\E\|u^N\|_{L^2([0,k],H^\beta)}^2=\E\int_0^k\|u^N(t)\|^2_{H^\beta} \, dt\leq Ck.
\end{align*}
Now let us write the equation in the integral formulation:
\begin{align}
u^N_k(t) = u^{N}_k(0)+\i\int_0^t\Delta u^N_k(s) \, ds-\i\int_0^t\Pi^N|u^N_k(s)|^{2q}u^N_k(s) \, ds.\label{Eq:N:int}
\end{align}
 We remark the following simple inequalities
\begin{align*}
& \E\|\int_0^t\Delta u^N_k \, ds\|_{H^1([0,k],H^{\beta-2})}^2\leq \E\int_0^k\|u^N_k\|_{H^{\beta}}^2 \, dt, \\
& \E\|\int_0^t\Pi^N|u^N_k|^{2q} u^N_k \, ds\|_{H^{1}([0,k],L^2)}^2\leq \E\int_0^k\|\Pi^N|u^N_k|^{2q} u^N_k\|_{L^2}^2 \, dt\leq C\E\int_0^k\mathcal{E}(u^N_k(t)) \, dt.
\end{align*}
We then invoke \eqref{Est:E:Super} to obtain 
\begin{align*}
\E\|u^N_k\|_{X_k}^2\leq Ck.
\end{align*}
We finish the proof of Proposition \ref{Prop:Tight}.
\end{proof}
Let us give the rest of the existence proof. Thanks to the Skorokhod representation theorem (see Theorem \ref{Thm:Skorokhod}), there are random variables $\tilde{u}_k^N$ and $\tilde{u}_k$, all defined on a same probability space that we denote again by $(\Omega,\P)$, such that
\begin{enumerate}
\item $\tilde{u}_k^N\to \tilde{u_k}$ weakly on $X_k$,
\item $\tilde{u}_k^N$ and $\tilde{u}_k$ are distributed by $\nu_k^N$ and $\nu_k$ respectively. 
\end{enumerate} 
Since $\tilde{u}_k^N$ are distributed as $u_k^N$ in $C([0,k],H^{(\beta-2)-})$, we have that $\tilde{u}_k^N$ solves the equation \eqref{Eq:N:int} almost surely.
Passing to the limit $N\to\infty$ in \eqref{Eq:N:int} and using a diagonal argument we find a limiting process $u\in C(\R_+,H^{(\beta-2)-})$ (which is embedded in $L^1(\R_+,H^{-2})$, since $\beta>0$) satisfying $\P-$almost surely the equation
\begin{align}
u(t) = u(0)+\i\int_0^t\Delta u(s)ds-\i\int_0^t|u(s)|^{2q}u(s) \, ds \label{Eq:int}.
\end{align}
The convergence above  is viewed in $L^1(\R_+,H^{-2})$, where we justify the almost sure convergence of $|u^N_k|^{2q}u^N_k$ to $|u|^{2q}u$ as follows:
\begin{enumerate}
\item Thanks to the construction above $\tilde{u}^N_k$ converges to $u$ in  $L^2_tH^1_x$, up to a subsequence denoted by $\tilde{u}^N_k$ as well;
\item Using the lower semicontinuity of $\mathcal{M}(\cdot)$, we see that $\E\mathcal{M}(u(t))<\infty$. In particular $\|u\|_{L^p}^p$ is finite for any $p>1$ almost surely. 
\item Combining the Prokhorov and Skorokhod theorems with the estimate \eqref{Exp:u^p:H^1}, we see that $|\tilde{u}^N_k|^{p+1}$ is compact in $L^2_{t,x}$ for any $p\geq 1$ almost surely. In particular $\|\tilde{u}^N_k\|_{L^{2p+2}_{t,x}}^{2p+2}$ is bounded in $N$ for any $p,$ almost surely. 
\end{enumerate}
Gathering these observations, we write for $\P-$almost all $\omega\in \Omega$
\begin{align*}
\||\tilde{u}^N_k|^{2q}\tilde{u}^N_k-|u|^{2q}u\|_{L^1_{t,x}}\leq C_q\||\tilde{u}^N_k-u|(|\tilde{u}^N_k|^{2q}+|u|^{2q})\|_{L^1_{t,x}}&\leq C_q\|\tilde{u}^N_k-u\|_{L^2_{t,x}}(\|\tilde{u}^N_k\|^{2q}_{L^{4q}_{t,x}}+\|u\|^{2q}_{L^{4q}_{t,x}})\\
&\leq C_{q,\omega}\|\tilde{u}^N_k-u\|_{L^2_{t,x}}
\end{align*}
We can now pass to the limit $N\to\infty$ and use a diagonal procedure to obtain the desired almost sure convergence.

\subsection{Conservation laws and Estimates}
For the conservation of $M(u)$ and $E(u)$, let us remark that from Prokhorov and Skorokhod theorems, the almost sure convergence of $\tilde{u}^N_k$ towards $u$ holds in any $H^s$ for $s<\beta$ (so in $H^1$) and also $|u^N|^{q+1}$ is compact in $L^{2}$ because of the boundedness of $\E\|u^{q+1}\|_{{H}^1}^2$ (see \eqref{Exp:u^p:H^1}) and the Rellich-Kondrachov theorem (Theorem \ref{thm RK}). All this ensures the pointwise convergence of the energy, and a less difficult procedure provides the convergence of the mass. Therefore for any integer $k$, any $t \in [0,k]$
\begin{align*}
&\lim_{N\to\infty}E(\tilde{u}^N_k(t))=E(u(t)),
\quad\lim_{N\to\infty}M(\tilde{u}^N_k(t))=M(u(t)).
\end{align*} 
Now we recall that $E(\tilde{u}^N_k(t))$ and $M(\tilde{u}^N_k(t))$ are constants in time as conservation laws for finite-dimensional projections of \eqref{NLS}, this gives the claimed conservation. From the conservation we obtain that $u\in L^\infty_t H^1_x\cap L^{2q+2}_x.$

Let us now present the following statistical estimates on $u(t)$.
\begin{thm}\label{thm StatEst}
The process $u(t)$ satisfies the following
\begin{align}
\E \mathcal{M}(u(t)) &=\frac{A_0}{2};\label{Id}\\
\E \|u(t)\|_{H^\beta}^2 &\leq C. \notag
\end{align}
\end{thm}
\begin{rmq}
The identity \eqref{Id} is crucial to ensure that large data are contained in the support of $\mu$ as explained in Remark \ref{Rmk:Intro}.
\end{rmq}
\begin{proof}[Proof of Theorem \ref{thm StatEst}]
First, the inequalities
\begin{align*}
\E \mathcal{M}(u(t)) &\leq\frac{A_0}{2},\\
\E \|u(t)\|_{H^\beta}^2 &\leq C ,
\end{align*}
can be easily shown using the proof of \eqref{Est:M} and \eqref{Est:E}. The most difficult is the reverse inequality
\begin{align*}
\frac{A_0}{2}\leq \E \mathcal{M}(u(t)).
\end{align*}
To show this reverse inequality, let us first make a few observations:
\begin{enumerate}
\item By the the boundedness of $\E\mathcal{E}(u)$, we have weak compactness of $(u^N)_N$ in $H^{\beta-1}$ and of $|u^N|^{p+1}$ in $L^{6-}$ for every $p\in \N$ (because of \eqref{Exp:u^p:H^1}).
\item (Estimation of the quadratic part) Let $M$ be a positive integer, using \eqref{Mcal:leq:Ecal} and \eqref{EEcal:Est} we have
\begin{align*}
\E\|u^N\|_{H^{\beta-1}}^2 & \leq \E\|u^N_{\leq M}\|_{H^{\beta-1}}^2+M^{-1}\E\|u^N_{\leq M}\|_{H^{\beta}}^2\leq \E\|u^N_{\leq M}\|_{H^{\beta-1}}^2+M^{-1}\E\mathcal{E}(u^N)\leq \E\|u^N_{\leq M}\|_{H^{\beta-1}}^2+M^{-1}C\\
&
\leq \E\|u^N_{\leq M}\|_{H^{\beta-1}}^21_{\|u\|_{L^2}\leq R}+\E\|u^N_{\leq M}\|_{H^{\beta-1}}^21_{\|u\|_{L^2}\geq R}+M^{-1}C\\
&\leq \E\|u^N_{\leq M}\|_{H^{\beta-1}}^21_{\|u\|_{L^2}\leq R}+\E\mathcal{E}(u^N)\frac{E(u^N)}{E(u^N)}1_{\|u\|_{L^2}\geq R}+M^{-1}C
\end{align*}
then,
\begin{align*}
\E\|u^N\|_{H^{\beta-1}}^2 &\leq \E\|u^N_{\leq M}\|_{H^{\beta-1}}^21_{\|u\|_{L^2}\leq R}+CR^{-2}+M^{-1}C.
\end{align*}
\item (Estimation of the exponential part) Using the same $M$ as above we establish the following
\begin{align*}
\E e^{2\|\Pi^N|u^N|^{2q}u^N\|_{L^2}^2}\|u^N\|_{L^{2q+2}}^{2q+2} &= \E e^{2\|\Pi^N|u^N|^{2q}u^N\|_{L^2}^2}\||u^N|^{q+1}\|_{L^{2}}^{2}\\
&=\E e^{2\|\Pi^N|u^N|^{2q}u^N\|_{L^2}^2}\|\Pi^M|u^N|^{q+1}\|_{L^{2}}^{2}+\E e^{2\|\Pi^N|u^N|^{2q}u^N\|_{L^2}^2}\|\Pi^{\geq M}|u^N|^{q+1}\|_{L^{2}}^{2}\\
&\leq \E e^{2\|\Pi^N|u^N|^{2q}u^N\|_{L^2}^2}\|\Pi^M|u^N|^{q+1}\|_{L^{2}}^{2}+M^{-1}\E e^{2\|\Pi^N|u^N|^{2q}u^N\|_{L^2}^2}\langle |\nabla u^N|^2,|u^N|^{2q}\rangle.
\end{align*}
Therefore, we can use the estimate \eqref{Est:E:Super} to bound the second term. We obtain
\begin{align*}
& \quad \E e^{2\|\Pi^N|u^N|^{2q}u^N\|_{L^2}^2}\|u^N\|_{L^{2q+2}}^{2q+2}  \leq \E e^{2\|\Pi^N|u^N|^{2q}u^N\|_{L^2}^2}\|\Pi^M|u^N|^{q+1}\|_{L^{2}}^{2}+M^{-1}C\\
&\leq \E e^{2\|\Pi^N|u^N|^{2q}u^N\|_{L^2}^2}\|\Pi^M|u^N|^{q+1}\|_{L^{2}}^{2}1_{\|\Pi^N|u^N|^{2q}u^N\|_{L^2}\leq R}+\E e^{2\|\Pi^N|u^N|^{2q}u^N\|_{L^2}^2}\|\Pi^M|u^N|^{q+1}\|_{L^{2}}^{2}1_{\|\Pi^N|u^N|^{2q}u^N\|_{L^2}\geq R}+M^{-1}C\\
& =: (i)+(ii)+M^{-1}C.
\end{align*}
Now, using the estimates \eqref{Mcal:leq:Ecal} and \eqref{EEcal:Est}, we achieve the following computations
\begin{align*}
(i) &\leq \E e^{2\|\Pi^N|u^N|^{2q}u^N\|_{L^2}^2}\|\Pi^M|u^N|^{q+1}\|_{L^{2}}^{2}1_{\|\Pi^N|u^N|^{2q}u^N\|_{L^2}\leq R}1_{\|u^N\|_{L^{2q+2}}\leq R}\\
& \quad +\E e^{2\|\Pi^N|u^N|^{2q}u^N\|_{L^2}^2}\|\Pi^M|u^N|^{q+1}\|_{L^{2}}^{2}1_{\|\Pi^N|u^N|^{2q}u^N\|_{L^2}\leq R}1_{\|u^N\|_{L^{2}}\geq R}=:(i_1)+(i_2)\\
(i_2)&\leq \E \mathcal{M}(u^N)1_{\|u^N\|_{L^{2}}\geq R}\leq C\E \mathcal{E}(u^N)1_{\|u^N\|_{L^{2}}\geq R}\frac{E(u^N)}{E(u^N)}\leq C_1R^{-2}.
\end{align*}
We use again the estimate \eqref{EEcal:Est} and the observation 
\begin{align*}
\|\Pi^M|u^N|^{q+1}\|_{L^{2}}^{2}\lleq \|u^N\|_{L^{2q+2}}^{2q+2}\lleq E(u^N)
\end{align*}
to obtain
\begin{align*}
(ii) &= \E e^{2\|\Pi^N|u^N|^{2q}u^N\|_{L^2}^2}\|\Pi^M|u^N|^{q+1}\|_{L^{2}}^{2}\frac{\|\Pi^N|u^N|^{2q}u^N\|_{L^2}^2}{\|\Pi^N|u^N|^{2q}u^N\|_{L^2}^2}1_{\|\Pi^N|u^N|^{2q}u^N\|_{L^2}\geq R}\\
&\leq\frac{1}{R^{2}} \E e^{2\|\Pi^N|u^N|^{2q}u^N\|_{L^2}^2}\|\Pi^N|u^N|^{2q}u^N\|_{L^2}^2 E(u) \leq\frac{1}{R^{2}} \E[\mathcal{E}(u) E(u)]\leq C_2R^{-2}.
\end{align*}
Gathering the estimates on $(i)$ and $(ii)$ we arrive at
\begin{align*}
\E &e^{2\|\Pi^N|u^N|^{2q}u^N\|_{L^2}^2}\|u^N\|_{L^{2q+2}}^{2q+2}  \leq \E e^{2\|\Pi^N|u^N|^{2q}u^N\|_{L^2}^2}\|\Pi^M|u^N|^{q+1}\|_{L^{2}}^{2}1_{\|\Pi^N|u^N|^{2q}u^N\|_{L^2}\leq R}1_{\|u^N\|_{L^{2}}\leq R}+C_3R^{-2}+M^{-1}C. 
\end{align*}

\item (Estimation of the series part) Finally, using again \eqref{EEcal:Est} (and noticing that $\|u\|_{L^2}^2\lleq E(u)$) we have the following
\begin{align*}
\E\sum_{p\in\N}\frac{\|u^N\|_{L^{2p+2}}^{2p+2}}{p!} &=\E\sum_{p\in\N}\frac{\||u^N|^{p+1}\|_{L^{2}}^{2}}{p!} =\E\sum_{p\in\N}\frac{\|\Pi^M|u^N|^{p+1}\|_{L^{2}}^{2}}{p!}+\E\sum_{p\in\N}\frac{\|\Pi^{\geq M}|u^N|^{p+1}\|_{L^{2}}^{2}}{p!}=:(iii)+(iv)\\
(iv)&\leq M^{-1}\E\sum_{p\in\N}\frac{\langle |\nabla u^N|^2,|u^N|^{2p}\rangle}{p!}\leq M^{-1}C\\
(iii)&\leq \E1_{\|u^N\|_{L^2}\leq R}\left(\sum_{p\in\N}\frac{\|\Pi^M|u^N|^{p+1}\|_{L^{2}}^{2}}{p!}\right)+\E1_{\|u^N\|_{L^2}\geq R}\frac{\|u^N\|_{L^2}^2}{\|u^N\|_{L^2}^2}\left(\sum_{p\in\N}\frac{\|\Pi^M|u^N|^{p+1}\|_{L^{2}}^{2}}{p!}\right)\\
&\leq \E1_{\|u^N\|_{L^2}\leq R}\left(\sum_{p\in\N}\frac{\|\Pi^M|u^N|^{p+1}\|_{L^{2}}^{2}}{p!}\right)+R^{-2}\E1_{\|u^N\|_{L^2}\geq R}E(u)\left(\sum_{p\in\N}\frac{\|\Pi^M|u^N|^{p+1}\|_{L^{2}}^{2}}{p!}\right) \\
&\leq \E1_{\|u^N\|_{L^2}\leq R}\left(\sum_{p\in\N}\frac{\|\Pi^M|u^N|^{p+1}\|_{L^{2}}^{2}}{p!}\right)+R^{-2}C_0.
\end{align*}
\end{enumerate}
Overall, we obtain a  constant $c$, independent of $(N,M,R)$, such that 
\begin{align*}
\frac{A_0^N}{2}\leq\E\mathcal{M}(u^N) &\leq \E\|u^N_{\leq M}\|_{H^{\beta-1}}^21_{\|u\|_{L^2}\leq R}+2
\E e^{2\|\Pi^N|u^N|^{2q}u^N\|_{L^2}^2}\|\Pi^M|u^N|^{q+1}\|_{L^{2}}^{2}1_{\|\Pi^N|u^N|^{2q}u^N\|_{L^2}\leq R}1_{\|u^N\|_{L^{2}}\leq R}\\
& \quad +\E1_{\|u^N\|_{L^2}\leq R}\left(\sum_{p\in\N}\frac{\|\Pi^M|u^N|^{p+1}\|_{L^{2}}^{2}}{p!}\right)+c(R^{-2}+M^{-1}).
\end{align*}
We can pass to the limit $N\to\infty$, then $M\to\infty$ and $R\to\infty$ to obtain the desired reverse inequality.

Now the proof of Theorem \ref{thm StatEst} is complete.
\end{proof}

\subsection{Uniqueness, continuity, invariance and regularity of the solutions}
\subsubsection{Uniqueness}
Let us write the Duhamel formulation (that is weaker than the formulation \eqref{Eq:int}), we have
\begin{align*}
u(t)=e^{\i t\Delta }u_0-\i\int_0^te^{\i(t-s)\Delta}|u(s)|^{2q}u(s) \, ds.
\end{align*}
For $u^1,\ u^2$ two solutions starting from $u_0$ in the support of $\mu$, set $w=u^1-u^2$. We have
\begin{align*}
\int_{1\geq|x|\geq \epsilon}|w|^2 \, dx &\leq2 \int_{1\geq|x|\geq \epsilon}|e^{\i t\Delta}w_0|^2 \, dx+C_qt\int_0^t\int_{1\geq|x|\geq \epsilon}|(|u^1|^{2q}+|u^{2}|^{2q})w|^2 \, dxds.
\end{align*}
From the radial Sobolev embedding (Lemma \ref{lem Radial Sobolev}), we have
\begin{align*}
\int_{1\geq|x|\geq \epsilon}|w|^2 \, dx &\leq2 \int_{1\geq|x|\geq \epsilon}|e^{\i t\Delta}w_0|^2 \, dx+C_qt\epsilon^{-4q}\int_0^t(\|u^1\|_{H^1}^{4q}+\|u^{2}\|_{H^1}^{4q})\int_{1\geq|x|\geq \epsilon}|w|^2 \, dxds.
\end{align*}
By the Gronwall lemma,
\begin{align}
\quad \int_{1\geq|x|\geq \epsilon}|w|^2 \, dx & \leq 2e^{C_qt\epsilon^{-4q}\int_0^t(\|u^1\|_{H^1}^{4q}+\|u^{2}\|_{H^1}^{4q}) \,ds}\int_{1\geq|x|\geq \epsilon}|e^{\i t\Delta}w_0|^2 \,dx \label{Gronw}\\
&\leq 2e^{C_qt\epsilon^{-4q}\int_0^t(\|u^1\|_{H^1}^{4q}+\|u^{2}\|_{H^1}^{4q}) \, ds}\int_{\Theta}|e^{\i t\Delta}w_0|^2 \, dx \leq 2e^{C_qt\epsilon^{-4q}\int_0^t(\|u^1\|_{H^1}^{4q}+\|u^{2}\|_{H^1}^{4q}) \, ds}\int_{\Theta}|w_0|^2 \, dx =0.\nonumber
\end{align}
We let $\epsilon\to 0$ by using the monotone convergence theorem and obtain that $\|w\|_{L^2}=0.$ Hence the uniqueness.
\begin{rmq}
The uniqueness argument is deterministic. The  condition is $u\in L^{4q}_tH^1_x$, this is, in particular, satisfied using the energy conservation proved in the previous subsection.
\end{rmq}

\subsubsection{Continuity}
Let $(u_0^k)$ be a sequence converging to $u_0$ in $H^1\cap L^{2q+2}$, set $(u^k)$ and $u$ the solutions associated to $(u_0^k)$ and $u_0$ respectively. Set $w_0^k=u_0^k-u_0$ and $w^k=u^k-u$, and define the sequence 
\begin{align*}
\gamma_k=
\begin{cases}
(-\ln\|w_0^k\|_{L^2})^\frac{-1}{8q}, & \text{if $\|w_0^k\|_{L^2}>0$}\\
0, & \text{if $\|w_0^k\|_{L^2}=0$}.
\end{cases}
\end{align*}
  We claim the following:
\begin{align}
\int_{1\geq |x|\geq\gamma_k}|w^k(t,x)|^2 \, dx\lleq e^{C(q,R)t^2}\|w_0^k\|_{L^2}^{\frac{3}{2}},\label{Claim0}
\end{align}
where $R$ is the size of a ball in $H^1\cap L^{2q+2}$ centered at $0$, containing $(u_0^k)$ and $u_0$. To prove \eqref{Claim0}, we first observe that it is true when $w_0^k=0$ because of the uniqueness; then assuming $w_0^k\neq 0$ and taking $\epsilon=\gamma_k\neq 0$ in \eqref{Gronw}, we write
\begin{align*}
\int_{1\geq|x|\geq \gamma_k}|w^k|^2 \, dx & \leq 2e^{C_q\gamma^{-4q}t^2(E(u^k_0)^{2q}+E(u_0)^{2q})}\|w_0^k\|_{L^2}^2 \,dx\\
&\leq 2e^{C(q,R)t^4}e^{\gamma^{-8q}}\|w_0^k\|_{L^2}^2=2e^{C(q,R)t^4}\frac{\|w_0^k\|_{L^2}^2}{\|w_0^k\|_{L^2}}\lleq e^{C(q,R)t^4}\|w_0^k\|_{L^2}.
\end{align*}
The proof of \eqref{Claim0} is complete.

We now have
\begin{align*}
\int_{\Theta}|w^k|^2 \, dx 
&\lleq e^{C(q,R)t^4}\|w_0^k\|_{L^2}+\int_{0\leq|x|\leq \gamma_k}|w^k|^2 \, dx.
\end{align*}
To obtain an $L^2$ continuity, it remains to show that $\int_{0\leq|x|\leq \gamma_k}|w^k|^2 \, dx=o(1)$ as $k\to\infty$. To see this, we have by H\"older inequality
\begin{align*}
\int_{0\leq|x|\leq \gamma_k}|w^k|^2 \, dx\leq \|1_{\{0\leq|x|\leq \gamma_k\}}\|_{L^4}\|w^k\|_{L^4}^2\lleq \|1_{\{0\leq|x|\leq \gamma_k\}}\|_{L^4}(E(u_0^k)+E(u_0)).
\end{align*}
And by using the Lebesgue dominated convergence theorem, we see that $\|1_{\{0\leq|x|\leq \gamma_k\}}\|_{L^4}\to 0$ as $k\to \infty$. Overall we obtain
\begin{align*}
\lim_{\|w_0^k\|_{H^1}+\|w_0^k\|_{L^{2q+2}}\to 0}\sup_{t\in [-T,T]}\int_{\Theta}|w^k(t,x)|^2 \, dx=0,
\end{align*}
for any $T>0.$ Now, by interpolation, we have for any $s<1\leq p<2q+2$
\begin{align*}
\lim_{\|w_0^k\|_{H^1}+\|w_0^k\|_{L^{2q+2}}\to 0}\sup_{t\in [-T,T]}(\|w^k(t)\|_{H^s}+\|w^k(t)\|_{L^{p}})=0 .
\end{align*}
This was the continuity claim.

\subsubsection{Invariance and regularity}
Now let us turn to the regularity of $u$. In the previous subsection we showed the conservation of energy. This implied that $u\in L^\infty_tH^1_x\cap L^{2q+2}_x$. Since the measure is concentrated on $H^\beta,$ with $\beta>1$, we expect to improve this result. Using the Duhamel formulation above, we have
\begin{align*}
\|u\|_{H^{\beta-}}\leq \|u_0\|_{H^{\beta-}}+\int_0^t\|u^{2q}u\|_{H^{\beta-}} \, ds.
\end{align*}
The Kato–Ponce inequality (Theorem \ref{thm KP}) implies that
\begin{align*}
\|u\|_{H^{\beta-}}\leq \|u_0\|_{H^{\beta-}}+\int_0^t\|u\|_{L^{2qr}}^{2q}\|u\|_{W^{\beta-,2+}} \, ds,
\end{align*}
where $\frac{1}{2}=\frac{1}{r}+\frac{1}{2+}.$ Now, we have $\|u\|_{W^{\beta-,2+}}\leq \|u\|_{H^{\beta}}$. We arrive at the following
\begin{align*}
\sup_{t\in[0,T]}\|u(t)\|_{H^{\beta-}}\leq \|u_0\|_{H^{\beta-}}+\int_0^T\mathcal{E}(u) \, ds,
\end{align*}
which is finite for every $T,$ $\P-$almost surely. We obtain $u\in C(\R,H^{\beta-})$ using the time reversibility of  \eqref{NLS} equation. This combined with the controls on $\mathcal{M}(u)$ and $\mathcal{E}(u)$ established in the previous subsection, we obtain that $u\in C(\R,H^{\beta-})\cap L^\infty(\R,L^{2q+2})\cap L^2_{loc}(\R,H^{\beta})\cap L^p_{loc}(\R,L^r)$ for any $p,r\in [1,+\infty)$.
\begin{rmq}
The role of $e^{|u|^2}u$ in the dissipation \eqref{Newdiss} can be seen in the estimate: for every $r>1$,
\begin{align*}
\|u\|_{L^{2qr}}^{2q}\|u\|_{H^{\beta}}\lleq \mathcal{E}(u) .
\end{align*}
\end{rmq}

For the invariance of the law of $u(t)$, let us denote by $\nu$ the distribution of the process $u=(u(t))_{t\in\R}$. For the subsequence in the $N$-parameter that produced $\nu$, we can extract a subsequence, using the Prokhorov theorem, that produce a measure $\mu$ as a weak limit point of $(\mu^N)$. Passing to the limit along this subsequence in the relation \eqref{Ms:Rest:ppt:N}, we see that $\mu ={\nu_|}_{t=t_0}$ for any $t_0\in\R$. This establishes that $\mu$ is an invariant law for $u$.

\appendix

\section{Some useful standard results}

In this section $(E,\|.\|)$ is a Banach space, $C_b(E)$ denotes the space of bounded continuous functions $f:E\to\R.$ We present some useful standard results from measure theory. 

\subsection{Convergence of random variables}
\begin{Def}
Let $(\mu_n)_n$ be a sequence of Borel probability measures and $\mu$ a Borel probability measure on $E$. We say that $(\mu_n)_n$ converges weakly to $\mu$ if for all $f\in C_b(E)$
\begin{align*}
\lim_{n\to\infty}\int_Ef(x)\mu_n(dx)=\int_Ef(x)\mu(dx).
\end{align*}
We write $\mu_n\rightharpoonup\mu.$
\end{Def}

\begin{Def}
A family $\Lambda$ of Borel probability measures on $E$ is said to be tight if for all $\epsilon>0$ there is a compact set $K_\epsilon\subset E$ such that for all $\mu\in \Lambda$
\begin{align*}
\mu(K_\epsilon)\geq 1-\epsilon.
\end{align*} 
\end{Def}

\begin{thm}[Portmanteau theorem, see Theorem 11.1.1 in \cite{dudley}]\label{Thm:Portm}
Let $(\mu_n)_n$ be a sequence of probability measures and $\mu$ a probability measure on $E$, the following are equivalent:
\begin{enumerate}
\item $\mu_n\rightharpoonup\mu$,
\item for all open sets $U$, $\liminf_{n\to\infty}\mu_n(U)\geq\mu(U)$,
\item for all closed sets $F$, $\limsup_{n\to\infty}\mu_n(F)\leq\mu(F)$.
\end{enumerate}
\end{thm}

\begin{thm}[Prokhorov theorem, see Theorem 2.3 in \cite{daz}]\label{Thm:Prokh}
A set $\Lambda$ of Borel probability measures is relatively compact in $E$ if and only if it is tight.
\end{thm}

\begin{thm}[Skorokhod representation theorem, see Theorem 2.4 in \cite{daz}]\label{Thm:Skorokhod}
For any sequence $(\mu_n)_n$ of Borel probability measures on $E$ converging weakly to $\mu$, there is a probability space $(\Omega_0,\mathcal{F}^0,\P_0)$, and random variables $X,\ X_1,\cdots$ on $(\Omega_0,\mathcal{F}^0,\P_0)$ such that
\begin{enumerate}
\item $\mathcal{L}(X_n)=\mu_n$ and $\mathcal{L}(X)=\mu$,
\item $\lim_{n\to\infty}X_n=X$, $\P_0-$almost surely.
\end{enumerate}
Here $\mathcal{L}(Y)$ stands for the law of the random variable $Y.$
\end{thm}

\subsection{Stochastic processes}

Recall that if $(X_t)_t$ is an $E-$valued martingale, then since $\|.\|$ is a convex function, we have that $(\|X_t\|)_t$ is a submartingale (by Jensen inequality). Here is a version of the Doob maximal inequality (see Theorem 3.8 in \cite{KS98}): 
\begin{thm}[Doob maximal inequality]\label{Thm:Doob}
We have that for $p>1.$
\begin{align*}
\E\left(\sup_{t\in [0,T]}\|X(t)\|\right)^p\leq \left(\frac{p}{p-1}\right)^p\E\|X(T)\|^p.
\end{align*}
\end{thm}

The following can be found in Chapters 3 and 4 in \cite{oks}.
\begin{Def}
An $N$-dimensional It\^o process is a proces $X=(X_t)_{t}$ of the form 
\begin{align}
X(t)=X(0)+\int_0^tu(s)  \, ds+\int_0^tv(s) \, dB(s)\label{Ito:proc}
\end{align}
where $B=(B_i)_{1\leq i\leq N}$ is an $N$-dimension Brownian motion, $u=(u_i)_{1\leq i\leq N}$ is an $N$-dimensional stochastic process and $v=(v_{ij})_{1\leq i,j\leq N}$ is a $N\times N$ matrix that both are adapted with respect to $(B_t)_t$ and satisfy the following
\begin{align*}
\P\left(\int_0^t|u(s)|+|v(s)|^2 \, ds<\infty\quad \forall t>0\right) &=1.
\end{align*}
\end{Def}

A short notation for \eqref{Ito:proc} is given by
\begin{align*}
dX=u \, dt+v \, dB.
\end{align*}
\begin{Thm}[$N$-dimensional It\^o formula]\label{Thm:Ito}
Let $X$ be a $N$-dimensional It\^o process as in \eqref{Ito:proc}. Let $f:\R^N\to\R$ be a $C^2$ function, then $f(X)$ is a $1$-dimensional It\^o process and satisfy
\begin{align*}
df(X)=\nabla f(X)\cdot dX+\frac{1}{2}\sum_{i,j}\partial^2_{x_ix_j}f(X) \, dX_idX_j,
\end{align*}
with the properties $dB_i \, dB_j=\delta_{ij} \, dt$, $dt \, dB_i=0$.
\end{Thm}
Using the properties above, we can remark that in the particular case where $v$ is diagonal (as in this paper), we have
\begin{align*}
\sum_{i,j}\partial^2_{x_ix_j}f(X) \, dX_idX_j=\sum_{i}\partial^2_{x_i}f(X)v_i^2 \, dt.
\end{align*}

\subsection{Estimates from harmonic analysis}
\begin{Thm}[Rellich–Kondrachov theorem]\label{thm RK}
Let $\Omega \subset \R^N$ be an open, bounded Lipschitz domain, and let $1 \leq p < N$. Then for any $1 \leq q < \frac{Np}{N-p}$
\begin{align*}
W^{1, p }  (\Omega) \subset \subset L^q (\Omega) .
\end{align*}
\end{Thm}

\begin{Thm}[Kato–Ponce inequality]\label{thm KP}
For $s > 0$ and $\frac{1}{r} = \frac{1}{p_1} + \frac{1}{p_2} = \frac{1}{q_1} + \frac{1}{q_2} $, $1 < r < \infty$, $1 < p_1, q_1 < \infty $, $1 < p_2, q_2 \leq \infty $
\begin{align*}
\norm{(-\Delta)^s (fg)}_{L^r (\R^N)} \lesssim \norm{(-\Delta)^s f}_{L^{p_1} (\R^N)} \norm{g}_{L^{p_2}(\R^N)} +\norm{(-\Delta)^s g}_{L^{q_1} (\R^N)}  \norm{f}_{L^{q_2}(\R^N)} .
\end{align*}
\end{Thm}

\section{Proofs of lemmas in Section \ref{Sect:Galerk}}\label{Apx B}
Let us present the proofs of Lemma \ref{Lem:KBarg} and Lemma \ref{Lem:MeanMeas} in Section \ref{Sect:Galerk}.

\begin{proof}[Proof of Lemma \ref{Lem:KBarg}]
Note that $P_r f \in  C_b$. We have
\begin{align*}
\langle f, P_r^* \mu\rangle & = \langle P_r f , \mu\rangle = \lim_{t_n \to \infty} \langle P_r f , \frac{1}{t_n} \int_0^{t_n} P_t^* \delta_0 \, dt\rangle  = \lim_{t_n \to \infty} \langle f , P_r^* \frac{1}{t_n} \int_0^{t_n} P_t^* \delta_0 \, dt\rangle \\
& = \lim_{t_n \to \infty} \langle f ,  \frac{1}{t_n} \int_0^{t_n} P_{t+r}^* \delta_0 \, dt\rangle  = \lim_{t_n \to \infty} \langle f ,  \frac{1}{t_n} \int_{r}^{t_n+r} P_{t}^* \delta_0 \, dt\rangle  \\
& = \lim_{t_n \to \infty} \langle f ,  \frac{1}{t_n} \int_0^{t_n} P_{t}^* \delta_0 \, dt\rangle  +  \langle f ,  \frac{1}{t_n} \int_{t_n}^{t_n +r} P_{t}^* \delta_0 \, dt\rangle  - \langle f ,  \frac{1}{t_n} \int_0^{r} P_{t}^* \delta_0 \, dt\rangle  \\
& =\langle f ,\mu\rangle
\end{align*}
where $ \lim_{t_n \to \infty} \langle f ,  \frac{1}{t_n} \int_{t_n}^{t_n +r} P_{t}^* \delta_0 \, dt\rangle  =0$ and  $\lim_{t_n \to \infty} \langle f ,  \frac{1}{t_n} \int_0^{r} P_{t}^* \delta_0 \, dt\rangle =0$, since $\int_{t_n}^{t_n +r} P_{t}^* \delta_0 \, dt  $ and $\int_0^{r} P_{t}^* \delta_0 \, dt $ are  bounded. Therefore, $P_r^* \mu = \mu$.
\end{proof}

\begin{proof}[Proof of Lemma \ref{Lem:MeanMeas}]
Consider $B_R(X)$, by Markov inequality
\begin{align*}
\nu_n(X \setminus B_R) = \nu_n ( \norm{u}_X > R) \leq \frac{C}{R^q},
\end{align*}
then $(\nu_n)$ is tight in $X_0$. Now, the Prokhorov theorem (see Theorem \ref{Thm:Prokh}) implies the existence of $\mu$ in $\mathfrak{P}(X_0)$.

Now, let us show that $\mu (X) =1$. It suffices to show that 
\begin{align*}
\int_{X_0} \norm{u}_{X}^r \mu (du)  \leq  C < \infty
\end{align*}
since, with such property, we obtain
\begin{align*}
\mu(B_R^c) \leq \frac{\int_{X_0} \norm{u}_{X}^r \mu (du)}{R^r} \to 0 \quad \text{ as } R \to \infty.
\end{align*}
Now, let $\chi_R$ be a $C^{\infty}$ function on $[0, \infty)$ such that $\chi_R =1 $ on $[0,R]$ and $\chi_R =0$ on $[R+1 , \infty)$.
\begin{align*}
\int_{X_0} \norm{u}_{X}^r  \chi_R (\norm{u}_{X_0})  \nu_n (du) \leq \int_{X_0} \norm{u}_{X}^r   \nu_n (du) \leq C .
\end{align*}
Since $\norm{u}_{X}^r  \chi_R (\norm{u}_{X_0}) $ is bounded continuous  on $X_0$, by Fatou's lemma, we have 
\begin{align*}
\int_{X_0} \norm{u}_{X}^r   \mu (du) \leq C .
\end{align*}
The proof is finished. 
\end{proof}
\bibliography{NLSfinal}
\bibliographystyle{plain}
\end{document}